\documentclass{siamart}

\usepackage{graphics,amsmath,amsopn,amstext,amsfonts,color,fancybox,hyperref,bm,amssymb}
\usepackage{subfig,placeins} 
\usepackage{graphicx}
\usepackage{verbatim}
\usepackage{algorithm}
\usepackage{algpseudocode}
\usepackage{longtable}
\usepackage{multirow}
\usepackage{diagbox}
\usepackage{tikz}
\usepackage{pgfplots}
\usepackage{multirow}

\newcommand{\epc}{\hspace{1pc}}

\newcommand{\E}{{\rm I\!E}}

\newcommand{\wt}{\widetilde}

\newcommand{\wh}{\widehat}

\newcommand{\gap}{\vspace{0.0in}}

\newtheorem{example}{Example}

\DeclareFontFamily{U}{mathx}{\hyphenchar\font45}
\DeclareFontShape{U}{mathx}{m}{n}{
      <5> <6> <7> <8> <9> <10>
      <10.95> <12> <14.4> <17.28> <20.74> <24.88>
      mathx10
      }{}
\DeclareSymbolFont{mathx}{U}{mathx}{m}{n}
\DeclareFontSubstitution{U}{mathx}{m}{n}
\DeclareMathAccent{\widecheck}{0}{mathx}{"71}
\DeclareMathAccent{\wideparen}{0}{mathx}{"75}

\title{
Composite Difference-Max
Programs for Modern Statistical Estimation Problems}

\author{Ying Cui \thanks{Department of Industrial and
Systems Engineering, University of Southern California, Los Angeles, CA 90089
(\email{yingcui@usc.edu; jongship@usc.edu}).  The work of these two authors was based on research partially
supported by the U.S.\ National Science Foundation grant IIS--1632971.}
\and Jong-Shi Pang \footnotemark[1]
\and Bodhisattva Sen \thanks{Department of Statistics,
Columbia University, New York, NY 10027 (\email{bodhi@stat.columbia.edu}).
The work of this author was based on research partially
supported by the U.S.\ National Science Foundation grant DMS--1712822.}}

\begin{document}
\maketitle

\begin{abstract}
Many modern statistical estimation problems are defined by three major components: a statistical model that postulates the dependence of an output variable on the input features; a loss function measuring the error between the observed output and the model predicted output; and a regularizer that controls the overfitting and/or variable selection in the model. We study the sampling version of this generic statistical estimation problem where the model parameters are estimated by empirical risk minimization, which involves
the minimization of the empirical average of the loss function at the data points weighted by the model regularizer. In our setup we allow all three component functions discussed above to be of the difference-of-convex (dc) type and illustrate them with a host of commonly used examples, including those in continuous piecewise affine regression and in deep learning (where the activation functions are piecewise affine).
We describe a nonmonotone majorization-minimization (MM) algorithm for solving the unified
nonconvex, nondifferentiable optimization problem which is formulated as a specially structured
composite dc program of the pointwise max type, and present convergence results to a directional stationary
solution.  An efficient semismooth Newton method is proposed to solve the dual of the MM subproblems.
Numerical results are presented to demonstrate the effectiveness of the proposed algorithm and the superiority of continuous piecewise affine regression over the standard linear model.
\end{abstract}


\section{Introduction}

Many modern statistical estimation problems are defined by three major components:
a statistical model that postulates the dependence of an output quantity on the input features; a loss function measuring the error between the observed output and model predicted output;
and a regularizer that controls the overfitting and/or variable selection in the model.
The overall estimation problem is to determine certain unknown parameters in the statistical
model. 
In practical computation, samples on the underlying covariates (inputs) and output are available
and a data-based empirical objective combined with a regularizer is formulated as a minimization
problem that constitutes the workhorse of the
estimation process.  This paper addresses the computational solution of the latter optimization
problem which is challenged by its nonconvexity and nondifferentiability.    These features
immediately raise the question about the stationarity, let alone minimizing, properties of
the computed solution by an iterative algorithm.  Our goal is to compute a directional stationary
solution \cite{PangRazaviyaynAlvarado16}, which is the sharpest kind among all stationary solutions.

Traditionally, a statistical model is usually described by a linear combination of the input features,
resulting in a linear regression or classification model.  Nevertheless, in recent years,
research into the use of piecewise affine, convex \cite{HannahDunson11,HannahDunson13,MCISen15}
and nonconvex \cite{HahnBanergjeeSen16} estimation functions has been on the rise.  There are other
piecewise affine/quadratic functions that arise from different applications, such as those using piecewise affine functions to approximate the signum function and multi-layer neural networks in deep learning~\cite[Chapter~4]{YuDeng15}.  These
nontraditional, nondifferentiable estimation functions have provided an important impetus for our work.

Some classical convex loss measures
are based on the absolute-value or squared errors between the model predicted outputs and the
observed outputs, leading to a $\ell_1$ 
or $\ell_2$ loss function for regression.  Other loss functions include a logarithmic loss for logistic regression,
and an exponential loss for boosting.  In the area of support vector machines, a hinge loss is commonly
used for binary classification.
In recent years, a truncated hinge loss function has been proposed to reduce the effect of outliers
\cite{ZhangPhamFuLiu17}.  A distinguished property of the latter loss function is that it is nonconvex and
nondifferentiable.  All these loss measures are composed with the statistical model
involving the parameters to be estimated, leading to an overall composite loss function to be minimized, which is nonconvex
and nondifferentiable, due either to a piecewise statistical model or a loss function that lacks convexity and/or differentiability.
In the area of support vector machines and other applications, a norm function of
the model variables is often used as the model regularizer to avoid overfitting.   In recent years,
sparsity constraints to avoid model overfitting \cite{HastieTibshiraniWainwright15} is an important consideration in
high-dimensional statistical learning.
Starting with the pioneering work of Fan and Li \cite{FanLi01}, various sparsity functions have been proposed as surrogates of the discontinuous
counting step function; it was shown in \cite{AhnPangXin17,LeThiPhamVo15} that all these surrogate sparsity functions can be expressed as
the difference of two convex functions of a particular type.

In this paper, we introduce a composite difference-of-convex-piecewise program of the pointwise max type
as a unification of a host of statistical models, loss functions, and model regularizers.  A highlight
of this formulation is the emphasis on the separate roles of each component function that is
key to the development of a nonmonotone majorization-minimization (MM) algorithm for solving the overall nonconvex,
nondifferentiable optimization problem.
While global optima of nonconvex problems cannot be computed,
stationary solutions of various kinds are computable by iterative algorithms with guaranteed convergence.
What is essential is that focus should be placed on computing sharp stationary solutions which distinguish
themselves as being the ones that must satisfy all other relaxed definitions of stationarity.
With guaranteed subsequential convergence, the developed MM algorithm aims to compute
such a stationary point that is based on the elementary notion of directional derivatives of the objective function.
By using the Kurdyka-{\L}ojaziewicz (KL) theory of semi-analytic
functions \cite{AttouchBolte2009,AttouchBolteSvaiter2013,BoltePauwels2016}, we also show the sequential convergence
of the iterates produced by the algorithm.  In order to handle the nondifferentiable pointwise max terms,
auxiliary variables are employed to express these pointwise maximum functions as constraints.  With the regularization
of the added variables in defining the subproblems, the resulting MM algorithm no longer guarantees the monotone property
of the original objective; this leads to our terminology of ``nonmonotone MM algorithm''.  Due to the modifications
which are introduced to facilitate the fast and efficient solution of the subproblems in the iterative steps,
a separate proof of convergence is needed for the algorithm.   To ensure the rapid convergence and high-accuracy required of
the solution of the subproblems, we employ the semismooth Newton
method (SN) \cite{QiSun1993, PangQi1995} for semismoothly differentiable (SC$^1$)
functions applied to the dual of such subproblems.
We present numerical results to demonstrate the effectiveness of the overall algorithm and the superiority
of a nontraditional continuous piecewise affine regression model over the traditional linear model using least-squares regression.
In addition to this specific problem, the nonmonotone MM algorithm combined with the SN method (abbreviated as MM+SN)
developed in this paper has wide applications
to many related problems that include the multi-layer neural network problems with piecewise activation functions
in deep learning.  Due to page limit, we will report the details of the other applications in separate papers.

The major contributions of this paper are fourfold: (a) we identify a unified composite difference-max program for a host of
modern statistical estimation problems and illustrate the formulation with a variety of special instances; (b) we develop a
majorization-minimization based algorithm for computing, for the first time, a directional stationary solution of such a
nonconvex, nonsmooth program; 
(c) we present a semismooth Newton method for solving the MM subproblems and illustrate its effectiveness when
applied to a piecewise affine regression problem; and (d) we report computational results on the latter application that
demonstrate the superiority of this nontraditional statistical estimation approach over traditional linear regression.


\section{A Unified Composite Difference-Convex-Piecewise Program}

In the definition of a statistical estimation problem, we place particular emphasis and care
on the mathematical properties of its defining functions; identifying these properties is needed to
effectively deal with the joint feature of nonconvexity and nondifferentiability and to facilitate the design
of the MM+SN algorithm and its analysis for solving the resulting optimization problem.
Specifically, the optimization problem consists of the following four major components, each playing a separate
role in the overall statistical estimation process.  As such, distinguishing them offers flexibility to accommodate
a variety of model specifications.

$\bullet $ A parametric statistical model: $y = \psi(x; \theta) + \varepsilon$, where $y$ is the response (output) variable
(taken to be a scalar) given the input $x \in  \mathbb{R}^d$, $\varepsilon$ is the unobserved random error assumed to have (conditional) mean zero, and $\theta \in  \mathbb{R}^m$ is the model parameter to be estimated,
for some positive integers $d$ (number of input features) and $m$ (number of parameters). In general, these two dimensions, $d$ and $m$, may be different (see e.g.\ (\ref{eq:d-affine})); however, they may be the same as in the case of linear regression.

$\bullet $ An output-dependent loss function $\varphi(y, \bullet)$, whose composition with the statistical model
$\psi(x; \theta)$ leads to the composite function $(\theta;y,x) \mapsto \varphi(y,\psi(x;\theta))$
that provides a measure of the deviation between the model predicted output $\psi(x;\theta)$ and the observed response
$y$; we stress the importance of this composition as properties of $\varphi$ and $\psi$ may be very
different and need to be understood separately for best results.

$\bullet $ A regularizing function $P(\theta)$ that is used either for the purpose of strongly convexifying the
objective function or in high-dimensional problems to control the selection of the model parameter $\theta$;
the latter control is particularly important to avoid overfitting when the number of features $d$ is
relatively large.

$\bullet $ An admissible set $\Theta \subseteq \mathbb{R}^m$ that further restricts the parameter space in the presence of domain knowledge on the parameter $\theta$; often this set $\Theta$ is the whole space $\mathbb{R}^m$, leading to an unconstrained selection of the parameter.

Together, the tuple $(\varphi, \psi, \Theta)$ defines the following constrained minimization problem that is central to many statistical estimation problems:
\begin{equation}\label{eq: model:origional}
\displaystyle{
\operatornamewithlimits{\mbox{minimize}}_{\theta \, \in \, \Theta}
} \ \E \left[ \, \varphi(y,\psi(x;\theta)) \, \right],
\end{equation}
where the expectation $\mathbb{E}$ is taken over the joint distribution of the input $x$ and the output $y$.
The objective expresses the average loss of accuracy in the model output versus the realized output taking into account the
uncertainties in the random input-output pair $(x, y)$. This perspective of statistical estimation is in line with stochastic
programming and is a departure from classical statistical estimation where there is always a postulate of a ``ground truth'' of the estimator.
In this paper, we investigate the numerical solution of problem \eqref{eq: model:origional} via the application of the
sample average approximation (SAA) scheme.  Specifically, taking $N$ independent and identically distributed samples
$\{(x^s,y_s)\}_{s=1}^N \subseteq \mathbb{R}^{d+1}$ and adding the regularizer $P(\theta)$ weighted by the scalar $\gamma_N \geq 0$,
we consider the  empirical optimization problem:
\begin{equation}\label{eq: model: SAA}
\displaystyle{
\operatornamewithlimits{\mbox{minimize}}_{\theta \, \in \, \Theta}
} \ \displaystyle{
\frac{1}{N}
} \, \displaystyle{
\sum_{s=1}^N
} \,
\left[ \, \varphi(y_s,\psi(x^s;\theta)) \, \right] + \gamma_N \, P(\theta).
\end{equation}
While this formulation \eqref{eq: model: SAA} is rather classical, the detailed treatment of the distinguished properties of the
functions $(\varphi,\psi,P)$ constitutes the novelty and intellectual merits of the present paper.
In what follows, we present some details of these functions and describe how they arise.  An important point of the discussion is to
motivate a focused formulation of problem (\ref{eq: model: SAA}) that allows a unified treatment of these statistical estimation
problems as a composite diff-max program by a common algorithmic procedure with guaranteed convergence properties; see
formulation (\ref{eq:optimization model}) and its setting (assumptions {\bf C1}-{\bf C3}) as well as developments in the subsequent
sections. \newline\vskip -0.7em

\noindent
\underline{{\bf Statistical models.}}
We are  interested in piecewise smooth models, including:

$\bullet $ A continuous piecewise affine parametric model: this is a (generally nonconvex) piecewise affine function expressed as
the difference of two convex piecewise affine functions \cite{Scholtes02} each with its own parameters: for two positive integers $k_a$ and $k_b$,
\begin{equation}\label{eq:d-affine}
\psi(x;\theta) \, \triangleq \, \displaystyle{
\max_{1 \leq i \leq k_a}
} \, \left\{ \, ( \, a^i \, )^T x + \alpha_i \, \right\} - \displaystyle{
\max_{1 \leq i \leq k_b}
} \, \left\{ \, ( \, b^i \, )^T x + \beta_i \, \right\},
\end{equation}
with parameter $\theta \,\triangleq \,\left\{ \left( a^i, \alpha_i \right)_{i=1}^{k_a}, \left( b^i, \beta_i \right)_{i=1}^{k_b} \right\}
\in \mathbb{R}^{(k_a + k_b)(d + 1)}$.
The convex case $k_b = 1$ was studied in \cite{HannahDunson11,HannahDunson13,MCISen15}; study of the nonconvex case (where $k_b > 1$)
can be found in the unpublished manuscript \cite{HahnBanergjeeSen16} and also in \cite{Bagirov10} with a max-min representation of the piecewise affine function. Model \eqref{eq:d-affine} includes the 1-layer neural network
by the rectified linear unit (ReLu) activation function \cite{NairHinton2010, GlorotBordesBengio2011} that has the simple
form $\psi(x;\theta) = \max(a^T x + \alpha, 0)$ where the parameter $\theta = (a, \alpha) \in \mathbb{R}^{d+1}$.

$\bullet $ A multi-layer neural network with the ReLu activation function: for simplicity, we present a 2-layer model in deep learning \cite{YuDeng15},
\begin{equation}\label{defn: 2-layer ReLu}
\psi(x;\theta) \, \triangleq \, \max\left( \, b^T \max\left( Ax + a, 0 \right) + \beta, \, 0 \right),
\end{equation}
where the parameter $\theta$ consists of the vectors $b$ and $a$ in $\mathbb{R}^k$, the matrix $A \in \mathbb{R}^{k \times d}$,
and scalar $\beta \in \mathbb{R}$.  The two occurrences of the max ReLu functions indicate the action of $2$ hidden layers,
where the max of $Ax + a$ and $0$ is taken along each coordinate.
Omitting the proof, we can show that this function can be written in the following difference-of-max form:
\begin{equation} \label{eq:dcp}
\psi(x;\theta) \, = \, \displaystyle{
\max_{1 \leq j \leq k_1}
} \, \psi_{1,j}(x;\theta) - \displaystyle{
\max_{1 \leq j \leq k_2}
} \, \psi_{2,j}(x;\theta),
\end{equation}
for some positive integers $k_1$ and $k_2$ and convex functions $\psi_{i,j}(x;\bullet)$ that are all once
but not twice continuously differentiable piecewise linear-quadratic.  (A continuous function $\psi$ is {\sl piecewise linear-quadratic} (PLQ) \cite[Definition 10.20]{RockafellarRWets98}
on a domain if the domain
can be represented as the union of finitely many polyhedral sets 
on each of which  $\psi$ is a quadratic function;
see \cite{RockafellarRWets98, Sun86, Sun92} for properties of piecewise functions of this type.)
The above difference-of-convex pointwise maximum 
representation (in short, difference-max representation)
of $\psi(x;\theta)$ has several advantages over the original definition (\ref{defn: 2-layer ReLu}).
Extending the piecewise affine model (\ref{eq:d-affine}) to a piecewise quadratic model, the formulation (\ref{eq:dcp})
motivates a
unified form of the function $\psi$ in these two models.  More importantly, one can take advantage of
the piecewise linear-quadratic components for the computation of directional
stationary solutions (to be defined later).  Furthermore, while it is clear
from (\ref{defn: 2-layer ReLu}) that $\psi(x;\bullet)$ is a piecewise quadratic function,
definition (\ref{defn: 2-layer ReLu}) does not immediately reveal the
linear-quadratic feature of this function.

We make two important remarks.  One, it is possible to extend the above 2-layer treatment in two major directions: (i) to multi-layers
and (ii) to general piecewise affine activation functions of which the ReLu function is a special case.  Due to their significance,
the full treatment of these deep-learning problems
is presented in a separate paper.  Two, in the algorithmic development, we will
take each $\psi_{i,j}(x;\theta)$ in (\ref{eq:dcp}) as a once continuously differentiable convex function. \newline\vskip -0.7em

%

\noindent
\underline{{\bf Loss functions.}}
These include both differentiable and piecewise affine (thus nondifferentiable) functions. Of particular interest are the following convex loss functions:
the classical quadratic function for least-squares regression; the Huber loss function in robust estimation; the quantile function \cite{KoenkerBassett1978}
that includes the absolute deviation loss function; the one-parameter exponential family via log-likelihood maximization \cite{BickelDoksum2006,Brown1986}.
We are also interested in the nonconvex yet piecewise affine truncated hinge loss function for binary and multicategory classification \cite{ZhangPhamFuLiu17}
with the form  $$\varphi(y,t) \,\triangleq \, \max( 1-(t-y), 0 ) - \max( \delta-(t-y), 0 )\epc \mbox{for some parameter $\delta \leq 0$}.$$
With the proof omitted, we can show that a composition of this function with the
diff-max function (\ref{eq:dcp}) is also a function of the same kind.
Therefore, the composite function $\varphi(y,\psi(x;\theta))$ is itself a difference-max function, thus amenable to
treatment by the methodology developed in the later sections. \newline\vskip -0.7em

\noindent
\underline{{\bf Regularizers.}}  Traditionally, strongly convex regularizers 
are very prominent; in sparse linear regression,
convex and difference-of-convex regularizers are becoming popular as they are employed as surrogate sparsity functions.
Most of the latter regularizers are not differentiable.
Frequently used convex regularizers include the $\ell_2$-norm used in support vector machines \cite{CortesVapnik95}, the weighted $\ell_1$ in sparsity representation \cite{HastieTibshiraniWainwright15} and the total variation norm in image processing \cite{TibshiraniSaunders2005}.
Examples of dc surrogate sparsity functions include the smoothly clipped
absolute deviation (SCAD) function \cite{FanLi01},
the minimax concave penalty  function \cite{Zhang10}, the truncated transformed $\ell_1$ \cite{YLHX2015, DongAhnPang2017}, the truncated logarithmic penalty
\cite{CandesWakinBoyd2008,DongAhnPang2017}. It can be shown that all of the above mentioned dc regularizers can be written in the unified form
$$P(\theta) = \displaystyle{
\sum_{i=1}^m
} \, c_i \, | \, \theta_i \, | - \displaystyle\sum_{i=1}^m p_i(\theta_i)$$ with each $p_i$ being a univariate differentiable convex function, see e.g., \cite{AhnPangXin17}. \newline\vskip -0.7em

Putting together the above families of functions $\varphi(y,\bullet)$, $\psi(x;\bullet)$ and $P(\bullet)$,
we arrive at the following detailed formulation of problem \eqref{eq: model: SAA}
\begin{equation}\label{eq:optimization model}
\displaystyle{
\operatornamewithlimits{\mbox{minimize}}_{\theta \, \in \, \Theta}
} \epc  f_N(\theta)\,\triangleq \,\displaystyle{
\frac{1}{N}
} \, \displaystyle{
\sum_{s=1}^N
} \, \varphi_s \circ \psi_s(\theta) + \gamma_N\,[\,P_1(\theta) - P_2(\theta)\,]
\end{equation}
with the composite diff-max structure summarized below:\\
{\bf C1:}  each $\varphi_s$ is a univariate convex function;\\
{\bf C2:} each $\psi_s$ is a difference-of-convex function of the pointwise max type (\ref{eq:dcp});
specifically, for some nonnegative integers $k_{s;1}$ and $k_{s;2}$,
\begin{equation}\label{dpc}
\psi_s(\theta) \,\triangleq\, \displaystyle{
\max_{1 \leq i \leq k_{s;1}}
} \, \psi_{s;1,i}(\theta) - \displaystyle{
\max_{1 \leq i \leq k_{s;2}}
} \,\psi_{s;2,i}(\theta)
\end{equation}
with each $\psi_{s;1,i}(\theta)$ and $\psi_{s;2,i}(\theta)$ being convex and continuously differentiable;\\
{\bf C3:} both $P_1$ and $P_2$ are convex with $P_2$ being additionally the pointwise maximum of finitely many convex differentiable functions.  \hfill $\Box$   \newline\vskip -0.7em

In this setting, the overall objective function $f_N$ is not convex.
Since the composition of a convex function with a dc function is of the dc type,
by \cite[Theorem II, page 708]{Hartman1959}, it follows that $f_N$ is a dc function; thus the difference-of-convex
algorithm (DCA) in dc programming \cite{LeThiPham05,LeHuynhPham09} is in principle applicable to compute a critical point
of $f_N$ on the feasible set $\Theta$. In general, for a dc program:
$\displaystyle\operatornamewithlimits{\mbox{minimize}}_{z\in Z} \;[\,g(z)-h(z)\,]$ where $Z$ is a closed convex set in $\mathbb{R}^n$
and $g$ and $h$ are convex functions, a vector $\bar{z}\in Z$ is a critical point if
$\partial h(\bar{z}) \, \cap \, [ \, \partial g(\bar{z})+ \mathcal{N}(\bar{z}; \, Z)\, ] \neq  \emptyset$,
where the notation $\partial \varphi(\bar{z})$ denotes the subdifferential of a convex function $\varphi$ at a given vector $\bar{z}$ and $\mathcal{N}(\bar{z};\, Z)$
denotes the normal cone of the set $Z$ at $\bar{z}\in Z$. Nevertheless, this criticality concept has several major drawbacks.
One, it depends on the dc representation of the objective function $f_N$. Two, although $f_N$ is known to be dc, a
dc decomposition is not readily available when each of the composite functions $\varphi_s \circ \psi_s$ is derived from
the statistical functions given above.  Third, as will be shown in the next section, criticality is a very weak property
and can have no bearings at all with a desirable minimizing property. For these reasons, our research goal is to seek an alternative definition of stationarity
that is the ``sharpest'' of its kind and which is applicable to  problem~\eqref{eq:dcp} without demanding a dc decomposition of the composite  functions
$\varphi_s \circ \psi_s$.
One shall see from the subsequent discussion that the particular (difference-of) pointwise maximum structures of $\psi_s$ and $P_2$ are critical
to achieve this goal.

\section{A Detour: Subdifferentials and Stationarity}  \label{sec: preliminary}

As a first step in studying the composite nonconvex, nondifferentiable problem (\ref{eq:optimization model}),
we take a detour to present some results from variational analysis; 
see \cite{RockafellarRWets98}
for details.  These materials will prepare for
the introduction of two key notions of stationary solutions that provide the objects of convergence of the iterative
algorithms for solving problem~\eqref{eq:optimization model}, to be presented in the following sections.
Let $\Phi:\Omega\to \mathbb{R}^m$ be a locally Lipschitz continuous vector-valued function defined on an open set $\Omega \subseteq \mathbb{R}^n$.
It follows that $\Phi$ is F(r{\'e}chet)-differentiable almost everywhere on $\Omega$ (c.f.~\cite[Theorem 9.60]{RockafellarRWets98}).
Denote by $\mathcal{D}_{\Phi} \subseteq \Omega$ the set of points where $\Phi$ is F-differentiable and by $J\Phi(x)\in \mathbb{R}^{m\times n}$
the Jacobian of $\Phi$ at $x\in \mathcal{D}_{\Phi}$.  Let $\bar{x} \in \Omega$ be given.
The B(ouligand)-subdifferential of $\Phi$ at $\bar{x}$ is denoted by
\[
\partial_B \Phi(\bar{x}) \triangleq  \left\{ \, V \, \in \, \mathbb{R}^{m\times n} \, \bigg| \; \, \exists \,
\{x^k\}\,\subseteq\, \mathcal{D}_\Phi \;{\rm with}\; \displaystyle{
\lim_{k \to \infty}
} \, x^k \, = \, \bar{x} \mbox{ and } \displaystyle{
\lim_{k \to \infty}
} \, J\Phi(x^k) \, = \,  V \, \right\}.
\] 
The Clarke subdifferential (also called the Clarke generalized Jacobian) of $\Phi$ at $\bar{x}$ is defined as $
\partial_C \Phi(\bar{x}) \,\triangleq \,  \mbox{conv}\left(\,\partial_B\,\Phi(\bar{x})\,\right)$,
where ``conv'' stands for the convex hull of a given set.  For a real-valued function $\phi : \Omega \to \mathbb{R}$,
the Clarke subdifferential of $\phi$ at $\bar{x}$ can also be characterized by
 \[
\partial_C \phi(\bar{x})  \,= \, \left\{ \, v \, \in \, \mathbb{R}^n \; \bigg| \;
\limsup_{x\to\bar{x},\; t\downarrow 0}  \; \frac{\phi(x+tw) - \phi(x)-t\,v^Tw}{t} \, \geq \, 0, \quad \forall \; w \, \in \, \mathbb{R}^n \, \right\}.
\]
The regular subdifferential of $\phi$ at $\bar{x}$ is defined as
\[
\wh{\partial} \phi(\bar{x}) \, \triangleq \, \left\{ \, v \, \in \mathbb{R}^n \; \bigg| \;
\liminf_{\bar{x} \neq x\to\bar{x}} \; \frac{\phi(x) - \phi(\bar{x})-v^T(x-\bar{x})}{\| \, x-\bar{x} \, \|} \, \geq \, 0 \, \right\}.
\]
The limiting subdifferential of $\phi$ at $\bar{x}$ is defined as
\[
\partial \phi(\bar{x}) \, \triangleq \, \left\{ \, v \, \in \, \mathbb{R}^n \; \bigg| \; \exists \, \{ x^k \} \to \bar{x}
\mbox{ and } \{ v^k \} \to v \mbox{ such that } v^k \, \in \, \wh{\partial} \, \phi(x^k) \mbox{ for all $k$ } \, \right\}.
\] 
When $\phi$ is convex, its Clarke subdifferential, regular subdifferential and limiting subdifferential
coincide with the set of all subgradients of $\phi$.  In general, the relationship between the above four
definitions is demonstrated in Figure \ref{figure:subdiff}.  The detailed explanations are given in the proposition below.

\begin{proposition}\label{prop: subdifferentials}
Let $\Omega$ be an open set in $\mathbb{R}^n$.
For any locally Lipschitz continuous function $\phi : \Omega\to \mathbb{R}$ and any $\bar{x}\in \Omega$, it holds that:\\
(i) $\left(\,\partial_B \phi(\bar{x}) \, \cup \, \wh{\partial} \phi(\bar{x})\,\right)
\,\subseteq \, \partial \phi(\bar{x})\,\subseteq\, \partial_C \phi(\bar{x})$;\\
%
%
(ii) ${\rm conv}\, \left(\,\wh{\partial} \phi(\bar{x})\,\right) \, \subseteq \, {\rm conv} \, \left( \, \partial_B \phi(\bar{x}) \,\right)
\, = \, {\rm conv}\, \left( \,\partial \phi(\bar{x})\, \right) \, = \, \partial_C \phi(\bar{x})$.\\
Moreover, counter-examples exist for the omitted inclusions.
\end{proposition}
\begin{proof}
It is known from \cite[Theorem 8.6]{RockafellarRWets98} that $\wh{\partial} \phi(\bar{x}) \, \subseteq \, \partial  \phi(\bar{x})$ and
from \cite[Theorem 3.57]{Mordukhovich2006} that ${\rm conv} ( \,\partial  \phi(\bar{x}) \, ) \, = \, \partial_C \, \phi(\bar{x})$.
Thus, 
it remains to show that $\partial_B \phi(\bar{x})\subseteq \partial \phi(\bar{x})$.
Let $\bar{v} \in \partial_B \phi(\bar{x})$.
Then there exists $\{x^k\} \subseteq \Omega$ converging to $\bar{x}$ such that $\phi$ is F-differentiable at $x^k$ and
$\nabla \phi(x^k) \to \bar{v}$.  Since $\wh{\partial} \phi(x^k) = \{ \nabla \phi(x^k)\}$ (c.f.\ \cite[Exercise 8.8]{RockafellarRWets98}),
we have $\bar{v}\in \partial \phi(\bar{x})$ by the definition of $\partial \phi$.

We provide two examples to show the possible invalidity of the omitted inclusions.
For $\phi(x) = |x|$, we have $ \partial \phi(0) \,=\,  \wh{\partial} \phi(0) \,= \,[-1,1] \nsubseteq  \partial_B \phi(0) \, = \, \{ -1,1 \}$.
For $\phi(x) = -|x|$,  we have $\partial_C \phi(0) \,= \,[-1,1] \, \nsubseteq \, {\partial}_B \phi(0) \, = \, \partial \phi(0) \, = \, \{-1,1\} \, \nsubseteq \,
\wh{\partial} \phi(0) \, = \, \emptyset$ and ${\rm conv}\, (\,\partial_B \phi(0)\,)\,=\,[-1,1]\,\nsubseteq\, {\rm conv}\, (\,\wh{\partial} \phi(0)\,) \, = \, \emptyset$.
\end{proof}

\begin{figure}[H]
\centering
\begin{minipage}{.47\textwidth}
\begin{center}
\begin{tikzpicture}[scale = 0.7]
    \begin{scope}[blend group=soft light]
    \fill[blue!40!white]    (0,0)  ellipse (4 and 2.1);
    \fill[blue!20!white]  (-1,0) ellipse (2.5 and 1.7);
    \fill[blue!2!white]   (-1,0.7) ellipse (1.1 and 0.8);
    \fill[blue!2!white] (-2,0) ellipse (0.85 and 1.1);
    \end{scope}
    \node at (2.5,0) (A) {\small ${\partial_C \phi(x)}$};
    \node at (0,-0.5) (B) {\small ${\partial \phi(x)}$};
    \node at (-0.6,0.6) (C) {\small ${\wh{\partial} \phi(x)}$};
    \node at (-2,-0.4) (D) {\small ${{\partial}_B \phi(x)}$};
    \end{tikzpicture}
    \caption{\scriptsize Relationship between the subdifferentials}
    \label{figure:subdiff}
    \end{center}
\end{minipage}
\begin{minipage}{.5\textwidth}
\begin{center}
\begin{tikzpicture}[scale = 0.65]
    \begin{scope}
    \fill[blue!40!white]    (0,0)  ellipse (4.5 and 2.3);
    \fill[blue!30!white]  (-1,0) ellipse (3.2 and 1.7);
    \fill[blue!20!white]   (-1.6,0) ellipse (2.4 and 1.3);
    \fill[blue!10!white] (-2.2,0) ellipse (1.7 and 1);
    \fill[blue!2!white]    (-2.9,0.2) ellipse (0.7 and 0.5);
    \end{scope}
\node at (2.6,-1) (A) {\scriptsize{\textbf{critical}}};
     \node at (1.8,-1.5) (A) {\scriptsize{\textbf{(for dc fncs.)}}};
    \node at (0.9,-0.9) (B)  {\scriptsize{\textbf{C-stat}}};
    \node at (-0.2,-0.6) (C)  {\scriptsize\textbf{l-stat}};
    \node at (-1.3,0) (D)  {\scriptsize\textbf{d-stat}};
     \node at (-2.2,-0.5) (A) {\scriptsize{\textbf{(for dd fncs.)}}};
\node at (-2.9,0.2) (E)  {\scriptsize\textbf{locmin}};
\end{tikzpicture}
\caption{\scriptsize Relationship between the stationary points}
 \label{figure:stationarity}
 \end{center}
\end{minipage}
\end{figure}
\vskip -0.5cm

Let $X$ be a closed convex set in $\Omega$. It is known that a necessary condition for $\bar{x}\in X$ to be a local minimum of $\phi$ is $0\in \wh{\partial}\left(\phi(\bar{x}) + \delta_X(\bar{x})\right)$ \cite[Theorem 10.1]{RockafellarRWets98}, where $\delta_X(x) \,\triangleq\, \left\{
\begin{array}{ll} 0 & \mbox{if $x\in X$} \\
+\infty & \mbox{otherwise}
\end{array}
\right.$ is the indicator function of $X$ at $x$.
If $\phi$ is  locally Lipschitz continuous near $\bar{x}$ and directionally differentiable (dd) at $\bar{x}$, the latter condition is equivalent to the d(irectional)-stationarity of $\bar{x}$, i.e.,
$$
\phi^{\, \prime}(\bar{x};v) \, \triangleq \, \displaystyle{
\lim_{\delta\downarrow 0}
} \, \displaystyle{
\frac{\phi(\bar{x} + \delta v) - \phi(\bar{x})}{\delta}
}  \,\geq\, 0, \epc \forall \, v \, \in \, X - \bar{x}.
$$
Using the limiting subdifferential and Clarke subdifferential, respectively, we say that a point $\bar{x}\in X$ is a
l(imiting)-stationary point of $\phi$ on $X$ if
$0\in {\partial}\left(\phi(\bar{x}) + \delta_X(\bar{x})\right)$, and a
C(larke)-stationary point if $0\in 
{\partial}_C \phi(\bar{x})  + \mathcal{N}(\bar{x};X)$, which implies that
\[\phi^{\circ}(\bar{x};v) \, \triangleq \displaystyle{
\limsup_{x \to \bar{x},\,\delta \downarrow 0}
} \, \displaystyle{
\frac{\phi(x + \delta v)- \phi(x)}{\delta}
} \, \geq \, 0, \epc \forall \, v \, \in \, X - \bar{x}.
\]
Based on Proposition \ref{prop: subdifferentials} and \cite[Exercise 10.10]{RockafellarRWets98}, l-stationarity implies C-stationarity; and if $\phi$ is locally Lipschitz near $\bar{x}$ and dd at this point, d-stationarity implies l-stationarity.
The following examples show that the reverse implications do not hold without additional assumptions.

\begin{example}
Consider the univariate function $\phi(x) =\max( \,-|x|, x-1\,)$ for $x\in \mathbb{R}$.  Since $\partial_C \,\phi(0) = [-1,1]$ and
$\partial \phi(0) = \{-1,1\}$, it holds that $x=0$ is a C-stationary point of $\phi$, but fails to be a l-stationary point.
The unique l-stationary point of $\phi$ on $\mathbb{R}$ is  $x = \frac{1}{2}$.
For the univariate function $\phi(x) = \max\left(\,-x-1,\,\min(-x,\,0)\,\right)$,
since $\partial \phi(0) = \{-1,0\}$ and $\wh{\partial} \phi(0) = \emptyset$, it holds that $x=0$ is a l-stationary point of $\phi$,
but not a d-stationary point.  The unique d-stationary point of $\phi$ on $\mathbb{R}$ is  $x = -1$. \hfill $\Box$
\end{example}

If  $\phi = \phi_1 - \phi_2$ is a difference of two convex functions $\phi_1$ and $\phi_2$, it holds that
$\partial_C \phi(x) \subseteq \partial_C \phi_1(x) -\partial_C \phi_2(x)$ for any $x\in \Omega$
(cf.\ \cite[Corollary 2]{Clarke1983}).   A point $\bar{x}\in \Omega$ is said to be a  critical point of $\phi$ on $X$ if
$$\partial_C \phi_2(\bar{x}) \, \cap \, [ \,\partial_C \phi_1(\bar{x}) + \mathcal{N}(\bar{x};X) \, ]
\, = \, \partial \phi_2(\bar{x}) \, \cap \, [ \, \partial \phi_1(\bar{x}) + \mathcal{N}(\bar{x};X) \, ] \neq \emptyset.$$
Different from the above mentioned concepts of stationary points, a critical point depends on the dc decomposition.
Since there are infinite many dc decompositions of a given dc function, it is likely that a critical point provides no information
on the local minima of that function.  This can be seen from the following example.
%


\begin{example}\label{ex: critical}
Consider the univariate function  $\phi(x) = \max\{x,-x-4\}$ on $\mathbb{R}$, for which $x = -2$ is the
unique C-stationary point (and the global minimizer).  For the dc decomposition $\phi = \phi_1 - \phi_2$ with
$\phi_1(x) =\max\{2x,0, -2x - 4\}$ and $\phi_2(x) = |x|$,
we have $\partial \phi_1(0) \, \cap \, \partial \phi_2(0) = [0,1]$,  implying that $x = 0$ is a critical point of $\phi$ on $\mathbb{R}$.
\hfill $\Box$
\end{example}

The relationship between the various kinds of stationary points and a critical point (for dc problems) is summarized in  Figure \ref{figure:stationarity}.
As a caution to the reader, we note that $0 \in \partial_B \phi(\bar{x})$ is not a necessary condition for $\bar{x}$
being a local minima of a locally Lipschitz continuous function $\phi$; this can be seen from $\phi(x) = |x|$ at $x = 0$.

We close this section by mentioning three properties of d-stationarity that further highlight the fundamental importance of this
stationarity concept.  The first property asserts that a d-stationary point must be ``locally $\varepsilon$-first-order minimizing'';
the second and third property are applicable to the least-squares piecewise affine regression problem.
We recall that a function $\psi$ is B(ouligand)-differentiable at a point $\bar{x}$ \cite[Definition~3.1.2]{FacchineiPang2003}
if it is both locally Lipschitz continuous near $\bar{x}$ and directionally differentiable at $\bar{x}$.
Proof of the proposition below
is omitted as it is not difficult; for related results, see \cite{CuiPang18,ChangHongPang17}.

\begin{proposition} \label{pr:properties of d-stat}
Let $X$ be a closed convex set contained in the open set $\Omega \subseteq \mathbb{R}^n$.  The following three statements hold:\\
(i) Let $\psi : \Omega \to \mathbb{R}$ be B-differentiable at a d-stationary point $\bar{x} \in X$.
It holds that for every $\varepsilon > 0$, there exists an open neighborhood ${\cal N}$ of $\bar{x}$ such that
$\psi(x) \geq \psi(\bar{x}) - \varepsilon \, \| x - \bar{x} \, \|$
for all $x \in X \cap {\cal N}$.\\
(ii) Let $\psi : \Omega \to \mathbb{R}^m$ be piecewise affine and $\Phi : \mathbb{R}^m \to \mathbb{R}$ be convex.  It holds that every
d-stationary point of the composite function $\Phi \circ \psi$ on the set $X$ is a local minimizer.\\
(iii)  If $X$ is polyhedral and $\psi$ is piecewise linear-quadratic on $X$, then the set of values of $\psi$ on the set of
d-stationary points of $\psi$ on $X$ is finite.
\end{proposition}

\section{The Nonmonotone Majorization-Minimization Algorithm} \label{sec: MM}

We aim to compute a d-stationary solution of problem \eqref{eq:optimization model}.
The approach we propose is to apply the majorization-minimization (MM) algorithm with suitable modifications.
Originally described in \cite[Section 9.3(d)]{OrtegaRheinboldt1970}, the idea of the basic MM algorithm is to solve
a sequence of convex minimization subproblems by creating surrogate functions that majorize the original objective function.
It unifies various optimization methods, including the projected/proximal gradient method \cite{Bertsekas16}
and the dc algorithm \cite{LeThiPham05,LeHuynhPham09}.  It is a generalization of the expectation-maximization (EM) algorithm, as known in the statistics community,
for finding maximum likelihood estimators of parameters in statistical models \cite{DempsterLairdRubin1977,Wu1983}.
See \cite{HunterLange2004,Lange16} for comprehensive discussions of the MM algorithm (without modifications)
and \cite{Mairal2013,Mairal2015,BoltePauwels2016} for recent developments.

Given a locally Lipschitz continuous function $\phi : \Omega \to \mathbb{R}$ defined on an open set $\Omega \subseteq \mathbb{R}^n$
containing the closed convex set $X$ and a point $y \in X$,
the continuous function $\wh{\phi}\,(\bullet,y)$ is said to be a majorizing function of $\phi$  if: (i) $\wh{\phi}\,(\bullet,\,y)$ is
convex on $X$; (ii) $\phi(x) \leq \wh{\phi}\,(x,\,y)$ for all $x\in X$; and (iii) $\wh{\phi}\,(y) = \wh{\phi}\,(y,\,y)$.
The key to a successful application of the MM algorithm to the optimization problem \eqref{eq:optimization model}
hinges on a readily available convex majorization of the objective function $f_N$ at a given iterate $\theta^{\, \nu}$.
In turn, it suffices to derive such a majorizing function for each summand $\varphi_s \circ \psi_s$;
for the dc regularizer $P(\theta) = P_1(\theta)- P_2(\theta)$ a convex majorizing function is easily obtained as:
$\wh{P} (\theta,\theta^{\, \nu}) = P_1(\theta)-  [ \, P_2(\theta^{\, \nu})+(a^{2; \nu})^T( \theta - \theta^{\, \nu}) \, ]$,
where $a^{2; \nu} \in \partial P_2(\theta^\nu)$.  The following fact combined with the difference-max structure \eqref{dpc}
of each function $\psi_s$ easily yields a convex majorant of each composite function $\varphi_s \circ \psi_s$.



\begin{lemma}\label{lemma: univariate function}
A univariate convex function $f$ can be written as the
sum of a convex non-decreasing function $f^{\, \uparrow}$ and a convex non-increasing function $f^{\, \downarrow}$.
\end{lemma}
\begin{proof}
%
Note that if there exists $t_0\in \mathbb{R}$ such that $f^{\, \prime}(t_0;\pm 1)\geq 0$, then $f$ achieves its minimum value at
$t_0$.  By the convexity of $f$, one can easily check that $f$ is non-increasing on $(-\infty, t_0\,]$ and non-decreasing on $[\, t_0, \infty)$,
leading to a choice of  $f^{\, \uparrow}(t) \, \triangleq \,\left\{\begin{array}{ll}
f(t) - f(t_0) & \mbox{if $t \geq t_0$} \\
 0 & \mbox{if $t < t_0$}
\end{array} \right.$ and $f^{\, \downarrow}(t) \, \triangleq \, \left\{ \begin{array}{ll}
f(t_0) & \mbox{if $t \geq t_0$}\\
f(t) & \mbox{if $t < t_0$}.
\end{array}\right.$
To see the convexity of $f^{\, \uparrow}$, it suffices to observe that for any $t_1$ and $t_2$ in
$\mathbb{R}$ satisfying $t_1 < t_0 < t_2$ and for any $\lambda\in (0,1)$,
\[ \begin{array}{rl}
f^{\, \uparrow}((1-\lambda)t_1  + \lambda t_2)
 \leq & f^{\, \uparrow}((1-\lambda)t_0  + \lambda t_2) \,= \,f((1-\lambda)t_0  + \lambda t_2)-f(t_0)\\[0.1in]
 \leq & (1-\lambda) f(t_0) + \lambda f(t_2) - f(t_0)
\, = \, (1-\lambda) \, f^{\, \uparrow}(t_1) + \lambda \, f^{\, \uparrow}(t_2).
\end{array}
\]
By a similar argument, we can show the convexity of $f^{\, \downarrow}$.

If such $t_0$ does not exist, then for all $t\in \mathbb{R}$, either $f^{\, \prime}(t;1) \geq 0 \geq f^{\, \prime}(t;-1)$, or
$f^{\, \prime}(t;1) \leq 0 \leq f^{\, \prime}(t;-1)$.  The former situation implies that $f$ is
non-decreasing on $\mathbb{R}$; hence we may take $f^{\, \uparrow}=f$ and $f^{\, \downarrow}=0$ in this case; while the
latter situation implies that $f$ is non-increasing on $\mathbb{R}$; hence we may choose $f^{\, \downarrow} = f$ and $f^{\, \uparrow} =0$.
\end{proof}


Since a convex non-decreasing (non-increasing) function composed with a convex (concave) function is convex,
the following corollary is easy to obtain.  The proof is omitted for brevity.

\begin{corollary}\label{corollary: composite majorant}
 Let $\varphi$ be a univariate convex function with the
decomposition $\varphi = \varphi^{\, \uparrow} + \varphi^{\, \downarrow}$ as in Lemma~\ref{lemma: univariate function}.
For any $\theta^{\, 0} \in \Theta$,
if $\wh{\psi}(\bullet,\theta^{\, 0})$ is a concave minorizing function and $\widecheck{\psi}(\bullet,\theta^{\, 0})$ a convex
majoring function of $\psi$ on $\Theta$, i.e.,
\[
\wh{\psi}(\theta,\theta^{\, 0})\leq \psi(\theta) \leq \widecheck{\psi}(\theta,\theta^{\, 0}),\epc \forall\;\theta\in \Theta,
\]
then
$\varphi^{\, \uparrow} \circ \widecheck{\psi}(\bullet,\theta^{\, 0}) + \varphi^{\, \downarrow} \circ \wh{\psi}(\bullet,\theta^{\, 0})$
is a convex majorant of the composite $\varphi\circ \psi$ on $\Theta$.
\end{corollary}

\gap
To simplify the notation, we present the MM algorithm and its convergence property for solving the following
single-summand formulation of \eqref{eq:optimization model}:
\begin{equation}\label{eq:optimization model2}
\operatornamewithlimits{\mbox{minimize}}_{\theta \in \Theta}\epc  \Psi(\theta)\,\triangleq \, \varphi \circ \psi(\theta),
\end{equation}
where $\varphi$ is a univariate convex function and $\psi(\theta) = g(\theta) - h(\theta)$, where
\[
g(\theta) \, \triangleq \, \max_{1\leq i\leq k_1} \psi_{1,i}(\theta) \epc \mbox{and} \epc
h(\theta) \, \triangleq \, \max_{1\leq i\leq k_2} \psi_{2,i}(\theta), 
\]
with each $\psi_{1,i}$ and $\psi_{2,i}$ being a differentiable convex function.
The treatment is clearly extendable to $\displaystyle\frac{1}{N}\sum_{s=1}^N \varphi_s \circ \psi_s(\theta) + \gamma_N \,P(\theta)$
with each pair $(\varphi_s, \psi_s)$ as above and $P$ satisfying assumption {\bf C3}.  We omit the details of the extended treatment
but will employ it in the computational experiments; see Section~\ref{sec:numerical}.

Denote the index sets of
maximizing functions in $g$ and $h$ at $\theta\in \Theta$ as,
\[
\mathcal{A}_{1}(\theta) \, \triangleq  \,\displaystyle{
\operatornamewithlimits{\mbox{argmax}}_{1\leq i\leq k_1}
} \, \left\{\, \psi_{1,i}(\theta)\, \right\} \epc \mbox{and} \epc
\mathcal{A}_{2}(\theta) \, \triangleq \, \displaystyle{
\operatornamewithlimits{\mbox{argmax}}_{1\leq i\leq k_2}
} \, \left\{\, \psi_{2,i}(\theta)\, \right\}.
\]
For any $(\theta, \bar{\theta})\in \Theta \times \Theta$ and any $(i_1,i_2)\in \mathcal{A}_{1}( \bar{\theta} )\times \mathcal{A}_{2}( \bar{\theta} )$,
we have
$$
\underbrace{\left[\, g( \bar{\theta} ) + \nabla \psi_{1,i_1}( \bar{\theta} )^T \left( \, \theta - \bar{\theta} \, \right) \, \right]}_{
\mbox{linearization of $\psi_{1,i_1}$ at $\bar{\theta}$}} - h(\theta) \, \leq \, \psi(\theta) \, \leq \,
g(\theta) - \underbrace{\left[ \, h( \bar{\theta} ) + \nabla \psi_{2,i_2}( \bar{\theta} )^T \left( \, \theta - \bar{\theta} \, \right) \, \right]}_{
\mbox{linearization of $\psi_{2,i_2}$ at $\bar{\theta}$}} \,,
$$
which, by Corollary~\ref{corollary: composite majorant}, leads to the following convex majorant of $\Psi$
in \eqref{eq:optimization model2}:
\[  \begin{array}{ll}
\mathcal{M}\Psi_{(i_1,i_2)}(\theta,\bar{\theta}) \triangleq
& \varphi^{\, \uparrow}\left(\,g(\theta) - \left[ h( \bar{\theta} ) + \nabla \psi_{2,i_2}( \bar{\theta} )^T \left( \theta - \bar{\theta} \right) \right] \,\right) +
\\[0.1in]
& \varphi^{\, \downarrow}\left(\,\left[ g( \bar{\theta} ) + \nabla \psi_{1,i_1}( \bar{\theta} )^T \left( \theta - \bar{\theta} \right) \right]
- h(\theta)\, \right).
 \end{array}
\]
The above function is nonsmooth because of the ``max'' in the functions $g$ and $h$.
Notice that we essentially leave $\varphi$ unchanged instead of approximating it.  In this way, we may presumably obtain a tighter approximation
of the original composite objective function $\varphi \circ \psi$.

\subsection{From pointwise max to constraints}

Before describing the MM-based algorithm that can be shown to converge to a d-stationary solution of
problem~\eqref{eq:optimization model2}, we first present the following lemma that characterizes a d-stationary solution of this problem 
as a solution of finitely many convex programs.

\begin{lemma}\label{lemma: d-stationary}
The point $\overline{\theta}\in \Theta$ is a d-stationary point of \eqref{eq:optimization model2} if and only if for all
$({i}_1, {i}_2)\in \mathcal{A}_1(\bar{\theta})\times \mathcal{A}_2(\bar{\theta})$,
$\bar{\theta}\in \displaystyle\operatornamewithlimits{\mbox{argmin}}_{\theta\in \Theta} \mathcal{M}\Psi_{({i}_1, {i}_2)}(\theta,\bar{\theta})$.
\end{lemma}

\begin{proof}
``Only if.''  If $\overline{\theta}\in \Theta$ is a d-stationary point of \eqref{eq:optimization model}, then for all $\theta\,\in \, \Theta$,
\[\begin{array}{rl}
0 \,  \leq \, \Psi^{\, \prime}(\overline{\theta};\theta - \overline{\theta}) \, = &
\varphi^{\, \prime}(\psi(\overline{\theta});\psi^{\, \prime}(\overline{\theta};\theta - \overline{\theta}))
\\[0.1in]
 = & \left( \varphi^{\, \uparrow} \right)^{\prime}\left(\psi(\overline{\theta});\psi^{\, \prime}(\overline{\theta};\theta - \overline{\theta})\right)
+ \left( \varphi^{\, \downarrow} \right)^{\prime}\left(\psi(\overline{\theta});\psi^{\, \prime}(\overline{\theta};\theta - \overline{\theta})\right).
\end{array}
\]
We have $\psi(\overline{\theta}) = g(\overline{\theta}) - h(\overline{\theta})$ and
\[
\psi^{\, \prime}(\overline{\theta};\theta - \overline{\theta}) \, = \, \underbrace{\displaystyle{
\max_{i \in \mathcal{A}_1(\overline{\theta})}
} \, \nabla \psi_{1,i}(\overline{\theta})^T \left( \, \theta - \overline{\theta} \, \right)}_{\mbox{$= g^{\, \prime}(\overline{\theta};\theta - \overline{\theta})$}}
- \underbrace{\displaystyle{
\max_{i \in \mathcal{A}_2(\overline{\theta})}
} \, \nabla \psi_{2,i}(\overline{\theta})^T \left( \, \theta - \overline{\theta} \, \right)}_{\mbox{$= h^{\, \prime}(\overline{\theta};\theta - \overline{\theta})$}}.
\] 
Since $\varphi^{\, \uparrow}$ is a univariate non-decreasing convex function, it follows that the directional derivative
$\left( \varphi^{\, \uparrow} \right)^{\prime}(t;\bullet)$ is a non-decreasing function of the direction for
every $t \in \mathbb{R}$.  Hence, for every $i_2 \in \mathcal{A}_2(\overline{\theta})$,
\[ \begin{array}{l}
\left( \varphi^{\, \uparrow} \right)^{\prime}(\psi(\overline{\theta});\psi^{\, \prime}(\overline{\theta};\theta - \overline{\theta}))
 =  \left( \varphi^{\, \uparrow} \right)^{\prime}\left( g(\overline{\theta}) - h(\overline{\theta});
g^{\, \prime}(\overline{\theta};\theta - \overline{\theta}) - h^{\, \prime}(\overline{\theta};\theta - \overline{\theta}) \right) \\ [5pt]
\qquad\qquad \qquad\qquad \qquad \epc \leq  \left( \varphi^{\, \uparrow} \right)^{\prime}\left( g(\overline{\theta}) - h(\overline{\theta});
g^{\, \prime}(\overline{\theta};\theta - \overline{\theta}) - \nabla \psi_{2,i_2}(\overline{\theta})^T \left( \, \theta - \overline{\theta} \, \right) \right).
\end{array}
\]
Similarly, we also have, for every $i_1 \in \mathcal{A}_1(\overline{\theta})$,
\[
\left( \varphi^{\, \downarrow} \right)^{\prime}(\psi(\overline{\theta});\psi^{\, \prime}(\overline{\theta};\theta - \overline{\theta}))
\, \leq \, \left( \varphi^{\, \downarrow} \right)^{\prime}\left( g(\overline{\theta}) - h(\overline{\theta});
\nabla \psi_{1,i_1}(\overline{\theta})^T \left( \, \theta - \overline{\theta} \, \right) - h^{\, \prime}(\overline{\theta};\theta - \overline{\theta}) \right),
\]
because $\left( \varphi^{\, \downarrow} \right)^{\prime}(t;\bullet)$ is a non-increasing function  for
every $t \in \mathbb{R}$.  It therefore follows that
${\cal M}\Psi_{(i_1,i_2)}(\bullet,\bar{\theta})^{\, \prime}(\bar{\theta};\theta - \bar{\theta}) \, \geq \, 0$ for all $\theta \, \in \, \Theta$.
Since ${\cal M}\Psi_{(i_1,i_2)}(\bullet,\bar{\theta})$ is a convex program, it follows that $\bar{\theta}$ is a minimizer of this function over
$\Theta$.

\gap

``If.''  Conversely, suppose that for every
$({i}_1, {i}_2)\in \mathcal{A}_1(\bar{\theta})\times \mathcal{A}_2(\bar{\theta})$, it holds that
$\bar{\theta}\in \displaystyle\operatornamewithlimits{\mbox{argmin}}_{\theta\in \Theta} \mathcal{M}\Psi_{({i}_1, {i}_2)}(\theta,\bar{\theta})$.
Let $\theta \in \Theta$ be arbitrary.  Pick $i_1 \in \displaystyle{
\operatornamewithlimits{\mbox{argmax}}_{i \in \mathcal{A}_1(\overline{\theta})}
} \, \nabla \psi_{1,i}(\overline{\theta})^T \left( \, \theta - \overline{\theta} \, \right)$ and
$i_2 \in \displaystyle{
\operatornamewithlimits{\mbox{argmax}}_{i \in \mathcal{A}_2(\overline{\theta})}
} \, \nabla \psi_{2,i}(\overline{\theta})^T \left( \, \theta - \overline{\theta} \, \right)$.  We then have
\[
\Psi^{\, \prime}(\overline{\theta};\theta - \overline{\theta})
\, = \, {\cal M}\Psi_{(i_1,i_2)}(\bullet,\bar{\theta})^{\, \prime}(\bar{\theta};\theta - \bar{\theta}) \, \geq \, 0.
\]
Since $\theta \in \Theta$ is arbitrary, it follows that $\bar{\theta}$ is a d-stationary point of $\Psi$ on $\Theta$.
\end{proof}

Lemma~\ref{lemma: d-stationary} indicates that there is a ``for all index pairs'' condition in the requirement of d-stationarity.
Though $\mathcal{M}\Psi_{(i_1,i_2)}(\bullet,\theta^{\,\nu})$ is a convex majorant of $\Psi$ for any pair
$(i_1, i_2)\in \mathcal{A}_1(\theta^{\,\nu})\times \mathcal{A}_2(\theta^{\,\nu})$, by arbitrarily picking a single pair of indices
$(i_1,i_2)$ at each iteration of the MM algorithm, one may not obtain an algorithm that converges to a d-stationary point.
This non-convergence has been observed in \cite{PangRazaviyaynAlvarado16} when the dc algorithm (a special MM algorithm)
is applied to solve dc programs (a special case of problem \eqref{eq:optimization model2} with $\varphi$ being the identity function);
see specifically Example~4 in this reference.
In order for the MM algorithm to converge to a d-stationary point, we employ the $\varepsilon$-technique to
expand the argmax index sets of the functions $g$ and $h$.
Specifically, given $\varepsilon > 0$ and $\theta\in \Theta$, let
$$
\left\{\begin{array}{ll}
\mathcal{A}_{1;\varepsilon}(\theta) \, \triangleq \, \varepsilon{\text -}\displaystyle{
\operatornamewithlimits{\mbox{argmax}}_{1\leq i\leq k_1}
} \, g (\theta) =  \left\{ \, 1 \leq i \leq k_1 \, \mid \, \psi_{1,i}(\theta) \, \geq \, g(\theta) - \varepsilon \, \right\}, \\ [0.15in]
\mathcal{A}_{2;\varepsilon}(\theta) \, \triangleq \, \varepsilon{\text -}\displaystyle{
\operatornamewithlimits{\mbox{argmax}}_{1\leq i\leq k_2}
} \, h (\theta) =  \left\{ \, 1 \leq i \leq k_2 \, \mid \, \psi_{2,i}(\theta) \, \geq \, h(\theta) - \varepsilon \, \right\}.
\end{array}\right.
$$

To facilitate the solution of the MM subproblem, we further introduce auxiliary variables $r$ and $s$ in $\mathbb{R}$
to write the max functions in $g$ and $h$ as constraints, thus smoothing out the arguments in the functions
$\varphi^{\, \uparrow}$ and $\varphi^{\, \downarrow}$, and then regularize the added variables in the objective function.
Specifically, for any $\bar{\theta}\in \Theta$ and any $(i_1, i_2)\in \mathcal{A}_1( \bar{\theta} )\times \mathcal{A}_2( \bar{\theta} )$, define the convex set
$$
{\small \begin{array}{ll}
{\cal Z}_{(i_1,i_2)}( \bar{\theta} )  \triangleq  \left\{ \begin{array}{ll}
z \, \triangleq \, (\theta, r,s) \\ [5pt]
\in  \Theta\times \mathbb{R}\times \mathbb{R}
\end{array}
 \bigg|  \begin{array}{ll}
 \psi_{1,j}(\theta) - h( \bar{\theta} ) - \nabla \psi_{2,i_2}( \bar{\theta} )^T \left( \, \theta - \bar{\theta} \, \right)
\, \leq \, r,\; 1\leq j\leq k_1 \\ [5pt]
 g( \bar{\theta} ) + \nabla \psi_{1, i_1}( \bar{\theta} )^T \left( \, \theta - \bar{\theta} \, \right) - \psi_{2,j}(\theta) \, \geq \,s,
\; 1\leq j\leq k_2
\end{array}
\right\},
\end{array}}
$$
which must contain the triple $( \bar{\theta},r,s )$ for any pair of scalars $(r,s)$ satisfying $r \geq \psi(\bar{\theta}) \geq s$.
Since, $\varphi^{\, \uparrow}$ is non-decreasing and $\varphi^{\, \downarrow}$ is non-increasing, we have
\begin{equation} \label{eq:rs inequality}
\varphi^{\, \uparrow}(r) + \varphi^{\, \downarrow}(s) \, \geq \,
{\cal M}\Psi_{(i_1,i_2)}(\theta,\bar{\theta}) \, \geq \,
\Psi(\theta), \epc \mbox{for all $( \theta,r,s ) \in {\cal Z}_{(i_1,i_2)}( \bar{\theta} )$}.
\end{equation}
Moreover, equalities hold throughout (\ref{eq:rs inequality}) if $r = s = \psi(\theta)$.

With the above preparations, the MM-based algorithm is given below.

\noindent\makebox[\linewidth]{\rule{\textwidth}{1pt}}

\noindent The Nonmonotone Majorization-Minimization algorithm for solving \eqref{eq:optimization model2}.

\noindent\makebox[\linewidth]{\rule{\textwidth}{1pt}}

\noindent
{\bf Initialization.}  Given are positive scalars $c$ and $\varepsilon$ and an initial point $\theta^{\, 0} \in \Theta$.
Let $z^{\, 0} \triangleq ( \theta^{\, 0},r_0,s_0 )$ where $r_0 \triangleq \psi( \theta^{\, 0} ) \triangleq s_0$.
Set $\nu = 0$.\\[0.05in]
{\bf Step 1.}  For {\sl every} pair $(i_1^\nu,i_2^\nu)\in \mathcal{A}_{1;\varepsilon}(\theta^{\,\nu})\times \mathcal{A}_{2;\varepsilon}(\theta^{\,\nu})$, compute  \begin{equation}\label{MM:subproblem}
z^{\,\nu+\frac{1}{2};i_1^\nu,i_2^\nu} \, \triangleq \, \displaystyle{
\operatornamewithlimits{\mbox{argmin}}_{z \in {\cal Z}_{(i_1^\nu,i_2^\nu)}(\theta^{\, \nu})}
} \, \left\{\,\wh{\Psi}_c
(z,z^{\, \nu})\, \triangleq \,
\varphi^{\, \uparrow}(r) + \varphi^{\, \downarrow}(s) +
\displaystyle{
\frac{c}{2}
} \, \| \, z - z^{\, \nu} \, \|^2\,\right\}.
\end{equation}
\vskip -0.5em
\noindent
{\bf Step 2.} Set $z^{\,\nu+1} \triangleq  z^{\,\nu+\frac{1}{2};\wh{i}_1^{\,\nu},\wh{i}_2^{\,\nu}}$, where $(\,\wh{i}_1^{\,\nu}\,,\, \wh{i}_2^{\,\nu})$ is a minimizing index in
\[
\displaystyle{
\operatornamewithlimits{\mbox{argmin}}
} \left\{ \, \wh{\Psi}_c
(z^{{\nu+\frac{1}{2}; i_1^{\,\nu}, i_2^{\,\nu}}}, z^{\, \nu})\, \big| \,
(i_1^{\,\nu}, i_2^{\,\nu}) \, \in \, \mathcal{A}_{1;\varepsilon}(\theta^{\,\nu})\times \mathcal{A}_{2;\varepsilon}(\theta^{\,\nu}) \, \right\}.
\]
\vskip -1em
\noindent
{\bf Step 3.}   If $z^{\,\nu+1}$ satisfies a prescribed stopping rule, terminate; otherwise, return to Step 1 with $\nu$ replaced by $\nu+1$.
\hfill $\Box$

\noindent\makebox[\linewidth]{\rule{\textwidth}{1pt}}
\vskip 0.2cm

Because of the regularization of the variables $r$ and $s$, subproblem (\ref{MM:subproblem}) is a slight modification of the MM subproblem which would be:
\begin{equation} \label{eq:straightforward MM}
\displaystyle{
\operatornamewithlimits{\mbox{minimize}}_{z \in {\cal Z}_{(i_1^{\nu},i_2^{\nu})}(\theta^{\, \nu})}
} \ \wt{\Psi}_c
(z,\theta^{\, \nu}) \, \triangleq \,
\varphi^{\, \uparrow}(r) + \varphi^{\, \downarrow}(s) +
\displaystyle{
\frac{c}{2}
} \, \| \, \theta - \theta^{\, \nu} \, \|^2.
\end{equation}
The modification is needed when we discuss the solution of \eqref{MM:subproblem} that is based on a smoothness property of its Lagrangian dual.  Unlike (\ref{eq:straightforward MM}), a minimizing triple $( \theta,r,s )$ of (\ref{MM:subproblem})
does not necessarily satisfy
$r = g(\theta) - \left[ \, h(\theta^{\, \nu}) + \nabla \psi_{2,i_2^{\nu}}(\theta^{\, \nu})^T \left( \, \theta - \theta^{\, \nu} \, \right) \, \right]$ or
$s = \left[ \, g(\theta^{\, \nu}) + \nabla \psi_{1,i_1^{\nu}}(\theta^{\, \nu})^T \left( \, \theta - \theta^{\, \nu} \, \right) \, \right] - h(\theta)$ because of
the regularization of these variables.  In addition, the feasible set ${\cal Z}_{(i_1^{\nu},i_2^{\nu})}(\theta^{\, \nu})$
changes with the iterate $\theta^{\, \nu}$.  Due to all these anomalies, care is needed in the proof of convergence of the algorithm. In particular, the sequence $\{ \Psi(\theta^{\, \nu}) \}$ of objective values of the original function $\Psi$ to be minimized is not shown
to be decreasing; instead it is the substituted sequence $\left\{ \varphi^{\, \uparrow}( r_{\nu+1} ) + \varphi^{\, \downarrow}( s_{\nu+1} ) \right\}$
that is decreasing.   The term ``nonmonotone'' is employed to highlight this non-standard feature of the algorithm. We first give a necessary and sufficient condition for a vector $\bar{\theta}$ to be an optimal solution of the problem: $\displaystyle{
\operatornamewithlimits{\mbox{minimize}}_{\theta \in \Theta}
} \, {\cal M}\Psi_{(i_1,i_2)}(\theta,\bar{\theta})$, this being a key requirement in the  d-stationarity
of $\bar{\theta}$.

\begin{lemma} \label{lm:regularized M}
A vector  $\bar{\theta} \in \displaystyle{
\operatornamewithlimits{\mbox{argmin}}_{\theta \in \Theta}
} \, {\cal M}\Psi_{(i_1,i_2)}(\theta,\bar{\theta})$, where $( i_1,i_2 ) \in {\cal A}_1(\bar{\theta}) \times {\cal A}_2(\bar{\theta})$,
if and only if there exist a scalar $c > 0$ and a pair of scalars $( \bar{r},\bar{s} )$
such that the triple $\bar{z} \triangleq ( \bar{\theta},\bar{r},\bar{s} ) \in \displaystyle{
\operatornamewithlimits{\mbox{argmin}}_{z \in {\cal Z}_{(i_1,i_2)}(\bar{\theta})}
} \, \wh{\Psi}_c
(z,\bar{z})$.
\end{lemma}

\begin{proof}
 ``Only if.''  Suppose  $\bar{\theta} \in \displaystyle{
\operatornamewithlimits{\mbox{argmin}}_{\theta \in \Theta}
} \, {\cal M}\Psi_{(i_1, i_2)}(\theta,\bar{\theta})$.  Let $c > 0$ be arbitrary.  It then follows that
$\bar{\theta} \in \displaystyle{
\operatornamewithlimits{\mbox{argmin}}_{\theta \in \Theta}
} \, \left[ \, {\cal M}\Psi_{(i_1, i_2)}(\theta,\bar{\theta}) + \displaystyle{
\frac{c}{2}
} \, \| \, \theta - \bar{\theta} \, \|^2 \, \right]$.
Define
$\bar{r} \triangleq \psi(\bar{\theta}) \triangleq \bar{s}$.  Then $\bar{z} \triangleq ( \bar{\theta},\bar{r},\bar{s} ) \in {\cal Z}_{(i_1,i_2)}(\bar{\theta})$.
Let $z \triangleq ( \theta,r,s ) \in {\cal Z}_{(i_1,i_2)}(\bar{\theta})$ be arbitrary.   we have
 $$ \begin{array}{lll}
\wh{\Psi}_c
(z,\bar{z}) & = &
\varphi^{\, \uparrow}(r) + \varphi^{\, \downarrow}(s) + \displaystyle{
\frac{c}{2}
} \, \| \, z - \bar{z} \, \|^2 \\[0.1in]
& \geq & {\cal M}\Psi_{(i_1,i_2)}(\theta,\bar{\theta}) + \displaystyle{
\frac{c}{2}
} \, \| \, \theta - \bar{\theta} \, \|^2 \epc \mbox{by (\ref{eq:rs inequality})} \\ [0.1in]
& \geq & {\cal M}\Psi_{(i_1,i_2)}(\bar{\theta},\bar{\theta}) \epc \mbox{by assumption on $\bar{\theta}$} \\ [0.1in]
& = & \varphi^{\, \uparrow}( \bar{r} ) + \varphi^{\, \downarrow}( \bar{s} )
 =  \wh{\Psi}_c
 (\bar{z},\bar{z}) \epc \mbox{since $\bar{r} = \psi(\bar{\theta}) = \bar{s}$}.
\end{array}
$$
``If.''  Conversely, suppose $\bar{z} \triangleq ( \bar{\theta},\bar{r},\bar{s} ) \in \displaystyle{
\operatornamewithlimits{\mbox{argmin}}_{z \in {\cal Z}_{(i_1,i_2)}(\bar{\theta})}
} \, \wh{\Psi}_c
(z,\bar{z})$ for some scalar $c > 0$.  It follows that
$\bar{z} \triangleq ( \bar{\theta},\bar{r},\bar{s} ) \in \displaystyle{
\operatornamewithlimits{\mbox{argmin}}_{z \in {\cal Z}_{(i_1,i_2)}(\bar{\theta})}
} \, \left[ \, \varphi^{\, \uparrow}(r) + \varphi^{\, \downarrow}(s) \, \right]$, which implies in particular that
$\bar{r} \geq \psi(\bar{\theta}) \geq \bar{s}$.  Hence
$( \bar{\theta},\psi(\bar{\theta}),\psi(\bar{\theta}) ) \in \displaystyle{
\operatornamewithlimits{\mbox{argmin}}_{z \in {\cal Z}_{(i_1,i_2)}(\bar{\theta})}
} \, \left[ \, \varphi^{\, \uparrow}(r) + \varphi^{\, \downarrow}(s) \, \right]$ and
$\varphi^{\, \uparrow}( \bar{r} ) + \varphi^{\, \downarrow}(\bar{s}) = \Psi(\bar{\theta}) = {\cal M}\Psi_{(i_1,i_2)}(\bar{\theta},\bar{\theta})$.
Let $\theta \in \Theta$ be arbitrary.  Define
\[
r \, \triangleq \, g(\theta) - \left[ \, h(\bar{\theta}) + \nabla \psi_{2,i_2}(\bar{\theta})^T \left( \, \theta - \bar{\theta} \, \right) \, \right],\epc
s \, \triangleq \, \left[ \, g(\bar{\theta}) + \nabla \psi_{1,i_1}(\bar{\theta})^T \left( \, \theta - \bar{\theta} \, \right) \, \right] - h(\theta).
\]
Then $( \theta,r,s ) \in {\cal Z}_{(i_1,i_2)}(\bar{\theta})$ and
$\varphi^{\, \uparrow}(r) + \varphi^{\, \downarrow}(s) = {\cal M}\Psi_{(i_1,i_2)}(\theta,\bar{\theta})$.
It therefore follows that ${\cal M}\Psi_{(i_1,i_2)}(\theta,\bar{\theta}) \geq {\cal M}\Psi_{(i_1,i_2)}(\bar{\theta},\bar{\theta})$,
establishing  that $\bar{\theta} \in \displaystyle{
\operatornamewithlimits{\mbox{argmin}}_{\theta \in \Theta}
} \, {\cal M}\Psi_{(i_1, i_2)}(\theta,\bar{\theta})$.
\end{proof}

Now we present the main result of this section on the subsequential convergence of $\{\theta^{\,\nu}\}$  to a d-stationary point
of~\eqref{eq:optimization model2}, which generalizes the result in \cite[Proposition 6]{PangRazaviyaynAlvarado16} for the dc algorithm to solve dc programs.

\begin{proposition}\label{proposition: d-stationary}
Suppose that $\Psi$ is bounded below on the closed convex set $\Theta$.
Let $\left\{ z^{\, \nu} = ( \theta^{\, \nu}, r_\nu, s_\nu) \right\}$ be a sequence generated by the MM algorithm.
The following four statements hold.\\
{\rm (a)} For any
$(i_1^{\,\nu}, i_2^{\,\nu})\in \mathcal{A}_{1;\varepsilon}(\theta^{\,\nu})\times \mathcal{A}_{2;\varepsilon}(\theta^{\,\nu}) $ and
$z \in {\cal Z}_{(i_1^{\,\nu},i_2^{\,\nu})}(\theta^{\, \nu})$,    it holds that
$$\varphi^{\,\uparrow}(r_{\nu+1})+\varphi^{\,\downarrow}(s_{\nu+1})+ \displaystyle\frac{c}{2}\,\|\,z^{\,\nu+1} - z^{\,\nu}\,\|^2 \, \leq \,\varphi^{\, \uparrow}(r) +
\varphi^{\, \downarrow}(s) + \, \displaystyle{
\frac{c}{2}
} \, \| \, z - z^{\, \nu} \, \|^2.$$\\[-0.2in]
{\rm (b)} $\displaystyle\lim_{\nu\to\infty}\|\,z^{\,\nu+1} - z^{\,\nu}\,\|\,=\,0$.\\[0.05in]
{\rm (c)} For any accumulation point $( \theta^{\, \infty},r_{\infty},s_{\infty} )$ of $\{z^{\,\nu}\}$, if it exists,
$\theta^{\, \infty}$ is a d-stationary point of \eqref{eq:optimization model2}.\\[0.05in]
{\rm (d)} If the set $\{ \theta \in \Theta \mid \Psi(\theta) \leq \Psi(\theta^{\, 0}) \}$ is bounded, then the
sequence $\{ \theta^{\, \nu} \}$ is bounded. If in addition, the set $\{ t \mid \varphi(t) \leq \Psi(\theta^{\, 0}) \}$
is bounded, then the sequence $\{ ( r_{\nu},s_{\nu} ) \}$ is also bounded.
\end{proposition}

\begin{proof}
(a) For any $(i_1^{\,\nu}, i_2^{\,\nu})\in \mathcal{A}_{1;\varepsilon}(\theta^{\,\nu})\times \mathcal{A}_{2;\varepsilon}(\theta^{\,\nu}) $ and $z \in {\cal Z}_{(i_1^{\,\nu},i_2^{\,\nu})}(\theta^{\, \nu})$, we derive
$$\begin{array}{rl}
&\varphi^{\, \uparrow}(r) +
\varphi^{\, \downarrow}(s) + \, \displaystyle{
\frac{c}{2}
} \, \| \, z - z^{\, \nu} \, \|^2\\[0.1in]
  \geq & \, \varphi^{\,\uparrow}\left(r_{\nu+\frac{1}{2}; i_1^{\,\nu}, i_2^{\,\nu}}\right)\,+ \,\varphi^{\,\downarrow}\left( s_{\nu+\frac{1}{2}; i_1^{\,\nu}, i_2^{\,\nu}}\right)   +\,\displaystyle\frac{c}{2}\,\|\,z^{\,\nu+\frac{1}{2}; i_1^{\,\nu}, i_2^{\,\nu}} - z^{\,\nu}\,\|^2 \\[0.1in]

\epc \geq &\, \varphi^{\,\uparrow}(r_{\nu+1})+\varphi^{\,\downarrow}(s_{\nu+1}) + \displaystyle\frac{c}{2}\,\|\,z^{\,\nu+1} - z^{\,\nu}\,\|^2,
\end{array}
$$
\vskip -0.3cm
\noindent
where the two inequalities are due to the definitions of $z^{\,\nu+\frac{1}{2};\,i_1^{\,\nu},\,i_2^{\,\nu}}$ and  ${z}^{\,\nu+1}$.\\
(b) By the strong convexity of $\wh{\Psi}_c
(\bullet,z^{\, \nu})$,
it follows that $z^{\,\nu+\frac{1}{2};i_1^\nu, i_2^\nu}$ is the unique minimizer of (\ref{MM:subproblem}).  We claim
that $z^{\,\nu} \in {\cal Z}_{(i_1,i_2)}(\theta^{\, \nu})$
for all $\nu$ and all $(i_1,i_2)\in \{1, \ldots,k_1\}\times \{1,\ldots, k_2\}$.  Indeed, this is clearly true for $\nu = 0$ by the choice of $( r_0,s_0 )$.
For any $\nu \geq 0$, since
$z^{\, \nu+1} \in {\cal Z}_{(i_1^{\nu},i_2^{\nu})}(\theta^{\, \nu})$ for some
$(i_1^{\nu}, i_2^{\nu})\in \mathcal{A}_{1;\varepsilon}(\theta^{\,\nu})\times \mathcal{A}_{2;\varepsilon}(\theta^{\,\nu})$, we have
\[ \begin{array}{lll}
r_{\nu+1} & \geq & g(\theta^{\, \nu+1}) - \left[ \, h(\theta^{\, \nu}) + \nabla \psi_{2,i_2^{\nu}}(\theta^{\, \nu})^T
\left( \, \theta^{\, \nu+1} - \theta^{\, \nu} \, \right) \, \right] \\ [0.1in]
& \geq & g(\theta^{\, \nu+1}) - \psi_{2,i_2^{\nu}}(\theta^{\, \nu+1})
\, \geq \, g(\theta^{\, \nu+1}) - h(\theta^{\, \nu+1}) \, = \, \psi(\theta^{\, \nu+1}).
\end{array} \]
Similarly, we can deduce $s_{\nu+1} \leq \psi(\theta^{\, \nu+1})$.  Consequently, the claim holds.
It then follows from (a) that
\begin{equation}\label{ineq:decreaseMM}
\varphi^{\,\uparrow}(r_{\nu+1})+\varphi^{\,\downarrow}(s_{\nu+1})+ \displaystyle\frac{c}{2}\,\|\,z^{\,\nu+1} - z^{\,\nu}\,\|^2\, \leq\, \varphi^{\, \uparrow}(r_\nu) +
\varphi^{\, \downarrow}(s_\nu), \epc \forall \, \nu.
\end{equation}
Thus the sequence $\left\{ \varphi^{\, \uparrow}(r_{\nu}) + \varphi^{\, \downarrow}(s_{\nu}) \right\}$ is non-increasing.  Since
$\varphi^{\, \uparrow}(r_{\nu}) + \varphi^{\, \downarrow}(s_{\nu}) \geq \Psi(\theta^{\, \nu})$ and $\Psi$ is bounded below on $\Theta$,
it follows that the sequence $\left\{ \varphi^{\, \uparrow}(r_{\nu}) + \varphi^{\, \downarrow}(s_{\nu}) \right\}$ converges.
Hence by (\ref{ineq:decreaseMM}), the sequence $\left\{ \| z^{\nu+1} - z^{\nu} \| \right\}$ converges to zero.\\
(c) Let $z^{\, \infty} \triangleq ( \theta^{\, \infty},r_{\infty},s_{\infty} )$
be the limit of a convergent subsequence $\left\{ z^{\, \nu} \right\}_{\nu \in \kappa}$.   We must have
$\theta^{\, \infty} \in \Theta$.
To prove that $\theta^{\,\infty}$ is a d-stationary point, it suffices to show, by Lemma \ref{lemma: d-stationary} and Lemma \ref{lm:regularized M}, that
$z^{\, \infty}$ belongs to $\displaystyle{
\operatornamewithlimits{\mbox{argmin}}_{z \in {\cal Z}_{({i}_1,{i}_2)}(\theta^{\, \infty})}
} \, \wh{\Psi}_c
(z,z^{\, \infty})$ for all $(i_1, i_2) \in \mathcal{A}_1(\theta^\infty)\times \mathcal{A}_2(\theta^\infty)$.
Consider any $(i_1, i_2)\in \mathcal{A}_1(\theta^{\,\infty})\times \mathcal{A}_2(\theta^{\,\infty})$ and any $z\in {\cal Z}_{({i}_1,{i}_2)}(\theta^{\, \infty})$.
Define
{\small $$\begin{array}{lll}
\wh{r}_{\nu}  \triangleq  \,r + \left| \, \left[ \, h(\theta^{\, \nu}) + \nabla \psi_{2,{i}_2}(\theta^{\, \nu})^T \left( \, \theta - \theta^{\, \nu} \, \right) \, \right] -
\left[ \, h(\theta^{\, \infty}) + \nabla \psi_{2,{i}_2}(\theta^{\, \infty})^T \left( \, \theta - \theta^{\, \infty} \, \right) \, \right] \right| \\ [0.1in]
\wh{s}_{\nu}  \triangleq \, s + \left| \, \left[ \, g(\theta^{\, \nu}) + \nabla \psi_{1,{i}_1}(\theta^{\, \nu})^T \left( \, \theta - \theta^{\, \nu} \, \right) \, \right] -
\left[ \, g(\theta^{\, \infty}) + \nabla \psi_{1,{i}_1}(\theta^{\, \infty})^T \left( \, \theta - \theta^{\, \infty} \, \right) \, \right] \right| .
\end{array}
$$}
We then have $\wh{z}^{\, \nu} \triangleq ( \theta,\wh{r}_{\nu},\wh{s}_{\nu} ) \in {\cal Z}_{({i}_1,{i}_2)}(\theta^{\, \nu})$ and
$\displaystyle{
\lim_{\nu (\in \kappa) \to \infty}
} \, ( \, \wh{r}_{\nu},\wh{s}_{\nu} \, ) = ( r,s )$.
One can easily show that for all $\nu\in \kappa$ sufficiently large,
$\mathcal{A}_1(\theta^{\,\infty})\times \mathcal{A}_2(\theta^{\,\infty}) \subseteq \mathcal{A}_{1;\varepsilon}(\theta^{\,\nu})
\times \mathcal{A}_{2;\varepsilon}(\theta^{\,\nu})$, which in turn implies $(i_1, i_2) \in
\mathcal{A}_{1;\varepsilon}(\theta^{\,\nu})\times \mathcal{A}_{2;\varepsilon}(\theta^{\,\nu})$ for all such $\nu$.
It therefore follows from (a) that
\[
\varphi^{\, \uparrow}( r_{\nu+1} ) + \varphi^{\, \downarrow}( s_{\nu+1} ) + \displaystyle{
\frac{c}{2}
} \, \| \, z^{\, \nu+1} - z^{\, \nu} \, \|^2 \, \leq \, \varphi^{\, \uparrow}( \wh{r}_{\nu} ) + \varphi^{\, \downarrow}( \wh{s}_{\nu} ) + \displaystyle{
\frac{c}{2}
} \, \| \, \wh{z}^{\, \nu} - z^{\, \nu} \, \|^2 .
\]
Passing to the limit $\nu (\in \kappa) \to \infty$, we obtain $z^{\, \infty} \in \displaystyle{
\operatornamewithlimits{\mbox{argmin}}_{z \in {\cal Z}_{({i}_1,{i}_2)}(\theta^{\, \infty})}
} \, \wh{\Psi}_c
(z,z^{\, \infty})$.
This completes the proof of statement (c).\\
(d) Since $\Psi(\theta^{\, \nu}) \leq \varphi^{\, \uparrow}(r_{\nu}) + \varphi^{\, \downarrow}(s_{\nu}) \leq
\varphi^{\, \uparrow}(r_0) + \varphi^{\, \downarrow}(s_0) = \Psi(\theta^{\, 0})$
for all $\nu$, it follows that the sequence $\{ \theta^{\, \nu} \}$ is bounded by assumption.  Since $r_{\nu} \geq s_{\nu}$,
it follows that
\[
\varphi^{\, \uparrow}(r_{\nu}) + \varphi^{\, \downarrow}(s_{\nu}) \, \geq \, \max\left( \, \varphi(s_{\nu}),\varphi(r_{\nu}) \, \right),
\]
the boundedness of $\{ ( r_{\nu},s_{\nu} ) \}$ follows similarly.
\end{proof}
\gap

One may further establish the convergence of the full sequence $\{z^{\,\nu}\}$ under an isolatedness assumption of an accumulation point
as in the general theory of sequential convergence \cite[Proposition 8.3.10]{FacchineiPang2003} or by invoking the
Kurdyka-{\L}ojaziewicz (KL) theory of semi-analytic functions as in \cite{AttouchBolte2009, AttouchBolteSvaiter2013, BoltePauwels2016, BolteST14},
Recall that
a proper lower semi-continuous function  $f:\mathbb{R}^n\to (-\infty, +\infty]$
is said to have the KL property at $\bar{x}\in {\rm dom}\,(\partial\, f) \triangleq \{\,x \in {\rm dom}\,f \,\mid \, \partial\, f(x)\neq  \emptyset\,\}$ if
there exist $\alpha>0$, a neighborhood $\mathcal{N}$ of $\bar{x}$ and a continuous concave function $\phi: [\,0\,,\, \alpha\,)\to \mathbb{R}_+$ such that\\
$\bullet $ $\phi(0) = 0$ and  $\phi$ is continuously differentiable on $(0, \alpha)$ with $\phi^{\, \prime} > 0$; and\\
$\bullet $ $\phi^{\prime}(f(x)- f(\bar{x})) \, {\rm dist}\,(0\,,\, \partial f(x)) \geq  1$ for any $x\in \mathcal{N}$ with $f(\bar{x})< f(x)<f(\bar{x})+\alpha$,
where $\mbox{dist}\,(x,C) \,\triangleq \displaystyle\operatornamewithlimits{minimize}_{y\in C} \,\|x - y\|$ is
the distance from a point $x\in \mathbb{R}^n$ to a nonempty closed set $C\subseteq \mathbb{R}^n$.
If the function $\phi$ in the above definition can be chosen as $\phi(s) = \gamma \,s^{1-\beta}$
 for some scalars $\gamma > 0$ and $\beta \in  [0, 1)$, 
 we say that $f$ has the KL property at $\bar{x}$ with an exponent $\beta$.
It is known that a proper closed semi-algebraic function is a KL function \cite{BolteDaniilidisLewis2006}.

To proceed, we introduce the following function of $z \,\triangleq\, (\theta, r,s)$:
$$
\overline\Psi(z) \,\triangleq\,
\left\{\begin{array}{ll}
\varphi^{\,\uparrow}(r) + \varphi^{\,\downarrow}(s), & \mbox{if}\; z\in \overline{Z}\,\triangleq\,\{(\theta, r, s)\in \Theta\times \mathbb{R}\times \mathbb{R}\mid
s\leq \psi({\theta})\leq r\}, \\[0.05in]
+\infty & \mbox{otherwise}
\end{array}\right.
$$
and also define the set
$Z^{\,\infty} \, \triangleq \, \Theta^{\,\infty} \, \times \, R_{\,\infty} \, \times \, S_{\,\infty}$
of all accumulation points of the sequence $\{ z^{\,\nu} \}$ produced by the nonmonotone MM algorithm.  This set is nonempty under condition (d)
of Proposition~\ref{proposition: d-stationary}; moreover, every one of its elements is a d-stationary solution of
problem~\eqref{eq:optimization model2}.

\begin{proposition}\label{proposition: full convergence}
Suppose that $\Psi$ is bounded below on the closed convex set $\Theta$.
Then the whole sequence $\{z^{\,\nu}\}$ converges to an unique element of the set $Z^{\, \infty}$
under either one of the following two conditions (a) or (b);\\[0.05in]
(a) $Z^{\,\infty}$ contains an isolated element;\\[0.05in]
(b) $\{z^{\,\nu}\}$ is bounded; the index sets
$\mathcal{A}_{1;2\varepsilon}(\theta)$ and $\mathcal{A}_{2;2\varepsilon}(\theta)$ are singletons for all $\theta\in \Theta^{\,\infty}$;
$\nabla \psi_{1,i_1}$ and $\nabla \psi_{2,i_2}$ are locally Lipschitz continuous near all $\theta\in \Theta^{\,\infty}$ for
all $(i_1,i_2)\in \mathcal{A}_{1;2\varepsilon}(\theta)\times \mathcal{A}_{2;2\varepsilon}(\theta)$, respectively;
the functions $\varphi^{\,\uparrow}$ and $\varphi^{\,\downarrow}$ are differentiable with Lipschitz continuous
gradients near all $r \in R_{\,\infty}$ and $s \in S_{\,\infty}$, respectively; moreover,
the function $\overline\Psi$ satisfies the KL property at all $z^{\,\infty}\in Z^{\,\infty}$.\\[0.05in]
If $\{z^k\}$ converges to $z^{\,\infty}\in Z^{\,\infty}$ under condition  (b) and the function $\overline\Psi$ has the KL exponent $\beta\in [0,1)$
at $z^{\,\infty}$, it holds that\\[0.05in]
$\bullet $ if $\beta = 0$, the sequence $\{z^{\,\nu}\}$ converges in a finite number of steps;\\[0.05in]
$\bullet $ if $\beta \in (0\,,\,\frac{1}{2}\,]$, the sequence $\{z^{\,\nu}\}$ converges R-linearly, i.e.,
there exist $\eta>0$  and $r \in [0, 1)$ such that $\|z^{\,\nu} - z^{\,\infty}\|\leq \eta\, r^{\,\nu}$ for all $\nu$ sufficiently large;\\[0.05in]
$\bullet $ if $\beta \in (\frac{1}{2}\,,1)$, the sequence $\{z^{\,\nu}\}$ converges R-sublinearly,
specifically, there exists $\eta > 0$ such that $\|z^{\,\nu} - z^{\,\infty}\|\leq \eta\, \nu^{-\frac{1-\beta}{2\beta-1}}$ for all $\nu$ sufficiently large.
\end{proposition}

\begin{proof}
It follows from Proposition \ref{proposition: d-stationary}(b) that $\displaystyle\lim_{\nu\to\infty}\|\,z^{\,\nu+1} - z^{\,\nu}\,\| = 0$.
Under assumption (a) here, the convergence of the whole sequence $\{z^{\,\nu}\}$ to the isolated element of $Z^{\, \infty}$ follows
from \cite[Proposition~8.3.10]{FacchineiPang2003}.  To prove the sequential convergence of $\{z^{\,\nu}\}$ under the conditions in assumption~(b),
it suffices to show the following three properties of $\{z^{\,\nu}\}$ and then apply the convergence results in
\cite[Theorem 2.9]{AttouchBolteSvaiter2013} to the function $\overline\Psi$:
\\[0.05in]
(i) $\overline\Psi(z^{\,\nu+1})\leq \overline\Psi(z^{\,\nu}) - \displaystyle\frac{c}{2}\,\|\,z^{\,\nu+1} - z^{\,\nu}\,\|^2$ for all $\nu\geq 1$\,; \\[0.05in]
(ii) there exists a subsequence $\{z^{\,\nu}\}_{\nu\in \kappa}$ of $\{z^{\,\nu}\}$ such that $\displaystyle\lim_{\nu(\in \kappa)\to \infty}z^{\,\nu}  \, =\,z^{\,\infty}$ and
$\displaystyle\lim_{\nu(\in \kappa)\to \infty}\overline\Psi(z^{\,\nu}) = \overline\Psi(z^{\,\infty})$\,;\\[0.05in]
(iii)  there exists a scalar $\eta>0$ such that for all $\nu$ sufficiently large, $a^{\,\nu+1}\in \partial \, \overline\Psi(z^{\,\nu+1})$ exists satisfying
$\|\,a^{\,\nu+1}\,\| \leq  \eta\,\|\,z^{\,\nu+1} - z^{\,\nu}\,\| $.

Properties (i) and (ii) readily follow from Proposition~\ref{proposition: d-stationary}.
To show (iii), we first note 
that $Z^{\,\infty}$ is a nonempty, compact, and connected set \cite[Proposition 8.3.9]{FacchineiPang2003}.
By assumption (b), we further derive the existence of $\bar{i}_j\in \{1, \ldots, k_j\}$ such
that $\mathcal{A}_{j;2\varepsilon}(\theta) = \{\bar{i}_j\}$ for all $\theta\,\in\, \Theta^{\,\infty}$, $j=1,2$.
Since the scalar sequence $\{\mbox{dist}\,(\theta^{\,\nu}, \Theta^{\,\infty})\}$ is bounded with a unique accumulation point $0$,
we have $\displaystyle\lim_{\nu\to \infty}\mbox{dist}\,(\theta^{\,\nu}, \Theta^{\,\infty}) = 0$.
Hence, for some constant $\kappa>0$ it holds that for all $\nu$ sufficiently large, some $\theta^{\,\nu;\infty}\in \Theta^{\,\infty}$ exists
for which $\|\,\theta^{\,\nu}-\theta^{\,\nu;\infty}\,\|\,\leq\, \displaystyle\frac{\varepsilon}{2\kappa}$ and
$|\,\psi_{1,{i}}(\theta^{\,\nu}) - \psi_{1,{i}}(\theta^{\,\nu;\infty})\,| \,\leq \, \kappa\,\|\,\theta^{\,\nu}-\theta^{\,\nu;\infty}\,\|\leq \,
\displaystyle\frac{1}{2}\,\varepsilon$ for any $i=1,\ldots, k_1$. Then for any ${i}\in \mathcal{A}_{1;\varepsilon}(\theta^{\,\nu})$,
$$\begin{array}{ll}
\psi_{1,{i}}(\theta^{\,\nu;\infty}) \,\geq \, \psi_{1, {i}}(\theta^{\,\nu}) -\displaystyle \frac{1}{2}\,\varepsilon\,\geq\,
\displaystyle{\max_{1\leq i\leq k_1}}\psi_{1, i}(\theta^{\,\nu}) - \frac{3}{2}\,\varepsilon\epc \mbox{since $i\in \mathcal{A}_{1;\varepsilon}(\theta^{\,\nu})$}\\[0.1in]
\geq\, \displaystyle\psi_{1, \bar{i}_1}(\theta^{\,\nu}) - \frac{3}{2}\,\varepsilon  \,\geq\, \psi_{1, \bar{i}_1}(\theta^{\,\nu;\infty})-2\,\varepsilon\\[0.1in]
\qquad\qquad\quad=\, \displaystyle\max_{1\leq i\leq k_1}\psi_{1, i}(\theta^{\,\nu;\infty})-2\,\varepsilon ,
\epc \mbox{since $ \mathcal{A}_{1;2\varepsilon}(\theta^{\,\nu;\infty})
\, = \, \{\bar{i}_1\}=\mathcal{A}_{1}(\theta^{\,\nu;\infty})$},
\end{array}
$$
i.e., ${i}\,\in\, \mathcal{A}_{1;2\varepsilon}(\theta^{\,\nu;\infty}) \,= \,\{\bar{i}_1\}$. Therefore, for all  $\nu$ sufficiently large,  $\mathcal{A}_{1;\varepsilon}(\theta^{\,\nu}) = \{\bar{i}_1\}$; and similarly, $\mathcal{A}_{2;\varepsilon}(\theta^{\,\nu}) = \{\bar{i}_2\}$.
By assumption (b) and \cite[Exercise 10.10]{RockafellarRWets98}, we further derive,
 \begin{equation}\label{proof:KL1}
 {\small
\left\{\begin{array}{ll}
\nabla \psi(\theta^{\,\nu+1}) \, = \, \nabla \psi_{1,\bar{i}_1}(\theta^{\,\nu+1}) - \nabla \psi_{2,\bar{i}_2}(\theta^{\,\nu+1})\\[0.1in]
\partial\, \overline\Psi(z^{\,\nu+1})
\, = \, \left(\begin{array}{cc}
0  \\
\nabla \varphi^{\,\uparrow}(r_{\nu+1}) \\
\nabla \varphi^{\,\downarrow}(s_{\nu+1})
\end{array}\right) + \mathcal{N}(z^{\,\nu+1};\, \overline{Z})
\end{array}\right. \mbox{for all  $\nu$ sufficiently large}.}
\end{equation}
Note that $\overline{Z}$ may not be a convex set; we write $\mathcal{N}(z; \,\overline{Z})$ as the normal
cone of $\overline{Z}$ at $z\in \overline{Z}$ based on the definition in \cite[Definition 6.3]{RockafellarRWets98}.
One can easily check that
  {\small  $$
  \begin{array}{ll}
 \left\{\begin{array}{cc}
  \alpha_1 \,\in \,\mathcal{N}(\psi(\theta^{\,\nu+1})-r_{\nu+1};\,\mathbb{R}_-), \; \alpha_2 \,\in\, \mathcal{N}(-\psi(\theta^{\,\nu+1})+s_{\nu+1};\,\mathbb{R}_-),\\[0.1in]
    0\,\in \,\left(\begin{array}{cc}
(\alpha_1-\alpha_2)\nabla \psi(\theta^{\,\nu+1}) +\mathcal{N}(\theta^{\,\nu+1};\Theta) \\ - \alpha_1 \\ \alpha_2
\end{array}\right)
\end{array}\right\}  \Longrightarrow \alpha_1=\alpha_2 = 0,
\end{array}
$$}
where we denote $\mathbb{R}_-\,\triangleq\,\{t\in \mathbb{R}\mid t\leq 0\}$.
It then follows from \cite[Theorem 6.14]{RockafellarRWets98} that
 \begin{equation}\label{proof:KL normalcone}
 {\small
 \begin{array}{ll}
\mathcal{N}(z^{\,\nu+1};\overline{Z})  =
 \left\{
\left(\begin{array}{cc}
(\alpha_1-\alpha_2)\nabla \psi(\theta^{\,\nu+1}) +\mathcal{N}(\theta^{\,\nu+1};\Theta) \\ - \alpha_1 \\ \alpha_2
\end{array}\right)   \middle\vert
\begin{array}{cc}
\alpha_1 \geq 0\,, \, \alpha_2 \geq 0\,, \\[0.1in]
 \alpha_1=0\;\mbox{if}\;\psi(\theta^{\,\nu+1})< r_{\nu+1}\\[0.1in]
\alpha_2 = 0  \;\mbox{if}\;\psi(\theta^{\,\nu+1})> s_{\nu+1}
 \end{array}
 \right\}.
 \end{array}
 }
 \end{equation}
For all  $\nu$ sufficiently large,  since $\mathcal{A}_{1;\varepsilon}(\theta^{\,\nu}) \times \mathcal{A}_{2;\varepsilon}(\theta^{\,\nu}) = \{(\bar{i}_1, \bar{i}_2)\}$, we deduce $z^{\,\nu+1} \,= \,z^{\,\nu+\frac{1}{2};\bar{i}_1, \bar{i}_2}\,= \, \displaystyle{
\operatornamewithlimits{\mbox{argmin}}_{z \in {\cal Z}_{(\bar{i}_1,\bar{i}_2)}(\theta^{\, \nu})}
} \, \wh{\Psi}_c
(z,z^{\, \nu})$. To proceed, we denote
$$
\begin{array}{cc}
I_1(z^{\,\nu+1}) \,\triangleq\, \left\{i:
\psi_{1,i}(\theta^{\,\nu+1}) - \psi_{2,\bar{i}_2}(\theta^{\,\nu}) - \nabla \psi_{2,\bar{i}_2}(\theta^{\,\nu})^T\,(\theta^{\,\nu+1} - \theta^{\,\nu})
\, = \, r_{\nu+1}\,\right\},\\[0.1in]
I_2(z^{\,\nu+1}) \,\triangleq\,  \left\{i: \psi_{1,\bar{i}_1}(\theta^{\,\nu}) +
\nabla \psi_{1,\bar{i}_1}(\theta^{\,\nu})^T\,(\theta^{\,\nu+1} - \theta^{\,\nu}) - \psi_{2,i}(\theta^{\,\nu+1})  \, =\,  s_{\nu+1}\,\right\}.
\end{array}
$$
It then follows from $z^{\,\nu+1}\,\in\, {\cal Z}_{(\bar{i}_1,\bar{i}_2)}(\theta^{\, \nu})$ that  one of the following four cases must hold for any $\nu$ sufficiently large:
(1) $I_1(z^{\,\nu+1}) = I_2 (z^{\,\nu+1})= \emptyset$;
(2) $I_1(z^{\,\nu+1}) = \{\bar{i}_1\}$ and  $I_2(z^{\,\nu+1}) = \emptyset$;
(3) $I_1(z^{\,\nu+1}) = \emptyset$ and $I_2(z^{\,\nu+1}) = \{\bar{i}_2\}$;
(4) $I_1(z^{\,\nu+1}) = \{\bar{i}_1\}$ and $I_2(z^{\,\nu+1}) = \{\bar{i}_2\}$.
For those $z^{\,\nu+1}$ satisfying case (1), we derive
$$\begin{array}{rll}
\psi(\theta^{\,\nu+1}) &= &\psi_{1, \bar{i}_1}(\theta^{\,\nu+1}) - \psi_{2,\bar{i}_2}(\theta^{\,\nu+1})\\[0.1in]
 &\leq &
\psi_{1, \bar{i}_1}(\theta^{\,\nu+1}) - \psi_{2,\bar{i}_2}(\theta^{\,\nu}) - \nabla \psi_{2,\bar{i}_2}(\theta^{\,\nu})^T\,(\theta^{\,\nu+1} - \theta^{\,\nu})  \, <\, r_{\nu+1} .
\end{array}
$$
Similarly one has $\psi(\theta^{\,\nu+1})>s_{\nu+1}$. Then
$
\mathcal{N}(z^{\,\nu+1};\,\overline{Z}) \,=\,
\left(\begin{array}{cc}
\mathcal{N}(\theta^{\,\nu+1};\,\Theta) \\ 0_{\mathbb{R}^2}
\end{array}\right)$.
Since all the inequality constraints in $z \in {\cal Z}_{(\bar{i}_1,\bar{i}_2)}(\theta^{\,\nu})$ are inactive at $z=z^{\,\nu+1}$,
by applying the optimality condition of the problem: $\displaystyle{
\operatornamewithlimits{\mbox{min}}_{z \in {\cal Z}_{(\bar{i}_1,\bar{i}_2)}(\theta^{\, \nu})}
} \, \wh{\Psi}_c
(z,z^{\, \nu})$, we deduce
\[
\begin{array}{cc}
0\in \mathcal{N}(\theta^{\,\nu+1};  \Theta)+ c\,(\theta^{\,\nu+1} - \theta^{\,\nu}),\epc\mbox{and}\\[0.1in]
\nabla\varphi^{\uparrow}(r_{\nu+1}) + c\,(r_{\nu+1} - r_{\nu}) =0,\epc
\nabla\varphi^{\downarrow}(s_{\nu+1})
 + c\,(s_{\nu+1} - s_{\nu})=0,
 \end{array}
\]
which, together with \eqref{proof:KL1}, yields $-\,c\,(z^{\,\nu+1} - z^{\,\nu})\,\in\, \partial\, \overline\Psi(z^{\,\nu+1}).
$
Hence, property (iii) holds with the constant $\eta_1 = c$.
For those $z^{\,\nu+1}$ satisfying
case (2), there exists a multiplier $\lambda_{\nu+1}\in \mathbb{R}$ corresponding to the active constraint $\psi_{1, \bar{i}_1}(\theta^{\,\nu+1}) \,-\,\psi_{2,\bar{i}_2}(\theta^{\,\nu}) - \nabla \psi_{2, \bar{i}_2}(\theta^{\,\nu})^T\left(\, \theta^{\,\nu+1} - \theta^{\,\nu}\,\right) = r_{\nu+1}$
such that
$$
0\,\in\, \left(\begin{array}{cc}
\lambda_{\nu+1} \,\left[\,\nabla \psi_{1,\bar{i}_1} (\theta^{\,\nu+1})- \nabla\psi_{2,\bar{i}_2}(\theta^{\,\nu})\,\right]  +
\mathcal{N}(\theta^{\,\nu+1}; \, \Theta)\\[0.05in]
\nabla\varphi^{\uparrow}(r_{\nu+1}) - \lambda_{\nu+1} \\[0.05in]
\nabla\varphi^{\downarrow}(s_{\nu+1})\end{array}\right) + c\,(z^{\,\nu+1} - z^{\,\nu}).
$$
Since $\varphi^{\uparrow}$ is non-decreasing, we have $\nabla\varphi^{\uparrow}(r_{\nu+1}) = \lambda_{\nu+1} - c\,(r_{\nu+1} - r_\nu)\geq 0$,
which, together with \eqref{proof:KL1} and  \eqref{proof:KL normalcone},  yields
$$
\left(\begin{array}{cc}
-b^{\,\nu+1} \\
\,c\,(r_{\nu+1}-r_{\nu})\\
0
\end{array}\right) - c\,(z^{\,\nu+1} - z^{\,\nu}) \,\in \, \partial\, \overline\Psi(z^{\,\nu+1}),\epc \mbox{where}
$$
{\small $$
 b^{\,\nu+1}\triangleq  c\,(r_{\nu+1} - r_\nu)\left[\nabla \psi_{1, \bar{i}_1}(\theta^{\,\nu+1}) - \nabla \psi_{2, \bar{i}_2}(\theta^{\,\nu})\right]  +
\nabla \varphi^{\uparrow}(r_{\nu+1})\left[\nabla \psi_{2, \bar{i}_2}(\theta^{\,\nu+1}) - \nabla \psi_{2, \bar{i}_2}(\theta^{\,\nu})\right].
$$}
In addition, $\{\nabla \psi_{1, \bar{i}_1}(\theta^{\,\nu+1}) - \nabla \psi_{2, \bar{i}_2}(\theta^{\,\nu})\}$ and $\{\nabla \varphi^{\,\uparrow}(r_{\nu+1})\}$
are bounded for all $\nu$ sufficiently large  under the assumptions in (b).
Hence, property (iii) holds for some constant $\eta_2>0$. Similarly, those points $z^{\,\nu+1}$ satisfying cases (3) and (4) can be proved to have
property (iii) for some positive constants $\eta_3$ and $\eta_4$, respectively.
Consequently, property (iii) holds for all $\nu$ sufficiently large with $\eta = \max\{\eta_1,\ldots, \eta_4\}$.
The desired global convergence of $\{z^{\,\nu}\}$ is thus established based on the above arguments.
Finally,
with the proven properties (i)-(iii), the stated convergence rate can be established by similar arguments
to \cite[Theorem 2]{AttouchBolte2009}; see also \cite[Proposition 4]{BoltePauwels2016}.
\end{proof}

The sequential convergence of the standard MM algorithm (without regularizing the slack variables) has been studied in \cite{BolteST14} under
a condition that is in the same spirit as property (iii) in the above proof.  In order to verify this property, we essentially assume that
$\varphi^{\, \uparrow}$, $\varphi^{\, \downarrow}$, $\psi_1$, and $\psi_2$ are differentiable with locally Lipschitz continuous gradients
near $Z^{\, \infty}$.  This type of assumptions have also been made in \cite{LeHuynhPham09,WenChenPong17,LuZhouSun17}
to prove the sequential convergence of the dc algorithm (with extrapolation) for solving the dc program:
$\displaystyle\operatornamewithlimits{minimize}_x\, f(x) \, -\, g(x)$, where  the function $g$ is assumed to be differentiable
with  a locally Lipschitz gradient near all accumulation points.  Unlike the dc algorithm that only employs a majorization of the function $g$,
we majorize both $\psi_1$ and $\psi_2$ and do not rely on a dc decomposition of the overall objective function $\Psi$.  Admittedly, this
differentiability assumption is rather restrictive for a nondifferentiable problem.  The difficulty in trying to relax this assumption is
the evaluation of the subdifferential $\partial \overline\Psi(z^{\,\nu+1})$ of the extended-value function $\overline{\Psi}$ that involves the
nonconvex set $\overline{Z}$ defined by the nondifferentiable function $\psi$.  It is our interest to pursue a thorough investigation
of this issue that links the theory of error bounds to the sequential convergence of iterative methods for specially structured
nondifferentiable optimization problems, without relaying the hard-to-evaluate subdifferential as needed here.  Regrettably, such a study is beyond
the scope of this paper.

\subsection{Two variants of the MM algorithm}

Though the MM algorithm discussed above converges to a d-stationary point, it potentially requires high computational cost
per iteration because multiple subproblems need to be solved.  To reduce the computational cost, we consider two variants of
the algorithm. The first one requires solving only one subproblem associated with an arbitrary pair of
indices $(i_1, i_2)\in \mathcal{A}_1(\theta^{\,\nu})\times \mathcal{A}_2(\theta^{\,\nu})$ in each iteration.

\vskip 0.2cm
\noindent\makebox[\linewidth]{\rule{\textwidth}{1pt}}
\noindent
MM-1: a simplified version of the MM algorithm for solving \eqref{eq:optimization model2}.\\
\noindent\makebox[\linewidth]{\rule{\textwidth}{1pt}}

\noindent
{\bf Initialization.}  Given are a scalar $c > 0$, an initial point $\theta^{\, 0} \in \Theta$ and an arbitrary
pair in indices $( i_1^0, i_2^0  )\in \mathcal{A}_1( \bar{\theta} ) \times \mathcal{A}_2( \bar{\theta} )$.  Let
$r_0 \triangleq \psi( \theta^{\, 0} ) \triangleq s_0$; thus
$z^{\, 0} \triangleq ( \theta^{\, 0}, r_0, s_0 ) \in {\cal Z}_{(i_1^0,i_2^0)}( \bar{\theta} )$.  Set $\nu = 0$.\\[0.05in]
{\bf Step 1.}  Let
$z^{\, \nu+1} \,\triangleq \, \displaystyle{
\operatornamewithlimits{\mbox{argmin}}_{z \in {\cal Z}_{(i_1^{\nu},i_2^{\nu})}(\theta^{\, \nu})}
} \, \wh{\Psi}_c(z,z^{\, \nu})$ for some $(i_1^\nu, i_2^\nu)\in \mathcal{A}_1(\theta^{\,\nu})\times \mathcal{A}_2(\theta^{\,\nu})$.\\[0.05in]
{\bf Step 2.} If $z^{\, \nu+1}$ satisfies a prescribed stopping rule, terminate; otherwise, return to Step 1 with $\nu$ replaced by $\nu+1$.
\hfill $\Box$

\noindent\makebox[\linewidth]{\rule{\textwidth}{1pt}}
\vskip 0.2cm

Unfortunately, the above MM algorithm is not guaranteed to converge to a d-stationary solution of \eqref{eq:optimization model2}.
The reason is easy to understand; namely, there is a ``for all index pairs'' condition in the requirement of d-stationarity, whereas each iteration
of the MM-1 algorithm solves only one subprogram corresponding to a single index pair.  Relaxing the ``for all'' requirement in d-stationarity,
we say that a point $\bar{\theta} \in \Theta$ is a {\sl weak} ${\cal M}$-{\sl stationary point} of \eqref{eq:optimization model2} if
{\sl there exists} $(\bar{i}_1,\bar{i}_2)\in \mathcal{A}_1(\bar{\theta})\times \mathcal{A}_2(\bar{\theta})$ such that
$\bar{\theta} \in \displaystyle{
\operatornamewithlimits{\mbox{argmin}}_{\theta\in \Theta}
} \, \mathcal{M}\Psi_{(\bar{i}_1,\bar{i}_2)}(\theta,\bar{\theta})$.  The prefix ${\cal M}$ is used to highlight that the
concept employs the majorization of the function $\Psi$ at $\bar{\theta}$ by the family of convex functions
$\left\{ \mathcal{M}\Psi_{(\bar{i}_1,\bar{i}_2)}(\bullet,\bar{\theta}) \right\}$.  In the case where $\varphi$ is the
identity function, a weak ${\cal M}$-stationary point of the function $\Psi$ on $\Theta$
coincides with a ``weak d-stationary point'' as defined in \cite[Subsection~3.3]{PangRazaviyaynAlvarado16}, thus
a weak ${\cal M}$-stationary point must be a critical point (in the sense of dc programming)
of a un-composed difference-max program.

The convergence of the MM-1 algorithm is stated in the proposition below, whose proof can be derived similarly to the proof of
Proposition~\ref{proposition: d-stationary}. We omit the details here.

\begin{proposition}\label{prop: weak dstat}
Suppose that $\Psi$ is bounded below on the closed convex set $\Theta$. For any accumulation point $( \theta^{\, \infty},r_{\infty},s_{\infty} )$ of the sequence $\left\{ z^{\, \nu} = ( \theta^{\, \nu}, r_\nu, s_\nu) \right\}$ generated by the MM-1 algorithm, if it exists,
$\theta^{\, \infty}$ is a weak ${\cal M}$-stationary point of \eqref{eq:optimization model2}. \hfill $\Box$
\end{proposition}

It is worth mentioning that unlike a d-stationary point of \eqref{eq:optimization model2} which is independent
of the representation of $\Psi$ by its components $\varphi$ and $\psi$, the concept of a weak ${\cal M}$-stationary point
not only depends on the decomposition of $\varphi = \varphi^{\, \uparrow} + \varphi^{\, \downarrow}$, but also depends on
the ``max-max'' representation of $\psi$.  To see the dependence on the former,
consider univariate functions:
{\small \[
\varphi(t)  \,= \,\underbrace{2t}_{\mbox{$\varphi^{\, \uparrow}_1(t)$}} + \underbrace{-t}_{\mbox{$\varphi^{\, \downarrow}_1(t)$}}
\, = \, \underbrace{t}_{\mbox{$\varphi^{\, \uparrow}_2(t)$}} +
\underbrace{0}_{\mbox{$\varphi^{\, \downarrow}_2(t)$}},\epc \psi(t) \,= \,
\underbrace{\max\,(2t, \, 1.5t )}_{\mbox{$g(t)$}} - \underbrace{\max\,(t\,,\, 0.5t)}_{\mbox{$h(t)$}}.
\]
}
One can easily check that $t=0$ is a ${\cal M}$-stationary point of $\varphi \circ \psi$ under the decomposition
$\varphi = \varphi_1^{\, \uparrow} + \varphi_1^{\, \downarrow}$ (with $(i_1, i_2) = (1,1)\in \mathcal{A}_1(0)\times \mathcal{A}_2(0)$),
but fails to be that under the decomposition $\varphi = \varphi_2^{\, \uparrow} + \varphi_2^{\, \downarrow}$.
To see the dependence on the latter, we may consider the example $\varphi(t) = t$ and $\psi(t) = \max\,(2t, 0) - \max\,(t, -t) = t\,-\,0$ for a scalar $t$,
and verify $t=0$ is a weak ${\cal M}$-stationary point of $\varphi\circ \psi$ under the former representation, but fails to be weakly ${\cal M}$-stationary
under the latter representation.   In addition, a weak ${\cal M}$-stationary point of \eqref{eq:optimization model2} may not be its
C-stationary point as indicated by $\varphi(t) = t$ and $\psi(t) = \max\,(2t, 0) - \max\,(t, -t)$ at $t=0$,
and vice versa as shown by $\varphi(t) = t$ and $\psi(t) = t - \max\,(0,2t)$ at $t=0$.
%
%
%
%
%

The second variant of the  MM algorithm is a probabilistic version  that requires solving only one single convex subprogram corresponding
to a randomly selected pair $(i_1,i_2)\in \mathcal{A}_{1;\varepsilon}(\theta^{\,\nu})\times \mathcal{A}_{2;\varepsilon}(\theta^{\,\nu})$,
each with positive probability of being selected.

\noindent\makebox[\linewidth]{\rule{\textwidth}{1pt}}

\noindent MM-2: A randomized version of the MM algorithm for solving \eqref{eq:optimization model2}.

\noindent\makebox[\linewidth]{\rule{\textwidth}{1pt}}

\noindent
{\bf Initialization.}
Given are a scalar $c > 0$, an initial point $\theta^{\, 0} \in \Theta$ and an arbitrary
pair in indices $( i_1^0, i_2^0  )\in \mathcal{A}_1( \bar{\theta} ) \times \mathcal{A}_2( \bar{\theta} )$.  Let
$r_0 \triangleq \psi( \theta^{\, 0} ) \triangleq s_0$; thus
$z^{\, 0} \triangleq ( \theta^{\, 0}, r_0, s_0 ) \in {\cal Z}_{(i_1^0,i_2^0)}( \bar{\theta} )$.  Set $\nu = 0$. \\[0.05in]
\noindent
{\bf Step 1.} Choose $(i_1^\nu,i_2^\nu)\in \mathcal{A}_{1;\varepsilon}(\theta^{\,\nu})\times \mathcal{A}_{2;\varepsilon}(\theta^{\,\nu})$ randomly and independently from the previous iterations so that
$$
p^{\nu,i_1,i_2}\,\triangleq \mbox{Prob}\left\{(i_1,i_2)\; \mbox{is chosen}\, \mid\,  \mbox{given}\; z^{\,\nu}\right\}  \, \geq\,  p_{\min} \,> \, 0.
$$
Let  $z^{\,\nu+\frac{1}{2}; i_1^\nu, i_2^\nu} \,\triangleq \, \displaystyle\operatornamewithlimits{\mbox{argmin}}_{z \in {\cal Z}_{(i_1^{\nu},i_2^{\nu})}(\theta^{\, \nu})}
 \, \wh{\Psi}_c(z,\,z^{\, \nu})$. \\[0.05in]
\noindent
{\bf Step 2.} Set $z^{\,\nu+1} \,\triangleq \, \left\{\begin{array}{ll}
z^{\,\nu+\frac{1}{2}; {i}_1^\nu,{i}_2^\nu}  \epc {\rm if}\;\, \wh{\Psi}_c
(z^{\,\nu+\frac{1}{2};\,{i}_1^\nu,\,{i}_2^\nu},\,z^{\,\nu}) \, <\, \varphi^{\,\uparrow}(r_{\nu})+\varphi^{\,\downarrow}(s_{\nu}),\\[0.1in]

z^{\,\nu}  \qquad\qquad {\rm otherwise}.
\end{array}\right.
$\\[0.05in]
\noindent
{\bf Step 3.}  If $z^{\,\nu+1}$ satisfies a prescribed stopping rule, terminate; otherwise, return to Step 1 with $\mu$ replaced by $\mu+1$.
\hfill $\Box$
\noindent\makebox[\linewidth]{\rule{\textwidth}{1pt}}

The following proposition extends the  convergence result of \cite[Proposition 7]{PangRazaviyaynAlvarado16} for dc programming.
We omit the proof again since it can be derived with little difficulty by combining the arguments of
Proposition~\ref{proposition: d-stationary} and \cite[Proposition 7]{PangRazaviyaynAlvarado16}.
\begin{proposition}\label{proposition:MMrandom}
Suppose that $\Psi$ is bounded below on the closed convex set $\Theta$.
Let $\left\{ z^{\, \nu} = ( \theta^{\, \nu}, r_\nu, s_\nu) \right\}$ be a sequence generated by the MM-2 algorithm. For any accumulation point $( \theta^{\, \infty},r_{\infty},s_{\infty} )$ of this sequence, if it exists,
$\theta^{\, \infty}$ is a d-stationary point of \eqref{eq:optimization model2} with probability one.  \hfill $\Box$
\end{proposition}

\section{A Semismooth Newton Method for the Subproblems}
\label{sec: subproblem}
The convex subproblems in Step 1 (see~\eqref{MM:subproblem}) are the workhorse of the MM algorithms. In the present section,
we discuss a semismooth Newton (SN) method for solving them.  There are two important motivations
for us to adopt this Newton-type method instead of various first-order methods that are quite popular in recent years.
One, although the computational cost for each step of the SN method may be more than that of first-order methods,
this cost is compensated by much fewer iterations, especially if the SN method is warm started by the solution of
the preceding subproblem.  The overall computational time of the SN method is thus potentially less than that of first-order methods.
The recent papers
\cite{QiSun2006,ZhaoSunToh2010,MilzarekUlbrich2014,YangSunToh2015,LiSunToh17} provide computational evidence to support the practical advantages of the SN method.
Two, the MM algorithm itself is a first-order method,
whose objective value can easily be stagnant after several iterations; see e.g.,~\cite[Chapter 1]{Lange16}.
First-order methods may perform fairly well if it is directly employed to solve stand-alone problems to low or moderate accuracy.  But unlike the SN method, which is able to compute very accurate solutions of the subproblems reasonably fast, first-order methods may not provide sufficiently accurate solutions when embedded in an iterative process such as
the MM algorithm as the methods of choice for solving the subproblems.

Recall that a locally Lipschitz continuous  function $\Phi: \Omega\to \mathbb{R}^m$  defined on the open set $\Omega$ is said to be
{\sl semismooth} at $\bar{x} \in \Omega$ if $\Phi$ is directionally differentiable at $\bar{x}$ and
for any $V \in \partial_C \Phi(x)$
with $x$ sufficiently near $\bar{x}$, $\Phi(x) - \Phi(\bar{x}) - {V} (x-\bar{x}) = o(\|\,x-\bar{x}\,\|)$;
$\Phi$ is said to be {\sl strongly semismooth} at $\bar{x}$ if $o(\|\,x-\bar{x}\,\|)$ is strengthened to
$O(\|\,x - \bar{x}\,\|^2)$ \cite{Mifflin1977,QiSun1993}.
The function $\Phi$ is said to be a (strongly) semismooth  function on $\Omega$ if it is (strongly) semismooth
everywhere in $\Omega$.  The class of semismooth functions is broad, including piecewise semismooth functions, and thus piecewise smooth functions \cite[Proposition 7.4.6]{FacchineiPang2003}.
It is also known that every piecewise affine map from $\mathbb{R}^n$ into $\mathbb{R}^m$ is strongly
semismooth \cite[Proposition 7.4.7]{FacchineiPang2003}.  A real-valued function $\phi : \Omega \to \mathbb{R}$ is
said to {\sl semismoothly differentiable} (SC$^1$) at $\bar{x} \in \Omega$ \cite[Subsections~7.4.1 and 8.3.3]{FacchineiPang2003}
if it is differentiable near $\bar{x}$ and its gradient is a semismooth function at $\bar{x}$.

With pioneering work by Kojima and Shindo \cite{KojimaShindo1986}, Kummer \cite{Kummer1988}, Pang \cite{Pang1990},
Qi and Sun \cite{QiSun1993}, among others,
the SN method is a generalization of Newton's method for solving the semismooth equation $\Phi(x) = 0$,
which has also been extended to a convex constrained
SC$^1$ minimization problem in \cite{PangQi1995}; this is the method that we propose to apply to solve the dual problem of
(\ref{MM:subproblem}).

Letting $\mathbf{1}$ denote the vector of all ones of appropriate dimensions, the subproblem \eqref{MM:subproblem} can
be written as: for the given pair $(i_1^{\nu},i_2^{\nu})$:
\begin{equation}\label{subproblem1}
\begin{array}{ll}
\displaystyle\operatornamewithlimits{\mbox{minimize}}_{\theta\in \Theta,\, r, s\in\mathbb{R}}\epc  &
 \, \varphi^{\,\uparrow}(r)\,+\,\varphi^{\,\downarrow}(s)\,  + \displaystyle\frac{c}{2}
\left[\,\|\,\theta - \theta^{\,\nu}\,\|^2 + (\, r - r_{\nu} \, )^2 + (\, s - s_{\nu} \, )^2 \,\right]\\[0.1in]
\mbox{subject to} \epc
&  b^{1;i_2^{\nu}}(\theta;\theta^{\, \nu}) - \mathbf{1} \, r \, \leq \, 0, \mbox{ and } b^{2_;i_1^{\nu}}(\theta;\theta^{\, \nu})
+ \mathbf{1} \, s \, \leq \, 0,
\end{array}
\end{equation}
with $b^{1;i_2^{\nu}}(\bullet;\theta^{\, \nu})$ and $b^{2;i_1^{\nu}}(\bullet;\theta^{\, \nu})$ denoting two vector functions
with convex differentiable  components given by:
\[ \begin{array}{ll}
b^{1;i_2^{\nu}}_i(\theta;\theta^{\, \nu}) \, \triangleq \, \psi_{1,i}(\theta) - \left[ \, h(\theta^{\, \nu}) +
\nabla \psi_{2,i_2^{\nu}}(\theta^{\, \nu})^T ( \, \theta - \theta^{\, \nu} \, ) \, \right], & \forall \, i \, = \, 1, \cdots, k_1,
\epc \mbox{and} \\ [5pt]
b^{2;i_1^{\nu}}_i(\theta;\theta^{\, \nu}) \, \triangleq \, g(\theta^{\, \nu}) +
\nabla \psi_{1,i_1^{\nu}}(\theta^{\, \nu})^T ( \, \theta - \theta^{\, \nu} \, ) - \psi_{2,i}(\theta), & \forall \, i \, = \, 1, \cdots, k_2.
\end{array} \]
The Lagrangian dual program is
\[
\displaystyle\operatornamewithlimits{\mbox{maximize}}_{\lambda\geq 0, \,\mu\geq 0}\; \zeta_{(i_1^{\nu},i_2^{\nu})}(\lambda,\mu;z^{\, \nu}),\epc\mbox{where}
\]
$${\small  \begin{array}{ll}
\zeta_{(i_1^{\nu},i_2^{\nu})}(\lambda,\mu;z^{\, \nu})  \,\triangleq\,  \displaystyle\operatornamewithlimits{\mbox{minimize}}_{\theta\in \Theta, \, r, s\in \mathbb{R}} \;
\left\{ \begin{array}{cc}
\varphi^{\,\uparrow}(r)  + \varphi^{\,\downarrow}(s) + \displaystyle\frac{c}{2}\,\|\,z-z^{\,\nu}\,\|^2\\[0.1in]
+ \lambda^T\,(\, b^{1;i_2^{\nu}}(\theta;\theta^{\, \nu}) - \mathbf{1} \, r\,) + \mu^T \,( \, b^{2;i_1^{\nu}}(\theta;\theta^{\, \nu}) + \mathbf{1} \, s \,)
\end{array}\right\} \\[0.3in]
=  \underbrace{\displaystyle\operatornamewithlimits{\mbox{minimize}}_{\theta\in \Theta} \;
\left\{\, \lambda^T \, b^{1;i_2^{\nu}}(\theta;\theta^{\, \nu}) + \mu^T \, b^{2;i_1^{\nu}}(\theta;\theta^{\, \nu})  +
\frac{c}{2} \, \|\,\theta-\theta^{\,\nu}\,\|^2\,\right\}}_{\mbox{sub-prob}_\theta^{\,\nu}} \\ [0.3in]
 + \underbrace{\displaystyle\operatornamewithlimits{\mbox{minimize}}_{r\in \mathbb{R}} \left\{ \varphi^{\,\uparrow}(r)-(\mathbf{1}^T \lambda) r
+\frac{c}{2}( r - r_{\nu}  )^2\right\}}_{\mbox{sub-prob}_r^\nu}
 +
\underbrace{\displaystyle \operatornamewithlimits{\mbox{minimize}}_{s\in \mathbb{R}} \left\{\varphi^{\,\downarrow}(s) + (\mathbf{1}^T\mu)s+\frac{c}{2}
( s - s_{\nu}  )^2\right\}}_{\mbox{sub-prob}_s^\nu} .
\end{array}}
$$
By Danskin's Theorem \cite[Theorem 10.2.1]{FacchineiPang2003},
the concave dual objective function $\zeta_{(i_1^{\nu},i_2^{\nu})}(\bullet,\bullet;z^{\, \nu})$ is differentiable with the gradient given by
$$
\nabla_\lambda \zeta_{(i_1^{\nu},i_2^{\nu})}(\lambda,\mu;z^{\,\nu}) \, =\, b^{1;i_2^{\nu}}(\bar{\theta}^{\, \nu};\theta^{\, \nu}) - \mathbf{1} \, \bar{r}_{\nu},
\;
\nabla_\mu \zeta_{(i_1^{\nu},i_2^{\nu})}(\lambda,\mu;z^{\,\nu}) \, = \, b^{2;i_1^{\nu}}(\bar{\theta}^{\, \nu};\theta^{\, \nu})+ \mathbf{1} \, \bar{s}_{\nu},
$$
where $\bar{\theta}^{\,\nu}$, $\bar{r}_\nu$ and $\bar{s}_\nu$ are, respectively, the unique minimizers of
sub-prob$_\theta^{\,\nu}$, sub-prob$_r^\nu$, and sub-prob$_s^\nu$ corresponding to the pair $( \lambda,\mu )$.
Here, we see the benefit of regularizing the scalar variables $r$ and $s$ that ensures the uniqueness of the minimizers
$\bar{r}_\nu$ and $\bar{s}_\nu$.  We can use the concept of proximal mappings to characterize these three minimizers.
For a proper closed convex function $\phi$ defined on $\mathbb{R}^n$ and a scalar $c > 0$, the proximal mapping $P^{\,c}_\phi$ is defined by
$$
P^{\,c}_\phi\,(u) \,\triangleq\, \operatornamewithlimits{\mbox{argmin}}_{v\in \mathbb{R}^n} \left\{\phi(v) + \frac{c}{2}\,\|\,v-u\,\|^2\right\},\quad u\in \mathbb{R}^n,
$$
which is globally Lipschitz continuous \cite[Proposition 12.27]{BauschkeCombettes11}.  One property of this mapping is that a function $\phi$ is convex piecewise linear-quadratic if and only if $P_\phi^{\,c}$ is piecewise linear~\cite[Proposition 12.30]{RockafellarRWets98}. (A function $\phi$ is called piecewise linear-quadratic if ${\rm dom}\,\phi$
can be represented as the union of finitely many polyhedral sets, relative to each of which $\phi(x)$ is given by a quadratic function;
see, e.g., \cite[Definition 10.20]{RockafellarRWets98}.)
Given that $b^{1;i_2^{\nu}}(\bullet;\theta^{\, \nu})$ and $b^{2;i_1^{\nu}}(\bullet;\theta^{\, \nu})$ are smooth mappings,
the function $\zeta_{(i_1^{\nu},i_2^{\nu})}$ is a SC$^1$ function for many interesting instances, in particular if $\Theta$ is a polyhedral set
and the functions $\lambda^T b^{1;i_2^{\nu}}(\bullet;\theta^{\, \nu}) + \mu^T b^{2;i_1^{\nu}}(\bullet;\theta^{\, \nu})$,
$\varphi^{\,\uparrow}$ and $\varphi^{\,\downarrow}$ are all convex piecewise linear-quadratic,
or the proximal mappings of these functions are piecewise smooth.

The following is the globally convergent and locally superlinearly convergent SN method to minimize a convex SC$^1$ function $\phi$ on
a closed convex set $X$; see \cite{PangQi1995} for details.  In the case of the dual program in question, the set $X$ is the
nonnegative orthant in the $(\lambda,\mu)$-space.  Specializations of the method are possible and will be illustrated with a
least-squares piecewise affine regression problem.

\noindent\makebox[\linewidth]{\rule{\textwidth}{1pt}}

\noindent A Semismooth Newton method for $\displaystyle{
\operatornamewithlimits{\mbox{minimize}}_{x \in X}
} \, \phi(x)$ with a SC$^{\, 1}$ function $\phi$.

\noindent\makebox[\linewidth]{\rule{\textwidth}{1pt}}

\noindent
{\bf Initialization.} Given an initial point $x^0 \in X$ and two scalars $\rho, \sigma \in (0,1)$. Set $k=0$.\\[0.1in]
{\bf Step 1.} (Find the search direction.) Given $x^k \in X$, 
pick $V^k\in \partial_C \nabla \phi(x^k)$
and a scalar $\varepsilon_k > 0$.  Solve the following strictly convex quadratic program  in the variable $d$ to obtain $d^k$:
{\small \[
\displaystyle\operatornamewithlimits{\mbox{minimize}}_{d \in X - x^k} \
\nabla \phi(x^k)^T d +
\displaystyle\frac{1}{2}\,d^{\,T}\,(V^k + \varepsilon_k I)\,d
\]}
{\bf Step 2.} (Line search by the Armijo rule.)
Set $x^{k+1} \triangleq x^k + \rho^{m_k} \, d^k$, where $m_k$ is the smallest nonnegative integer $m$ for which
{\small $$
\phi(x^k + \rho^m d^k) \, \leq \, \phi(x^k) +
\sigma \, \rho^m \, \nabla \phi(x^k)^T \, d^k.
$$}
{\bf Step 3.} If $x^{k+1}$ satisfies a prescribed stopping rule, terminate; otherwise,
return to Step~1 with $k$ replaced by $k+1$.  \hfill $\Box$

\noindent\makebox[\linewidth]{\rule{\textwidth}{1pt}}

The major computational cost of the above semismooth Newton method is to find the search direction in Step~1.
To accomplish this task, one may apply any quadratic programming solver.
For the special case when $\psi$ is a difference of convex piecewise affine function as in \eqref{eq:d-affine},
we may instead find the search direction by solving a system of linear equations, leading to a very effective,
provably convergent, overall algorithm for solving problem \eqref{eq:optimization model2} that includes the
problem of piecewise affine regression.

\subsection{A special case where $\psi$ is difference-convex-piecewise-affine}

Let $\psi$ be a difference of convex piecewise affine functions as in \eqref{eq:d-affine}.  Then the
functions $b^{1;i_2^{\nu}}(\bullet;\theta^{\, \nu})$ and $b^{2;i_1^{\nu}}(\bullet;\theta^{\, \nu})$ in \eqref{subproblem1} are affine,
for which we may write (dropping the indices $i_1^{\nu}$ and $i_2^{\nu}$ and the dependence on $\theta^{\, \nu}$) as
$b^1(\theta) \,\triangleq \, B^1\theta -\beta^1$ and $b^2(\theta) \,\triangleq \, B^2\theta-\beta^2$ for some
matrices $B^1 \in \mathbb{R}^{k_1\times m}$ and $B^2\in \mathbb{R}^{k_2\times m}$ and vectors $\beta^1\in \mathbb{R}^{k_1}$
and $\beta^2\in \mathbb{R}^{k_2}$.
Instead of \eqref{MM:subproblem}, we may consider the following alternative subproblems in the MM algorithm with two
additional auxiliary variables $\wh{r}\in \mathbb{R}^{k_1}$ and $\wh{s}\in \mathbb{R}^{k_2}$:
\[\begin{array}{ll}
\displaystyle\operatornamewithlimits{\mbox{minimize}}_{\theta\in \Theta, \,r,\,s,\, \wh{r},\, \wh{s}}\epc  &
 \, \varphi^{\,\uparrow}(r)\,+\,\varphi^{\,\downarrow}(s)\, \\
 & + \,\displaystyle\frac{c}{2}\left[\,\|\,\theta - \theta^{\,\nu}\,\|^2 + (\, r - r_{\nu} \, )^2 + (\, s - s_{\nu} \, )^2 +  \|\,\wh{r} - \wh{r}^{\,\nu}\,\|^2 + \|\,\wh{s} - \wh{s}^{\,\nu}\,\|^2  \,\right]\\[0.1in]
\mbox{subject to} \epc
&  B^1\,\theta - \mathbf{1}r + \wh{r}\,=\,\beta^1,\epc B^2\,\theta + \mathbf{1}s + \wh{s}\,=\, \beta^2,\epc \wh{r}\,\geq \, 0, \epc \wh{s}\,\geq \,0.
\end{array}
\]
The (subsequential) convergence of the MM algorithms discussed in Section~\ref{sec: MM} can be easily extended
to the above formulation with the additional regularization term $\displaystyle\frac{c}{2}\,\left[ \,
\|\,\wh{r} - \wh{r}^{\,\nu}\,\|^2 + \|\,\wh{s} - \wh{s}^{\,\nu}\,\|^2  \,\right]$.
The Lagrangian dual program is
\[
\displaystyle\operatornamewithlimits{\mbox{maximize}}_{\substack{\lambda\in \mathbb{R}^{k_1},\,
\mu\in \mathbb{R}^{k_2}}} \;\xi(\lambda,\mu;\wh{z}^{\,\nu}),
\epc \mbox{where} \epc \wh{z}^{\,\nu}\,\triangleq\, \left(\,\theta^{\,\nu}, r_{\nu}, s_{\nu}, \wh{r}^{\,\nu}, \wh{s}^{\,\nu}\,\right)
\epc \mbox{and}\epc
\]
\vspace{-0.8em}
\[ 
{\small \begin{array}{l}
\xi(\lambda,\mu;\wh{z}^{\,\nu})  \triangleq
\displaystyle\operatornamewithlimits{\mbox{minimize}}_{\substack{
\theta\in \Theta, \, r,  s\in\mathbb{R},\\[0.05in]
 \wh{r}\,\geq\, 0,\, \wh{s}\,\geq \,0}} \; \left\{\begin{array}{ll}
\varphi^{\,\uparrow}(r)  + \varphi^{\,\downarrow}(s) + \displaystyle\frac{c}{2}\,\big[\,\|\,\theta-\theta^{\,\nu}\,\|^2 +(\, r - r_{\nu} \, )^2 \\[0.1in]
+(\, s - s_{\nu} \, )^2
+  \|\,\wh{r} - \wh{r}^{\,\nu}\,\|^2
+ \|\,\wh{s} - \wh{s}^{\,\nu}\,\|^2 \,\big]\\[0.1in]
+\lambda^T(\,B^1\,\theta-\mathbf{1}r + \wh{r}-\beta^1\,)+\mu^T(\,B^2\,\theta + \mathbf{1}s + \wh{s}-\beta^2\,)
\end{array}\right\}\\[0.4in]
 = \, - \lambda^T\beta^1 \,-\, \mu^T\beta^2 \, +\,
\displaystyle\operatornamewithlimits{\mbox{minimize}}_{\theta\in \Theta} \; \left\{\,\lambda^TB^1\,\theta + \mu^T B^2\,\theta
+ \frac{c}{2}\,\|\,\theta-\theta^{\,\nu}\,\|^2\,\right\} \\[0.15in]
 + \, \displaystyle\operatornamewithlimits{\mbox{minimize}}_{r\in \mathbb{R}} \left\{ \varphi^{\,\uparrow}(r)-(\mathbf{1}^T \lambda)\,
r+\frac{c}{2}( r - r_{\nu}  )^2\right\}  + \displaystyle \operatornamewithlimits{\mbox{minimize}}_{s\in \mathbb{R}} \left\{\varphi^{\,\downarrow}(s)+(\mathbf{1}^T\mu)\,s+\frac{c}{2}( s - s_{\nu} )^2\right\}
\\[0.15in]
+ \,\displaystyle\operatornamewithlimits{\mbox{minimize}}_{\wh{r}\,\geq \,0}\; \left\{\lambda^T\,\wh{r}\,+\,\frac{c}{2}\,\|\,\wh{r}-\wh{r}^{\,\nu}\,\|^2\right\} \, +\,
\displaystyle \operatornamewithlimits{\mbox{minimize}}_{\wh{s}\,\geq \,0}\; \left\{\mu^T\,\wh{s}\,+\,\frac{c}{2}\,\|\,\wh{s}-\wh{s}^{\,\nu}\,\|^2\right\}.
\end{array}}
\] 
As in the previous discussions, the concave function $\xi(\bullet,\bullet;\wh{z}^{\, \nu})$ is continuously differentiable.
The proximal mappings associated with $\wh{r}$ and $\wh{s}$ are essentially the projections onto the nonnegative orthant, which must be
semismooth \cite[Theorem~4.5.2 \& Proposition~7.4.6]{FacchineiPang2003}.
If $\Theta$ is a polyhedral set and the proximal mappings of $\varphi^{\,\uparrow}$ and $\varphi^{\,\downarrow}$ are also semismooth,
the dual program is  an unconstrained SC$^1$ problem.  Step~1 in the SN method is thus to solve an unconstrained strictly convex quadratic program,
or equivalently, to solve a linear equation.  Furthermore, the minimization in $\theta$ within the dual function $\xi(\lambda,\mu;\wh{z}^{\,\nu})$ is the projection
onto the set $\Theta$, which is easy if $\Theta$ is a simple set; e.g., containing upper and lower bounds only.  In this case, the bulk of the
computational effort per iteration of the combined MM and SN methods 
essentially consists of solving systems of linear equations in the $(\lambda,\mu)$-space within the fast convergent SN method.

\section{Numerical Experiments} \label{sec:numerical}

In order to demonstrate the effectiveness of the proposed nonsmooth MM+SN method for solving some non-standard statistical estimation problems,
we conduct various numerical experiments on an unconstrained least-squares continuous piecewise affine regression problem by taking the error function
$\varphi(y,t) = \frac{1}{2}(t-y)^2$ and the piecewise affine model $\psi$ as in \eqref{eq:d-affine}.  Besides providing evidence of the promise
of the MM+SN method itself (see Table~\ref{table1} and Figures~\ref{figure: MM speed} and \ref{figure: Newton speed} for summaries
of its performance statistics), our experiments also aim
to establish the superiority of the piecewise affine statistical model over the standard linear model. 

Specifically, the optimization problem we consider
in this section is the following:
\[{\small
\displaystyle{
\operatornamewithlimits{\mbox{minimize}}_{\theta}
} \ \left\{ \begin{array}{l}
\displaystyle{
f_N(\theta)\triangleq  \frac{1}{N}
} \displaystyle{
\sum_{s=1}^N
}  \left[ \, y_s - \displaystyle{
\max_{1 \leq i \leq k_1}
} \, \left\{  ( \, a^i \, )^T x^s + \alpha_i \, \right\} - \displaystyle{
\max_{1 \leq i \leq k_2}
} \, \left\{  ( \, b^i \, )^T x^s + \beta_i \, \right\} \, \right]^2  \\ [0.15in]
\qquad\qquad  + \, \gamma_N \, \displaystyle{
\sum_{i=1}^m
} \, \left[  \, | \, \theta_i \, | - p_i(\theta_i) \, \right]
\end{array} \right\}}\]
where $\theta=\left\{ \left( a^i, \alpha_i \right)_{i=1}^{k_1}, \left( b^i, \beta_i \right)_{i=1}^{k_2} \right\}\in \mathbb{R}^{(k_1+k_2)(d+1)}$ and each $p_i$ is a differentiable convex function.  
For this optimization problem which must have a global minimizer,
Proposition~\ref{pr:properties of d-stat} yields some
interesting properties of its d-stationary points, provided that each univariate $p_i(\theta_i)$ is piecewise linear-quadratic.
All our computations are done in Matlab on Mac OS X with 1.7 GHz Intel Core i7 and 8 GB RAM.  We employ the randomized version of
the MM algorithm to solve the problem with $\varepsilon = 10^{-4}$ for the ``$\varepsilon$-argmax'' expansion.
The iterations of the MM algorithm are terminated if
$\displaystyle\frac{|\,f_N(\theta^{\,\nu+1}) - f_N(\theta^{\,\nu})\,|}{\max\,(1,\, |\,f_N(\theta^{\,\nu})\,|\,)}\,\leq \,10^{-4}$;
the iterations of the inner SN algorithm are terminated if
$\|\nabla \xi(\lambda^k, \mu^k;\wh{z}^{\,\nu})\|\leq \max\left(\,10^{-6},\,10^{-2}\,|\,f_N(\theta^{\,\nu+1}) - f_N(\theta^{\,\nu})\,|\,\right)$,
where $\xi$ is the objective of the $\nu$-th dual subproblem.

\subsection{Synthetic data} In this set of runs we consider two simulation examples. In both examples the penalty parameter $\gamma_N$ is taken to be 0.

\noindent
{\bf Example 1.}  We consider a 2-dimensional convex piecewise linear
model $$
y \, = \, \max\left\{
x_1 + x_2\,,\,
x_1 - x_2\,,\,
-2x_1 + x_2\,,\,
-2x_1 - x_2\,\right\} + \varepsilon,
$$
whose 3-D plot is given in Figure \ref{fig:convex model}. Although this is a convex model, the resulting least-squares problem is still nonconvex.
Random samples of size $N =50, 100, 200, 500$ are generated uniformly from $[-1,1]\times [-1,1]$, respectively, and $\varepsilon$ is drawn from
a uniform distribution in $[ \, -0.5,\, 0.5 \,]$.  The objective values of the computed solutions for $k_1=4$ and $k_2=0$ over $500$ runs
(with each run corresponding to one initial point) by the MM+SN algorithm are shown in Figure~\ref{fig:convex landscape}.
Two observations are made from this figure.  One, the d-stationary values (which are also locally minimum values) are clustered into a
few distinguished groups.  In a recent paper \cite{CuiPang18}, we provide theoretical justification of  this observation by showing
that there are only finitely many d-stationary values for several classes of piecewise programs, of which the least-squares piece affine
regression problem is a special instance.  Two, with the increase of the sample size, more and more initial points lead to the solutions
that appear to yield the globally minimum value.  The total numbers of initial points that lead to the smallest objective
values with respect to different sample sizes are listed in Figure~\ref{percentage}.  In Figure~\ref{fig:cvx_result}, we plot
the solutions given by the smallest objective values of different sample sizes in Figure \ref{fig:convex landscape}.
These solutions almost recover the  original convex regression model, thus (empirically) validating the MM+SN methodology for a nonconvex
minimization problem in the recovery of a convex statistical model.

\begin{figure}[h]
\begin{minipage}{.4\textwidth}
\centering
\includegraphics[width=.8\textwidth]{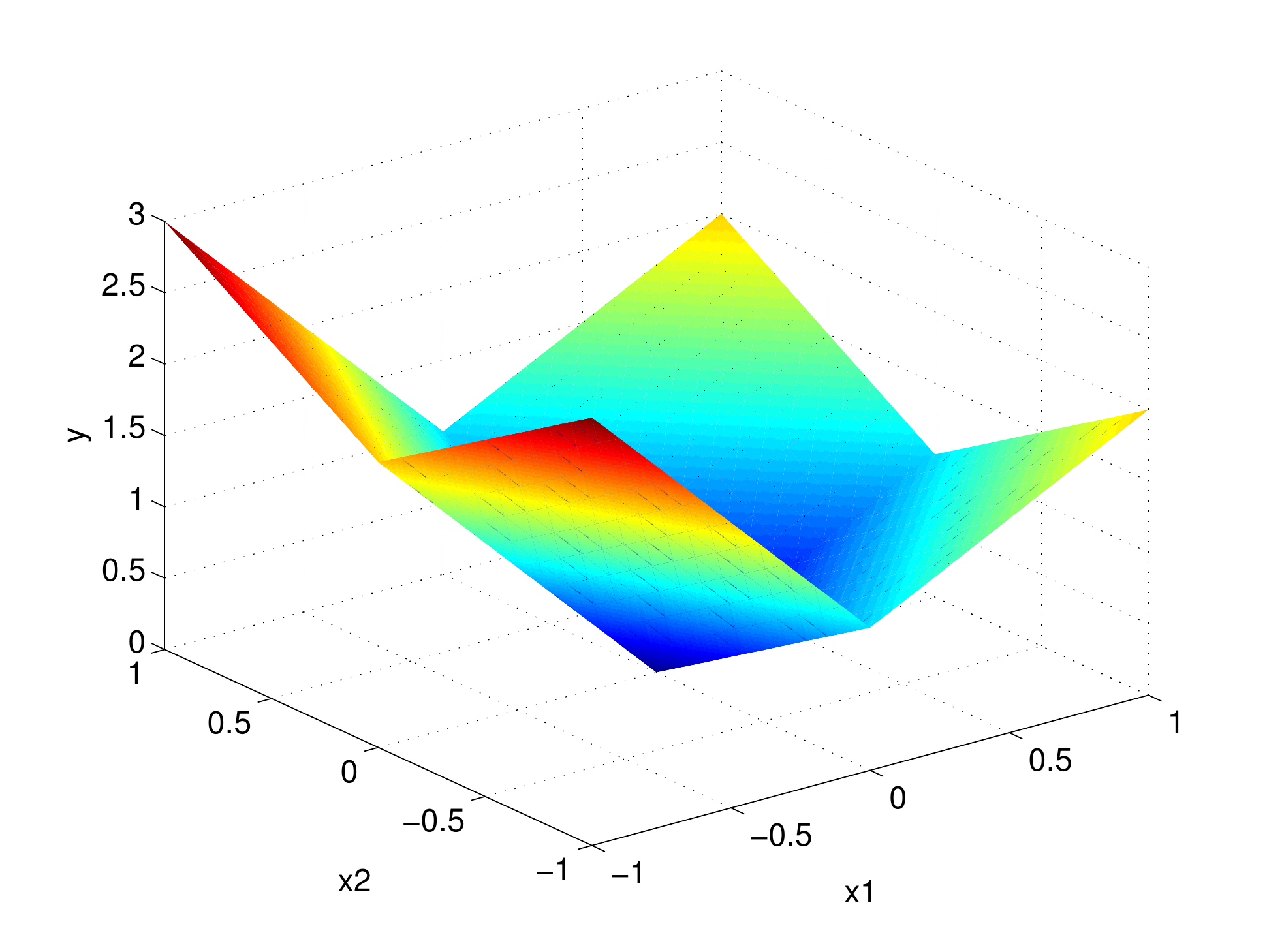}
\caption{\scriptsize{The 3-D plot of Example 1.}}
\label{fig:convex model}
\end{minipage}%
\epc
\begin{minipage}{.55\textwidth}
\centering
\includegraphics[width=0.5\textwidth]{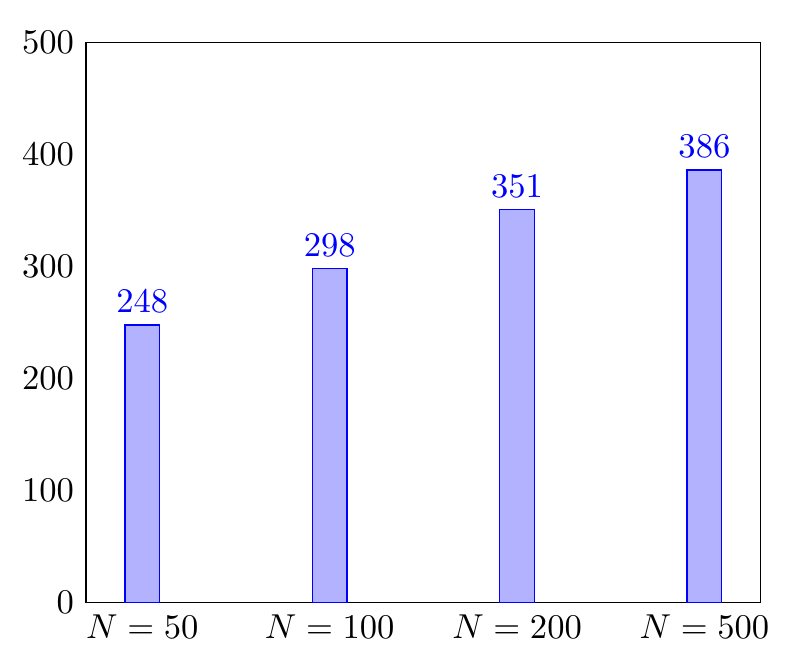}
\caption{\scriptsize{The number of initial points that lead to the smallest objective values.}}
\label{percentage}
\end{minipage}%
\end{figure}

\begin{figure}[h]
\begin{minipage}{.25\textwidth}
\centering
\includegraphics[width=\textwidth]{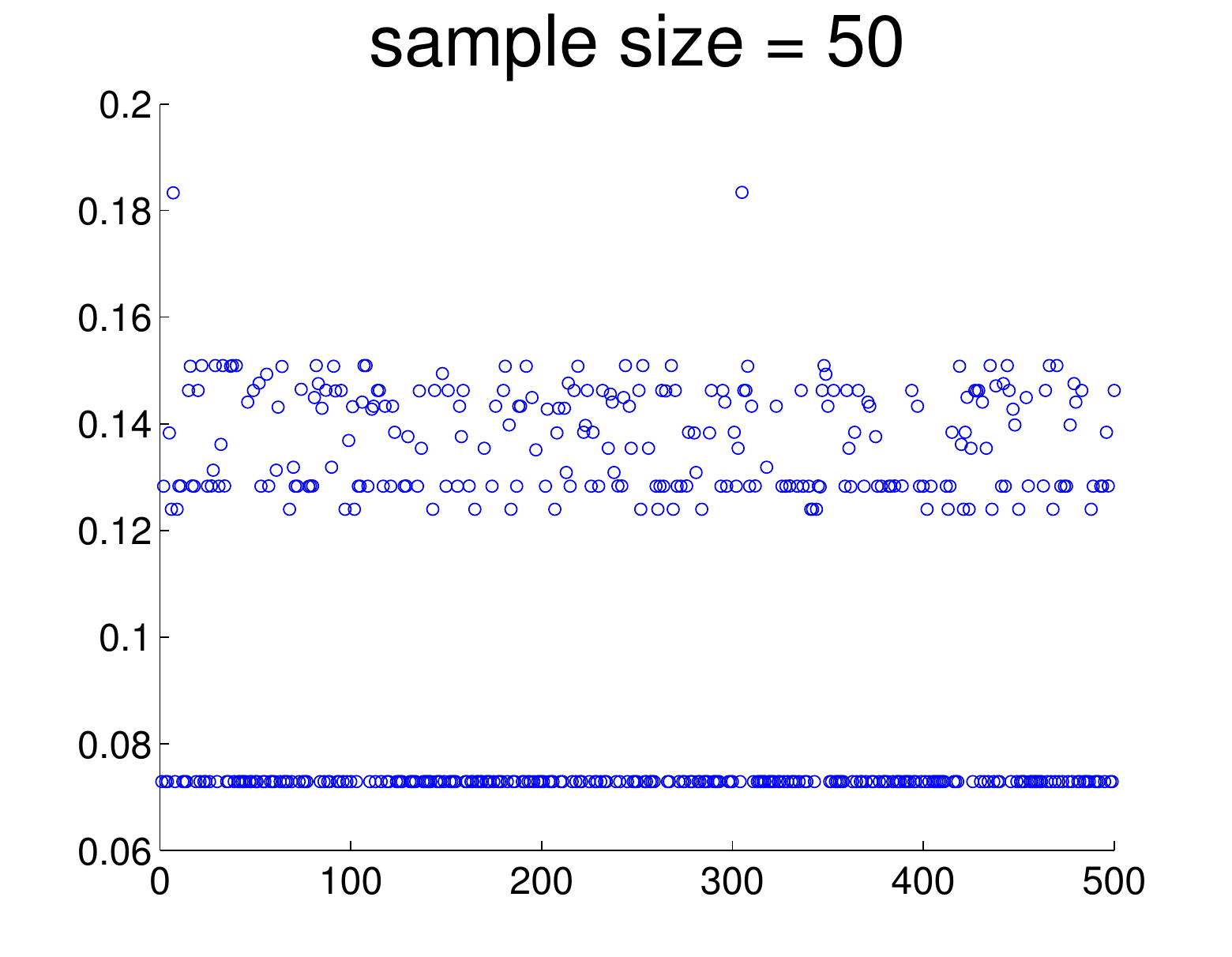}
\end{minipage}%
\begin{minipage}{.25\textwidth}
\centering
\includegraphics[width=\textwidth]{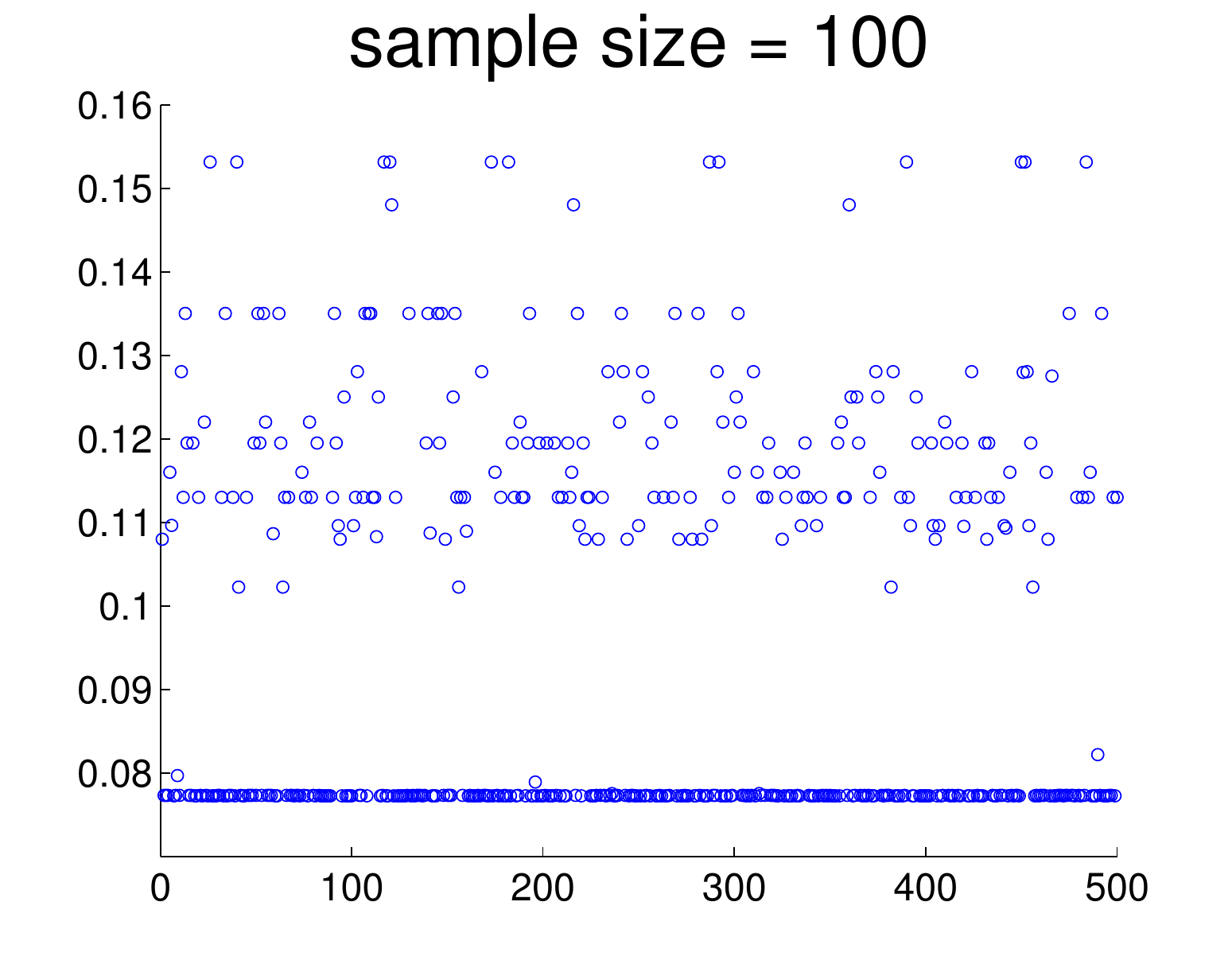}
\end{minipage}
\begin{minipage}{.25\textwidth}
\centering
\includegraphics[width=\textwidth]{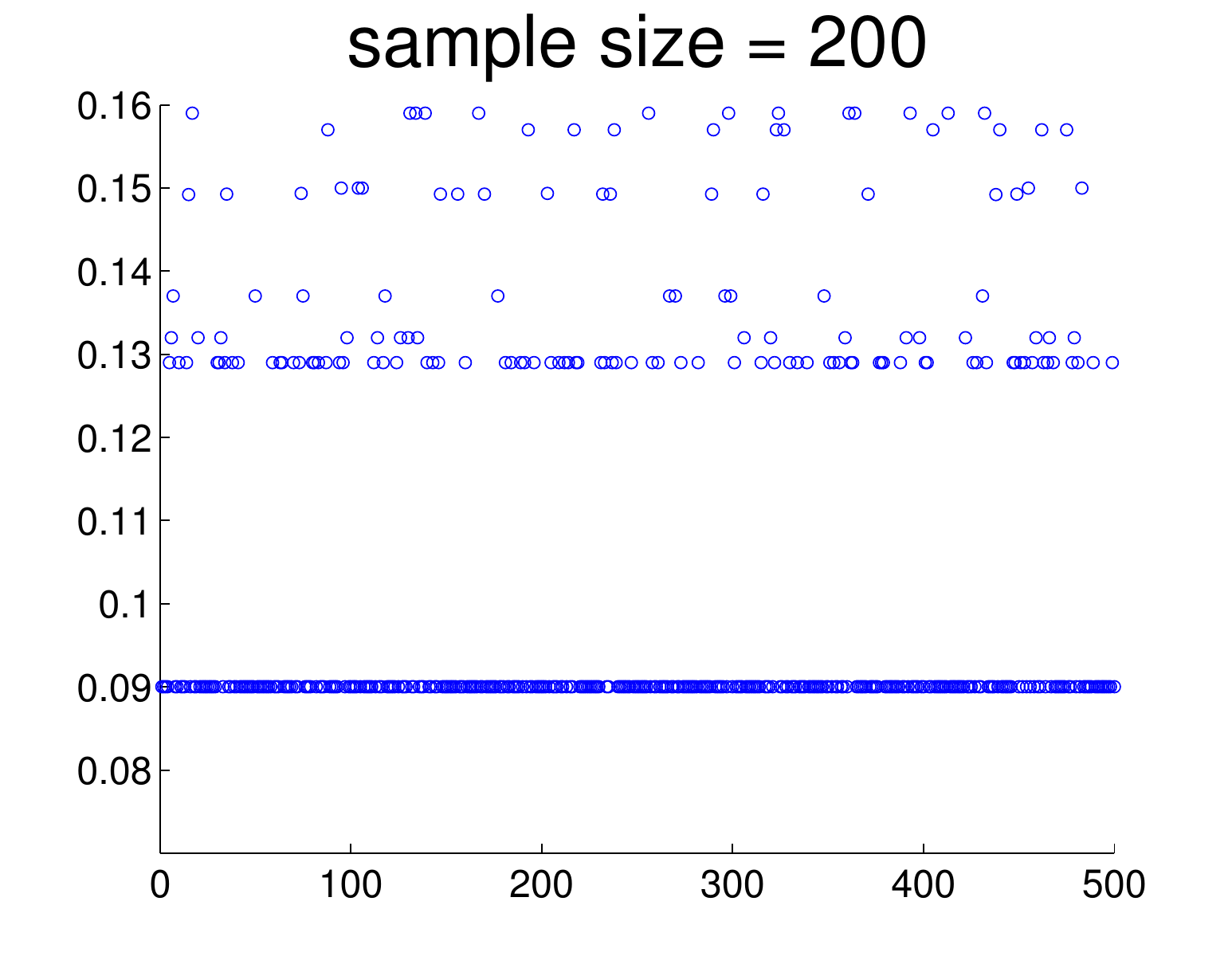}
\end{minipage}
\begin{minipage}{.25\textwidth}
\centering
\includegraphics[width=\textwidth]{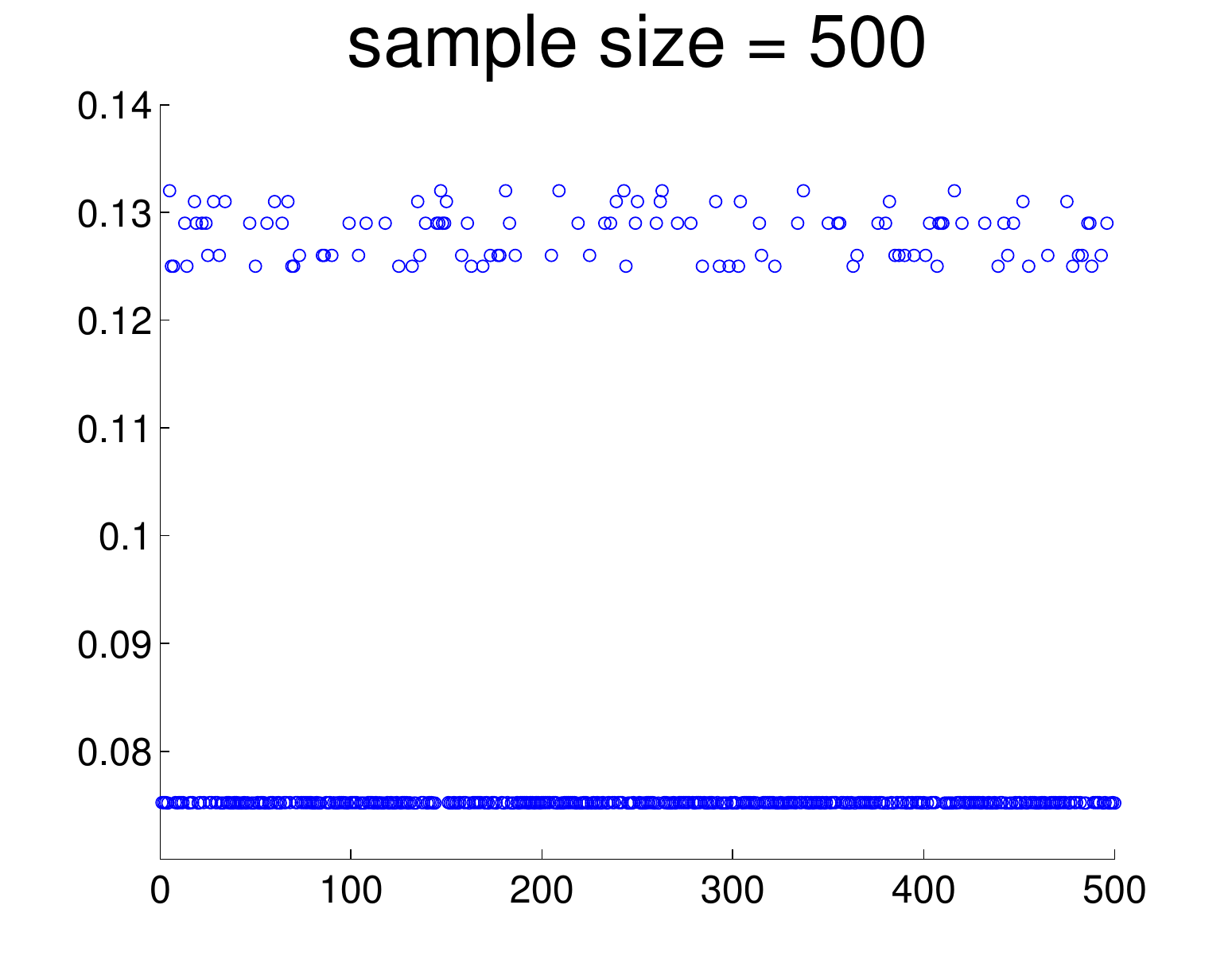}
\end{minipage}
\caption{\small{The objective values computed by the MM+SN algorithm of Example 1.}}
\label{fig:convex landscape}
\end{figure}

\begin{figure}[h]
\centering
\begin{minipage}{.24\textwidth}
\centering
\includegraphics[width=1\textwidth]{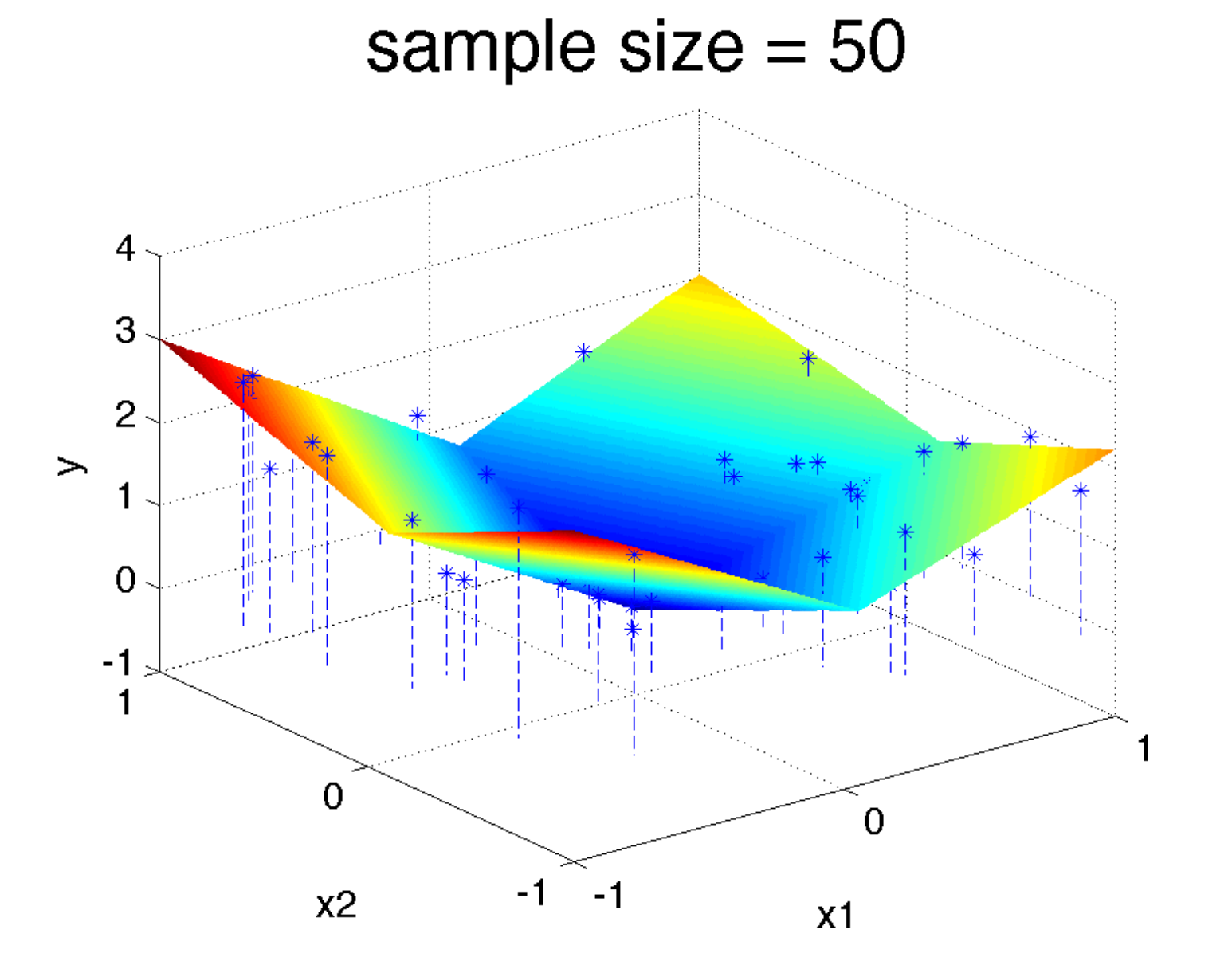}
\end{minipage}%
\begin{minipage}{.24\textwidth}
\centering
\includegraphics[width=1\textwidth]{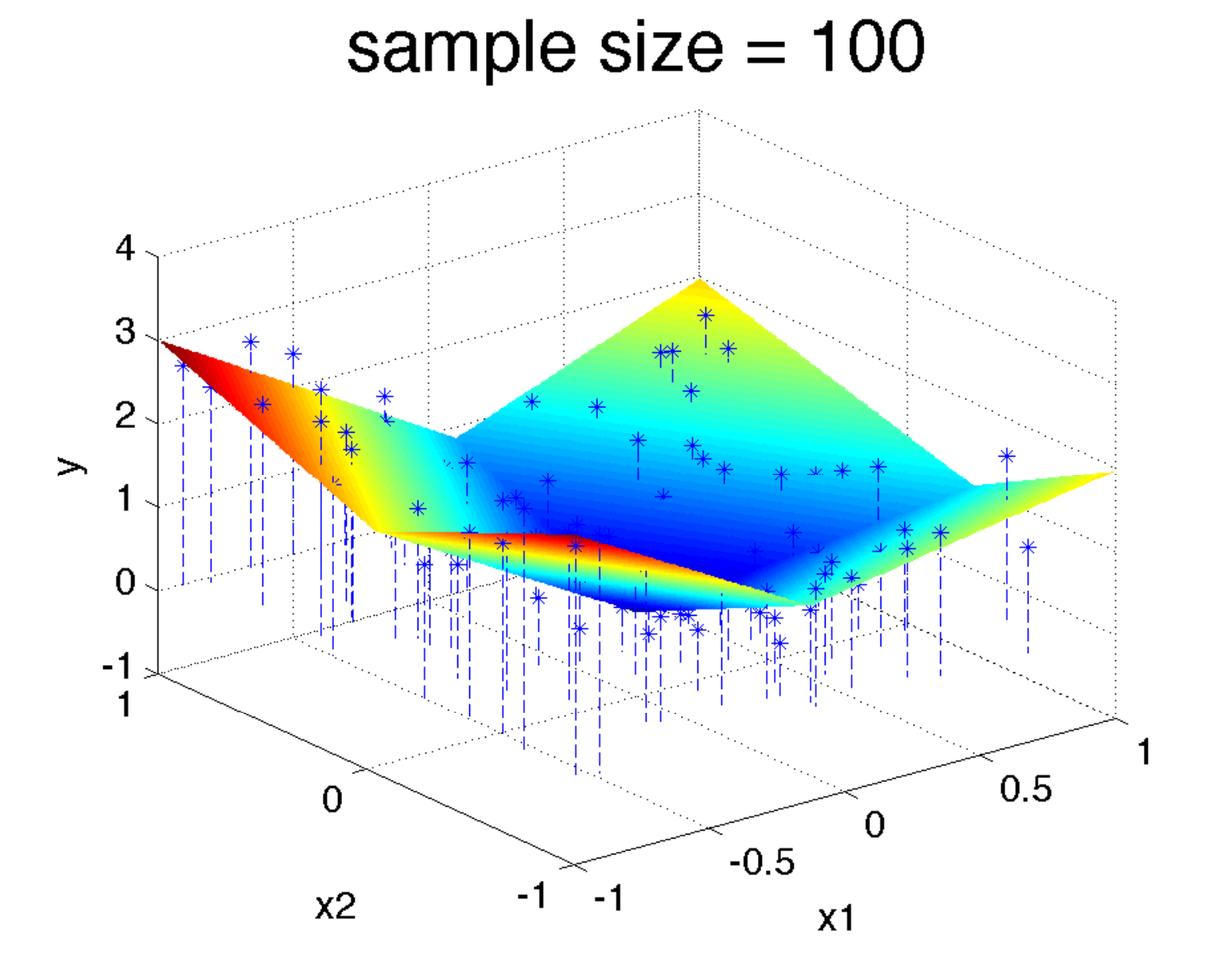}
\end{minipage}
\begin{minipage}{.24\textwidth}
\centering
\includegraphics[width=1\textwidth]{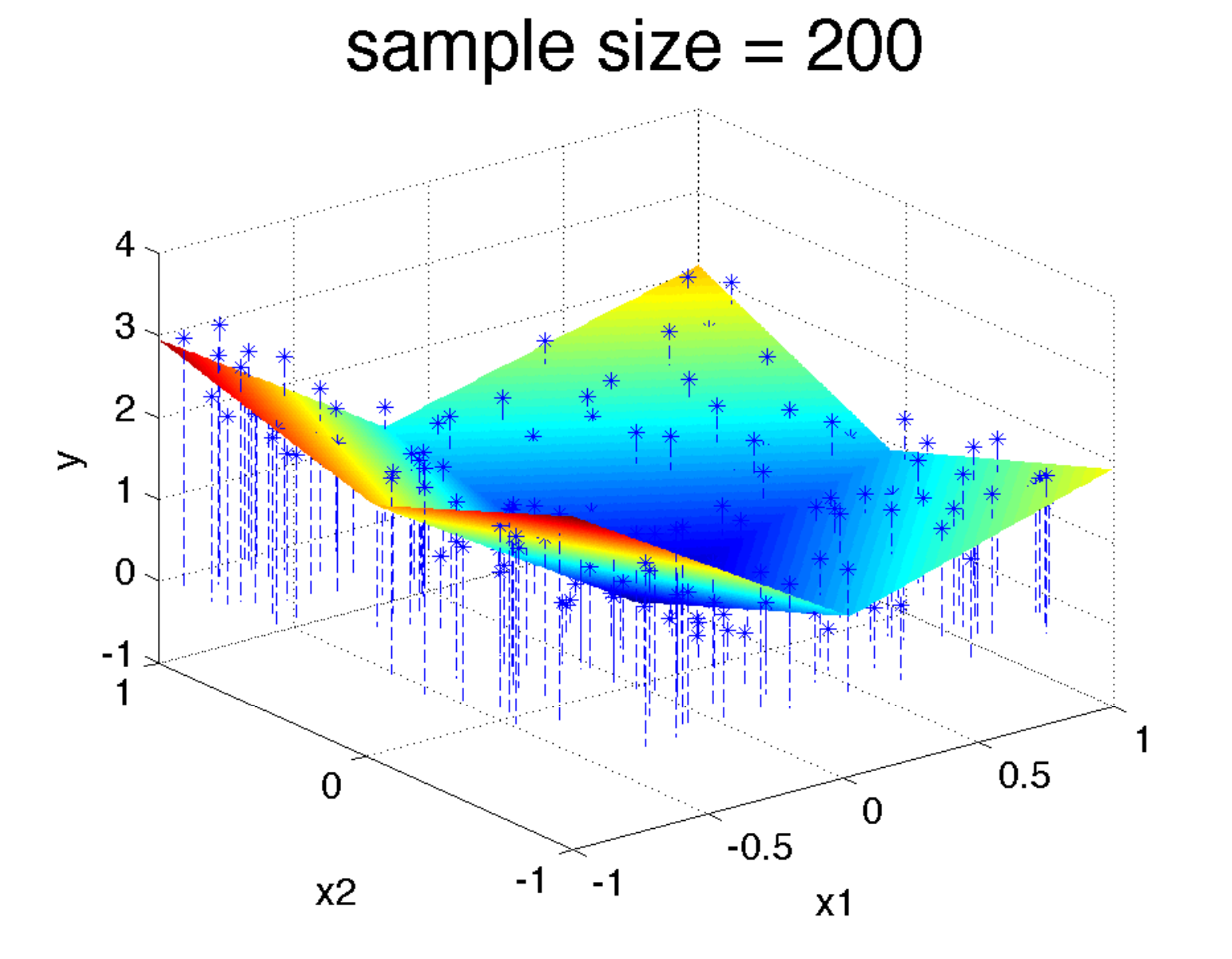}
\end{minipage}
\begin{minipage}{.24\textwidth}
\centering
\includegraphics[width=1\textwidth]{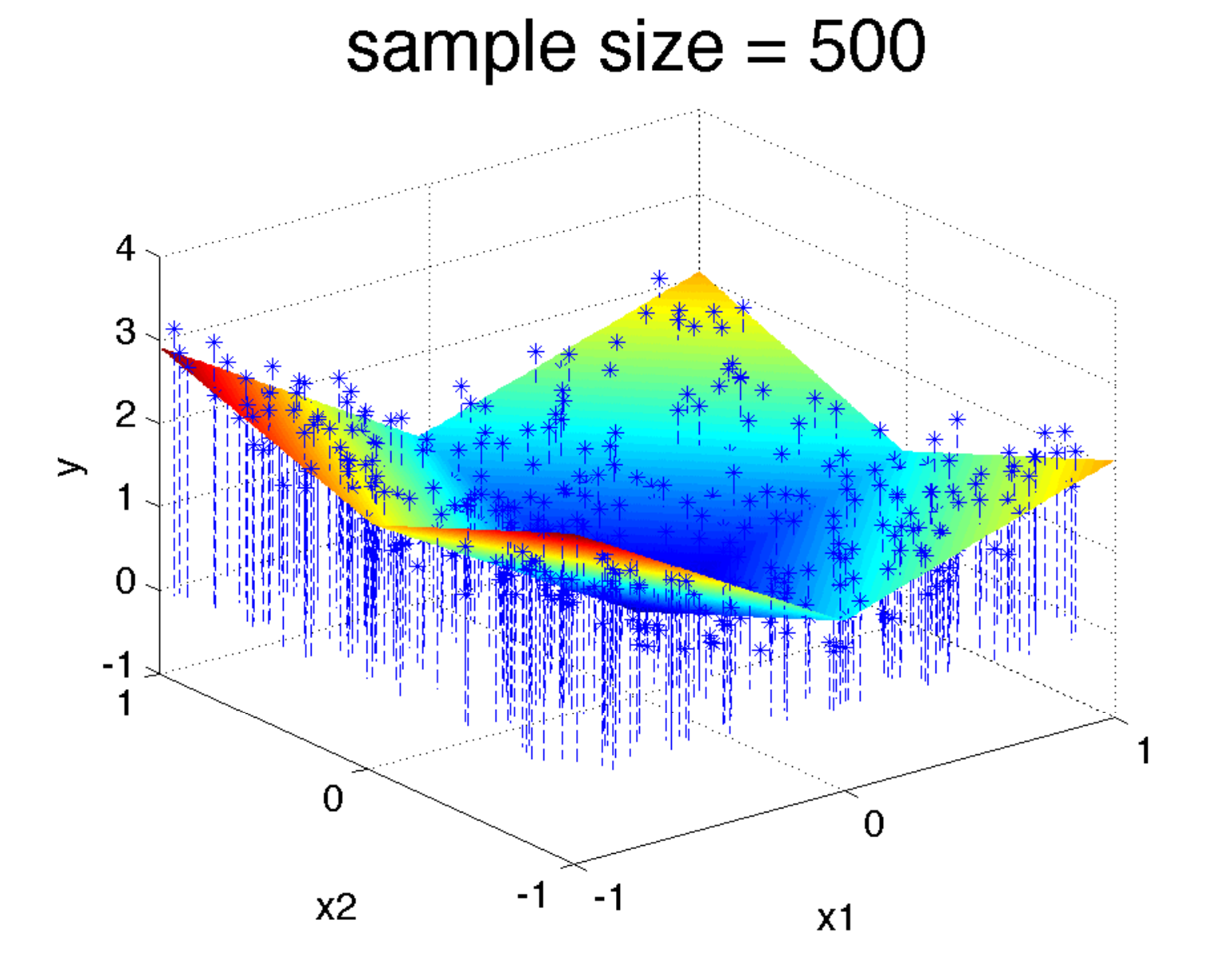}
\end{minipage}
\caption{{\small The solutions of Example 1 with the smallest objective values.}}
\label{fig:cvx_result}
\end{figure}

\noindent
{\bf Example 2.} Next we consider a 2-dimensional nonconvex piecewise affine model
$$
y = \max\left\{x_1-2x_2\,,\, -2x_1+ x_2 + 1\right\} - \max\left\{3x_1-2x_2\,,\, 2x_1+5x_2\right\} + \varepsilon,
$$
whose 3D-plot is given in Figure \ref{figure:noncvx model}. The projection of this figure onto $y = 0$ is demonstrated in
Figure~\ref{figure:noncvx projection}.  As in Example~1, random samples of size $N = 50, 100, 200, 500$ are generated
uniformly from $[-1,1]\times [-1,1]$, and $\varepsilon$ is drawn from a uniform distribution in $[-0.5,\, 0.5]$.
The objective values of the computed solutions over $500$ runs are shown in Figure~\ref{figure:noncvx obj}.
Similar observations to Example 1 can be made with regards to the concentration of the (computed) stationary values.
We also  plot the solutions corresponding to smallest objective values for different sample sizes in Figure \ref{figure:noncvx:result}.

\begin{figure}[H]
\begin{center}
\begin{minipage}{.3\textwidth}
\centering
\includegraphics[width=1.0\textwidth]{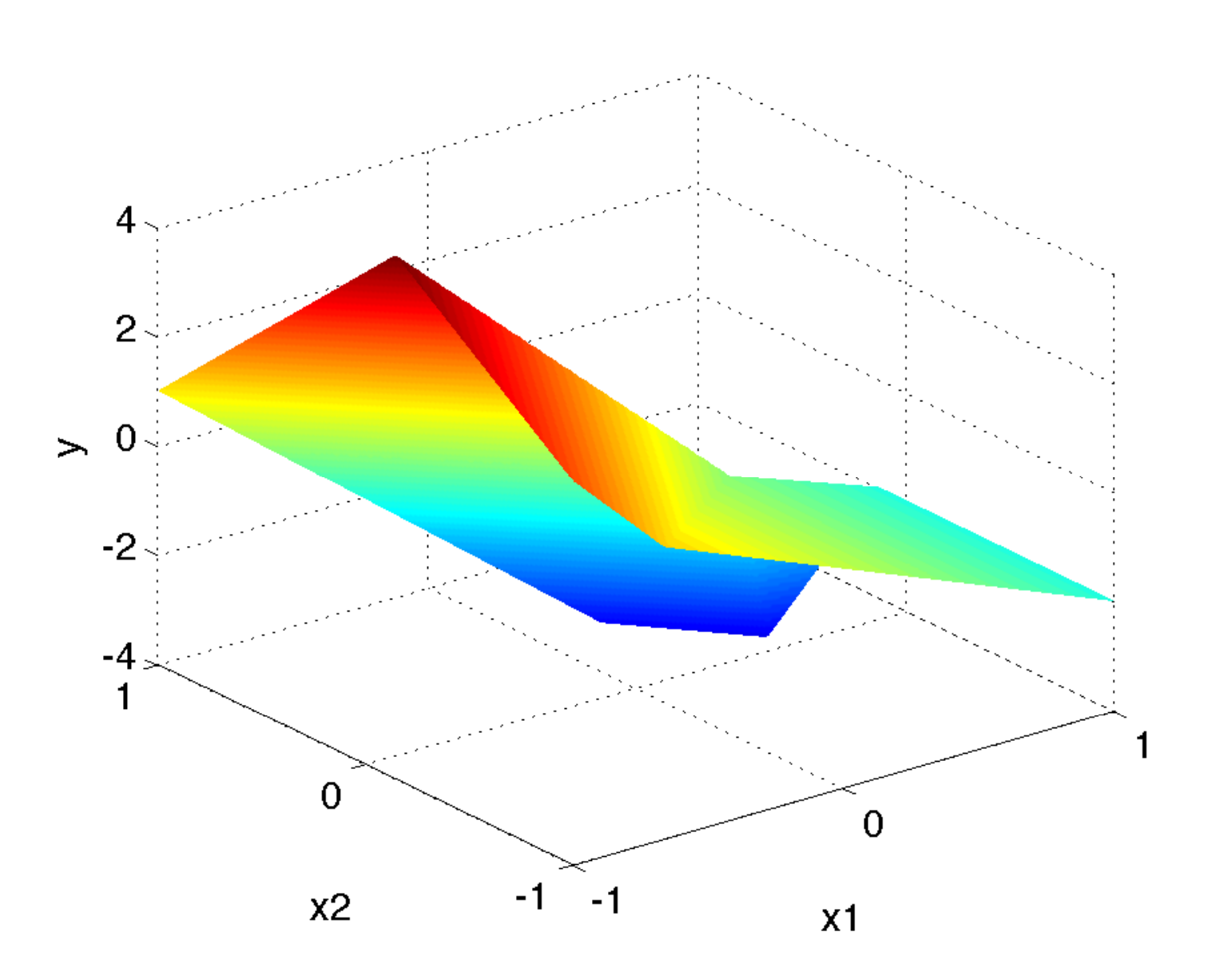}
\caption{\scriptsize{The 3-D plot of Example 2.}}
\label{figure:noncvx model}
\end{minipage}
\;\,
\begin{minipage}{.3\textwidth}
\centering
\includegraphics[width=1\textwidth]{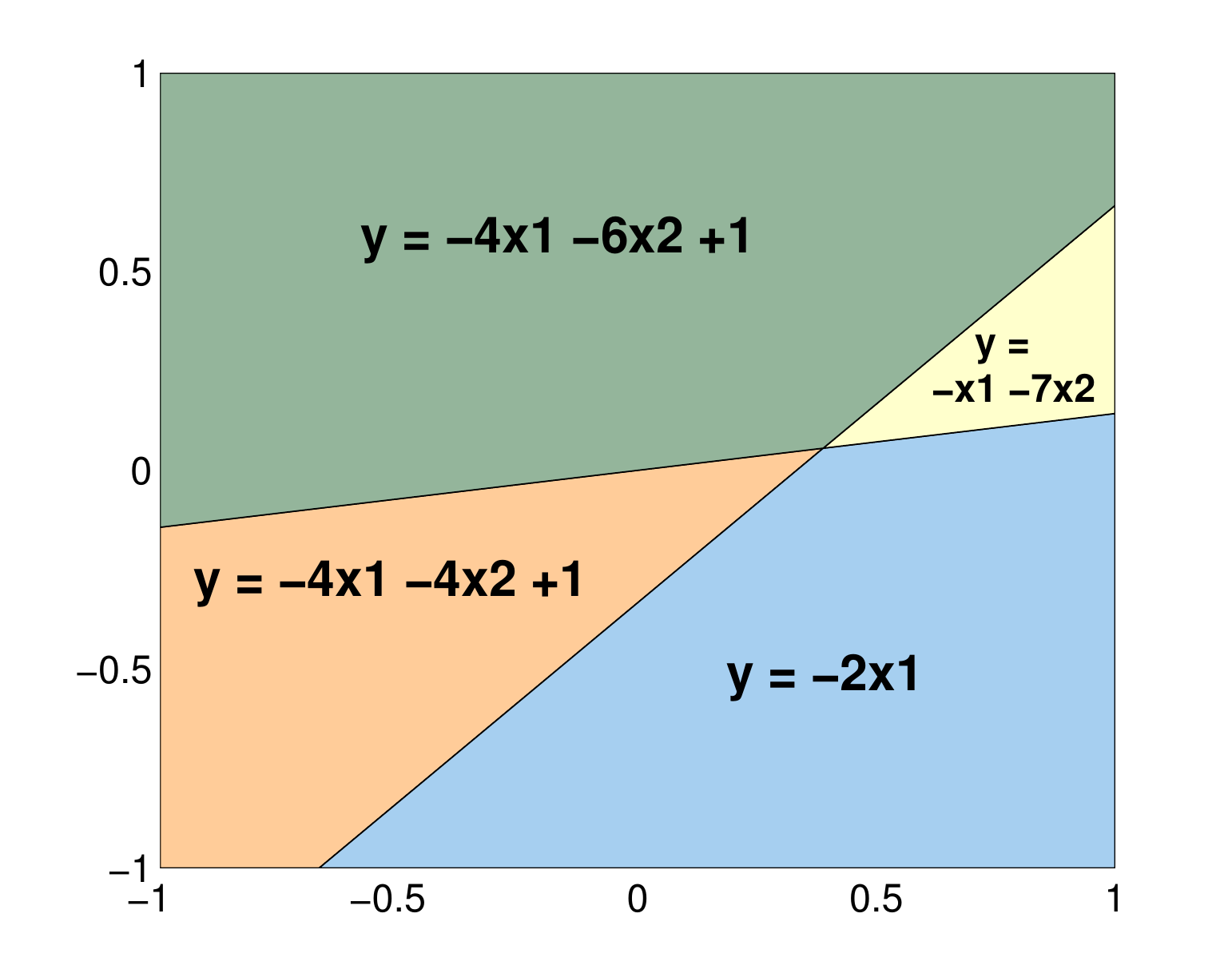}
\caption{\scriptsize{The projection of Example 2 onto $y=0$.}}
\label{figure:noncvx projection}
\end{minipage}
\;\,
\begin{minipage}{.3\textwidth}
\centering
\includegraphics[width=0.9\textwidth]{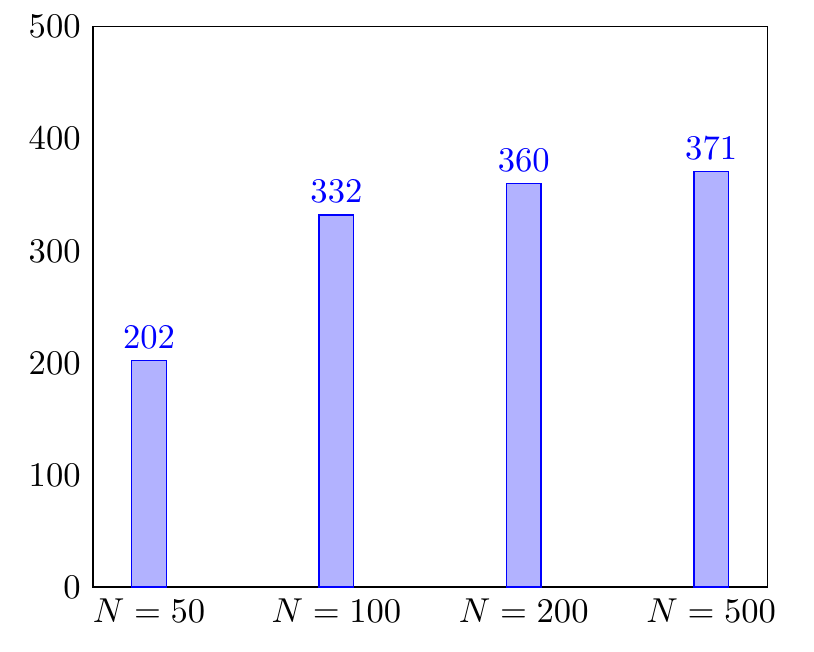}
\caption{\scriptsize{The number of initial points that lead to the samllest objective values.}}
\label{figure:noncvx landscape}
\end{minipage}
\end{center}
\end{figure}

\begin{figure}[H]
\begin{minipage}{.25\textwidth}
\centering
\includegraphics[width=\textwidth]{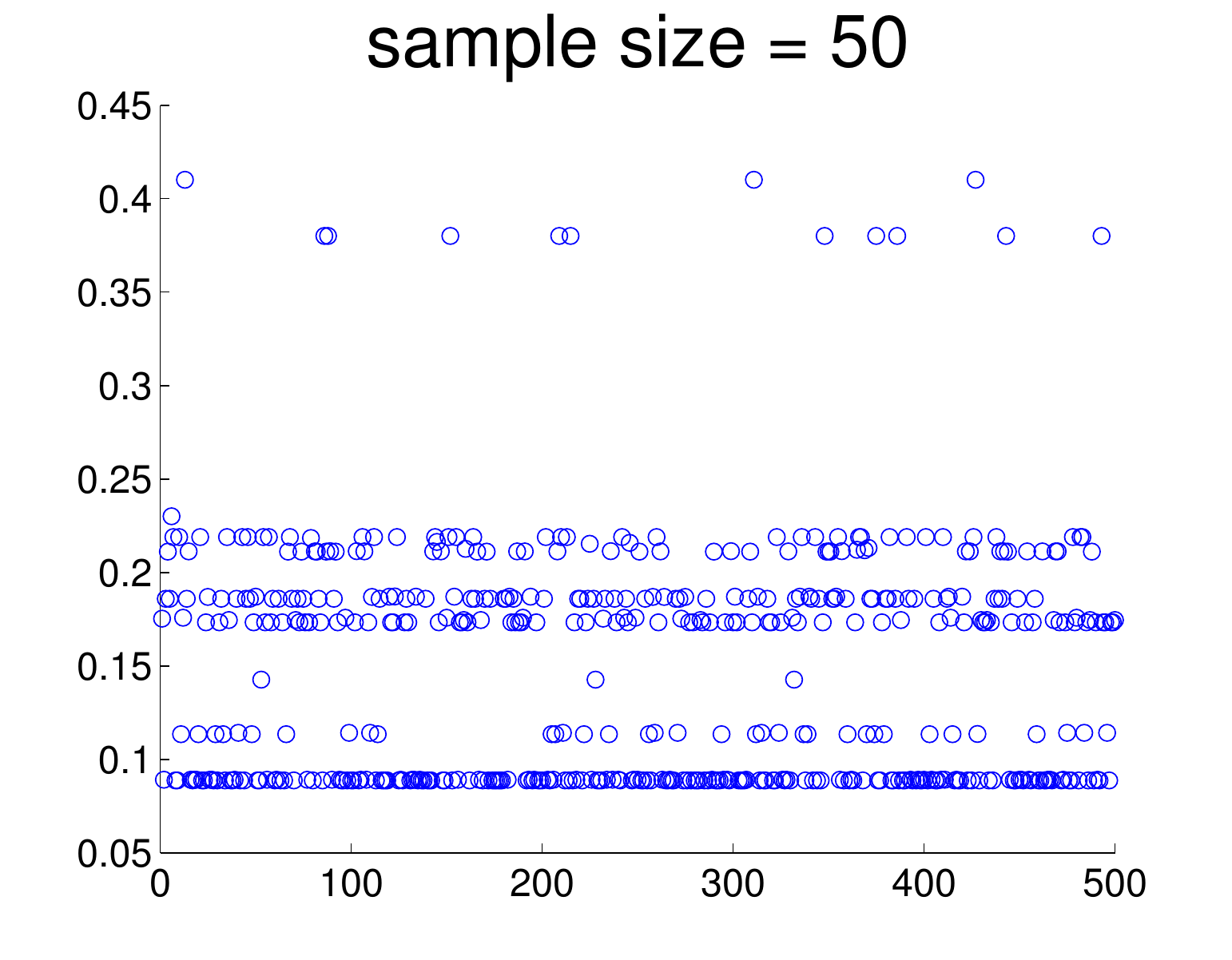}
\end{minipage}%
\begin{minipage}{.25\textwidth}
\centering
\includegraphics[width=\textwidth]{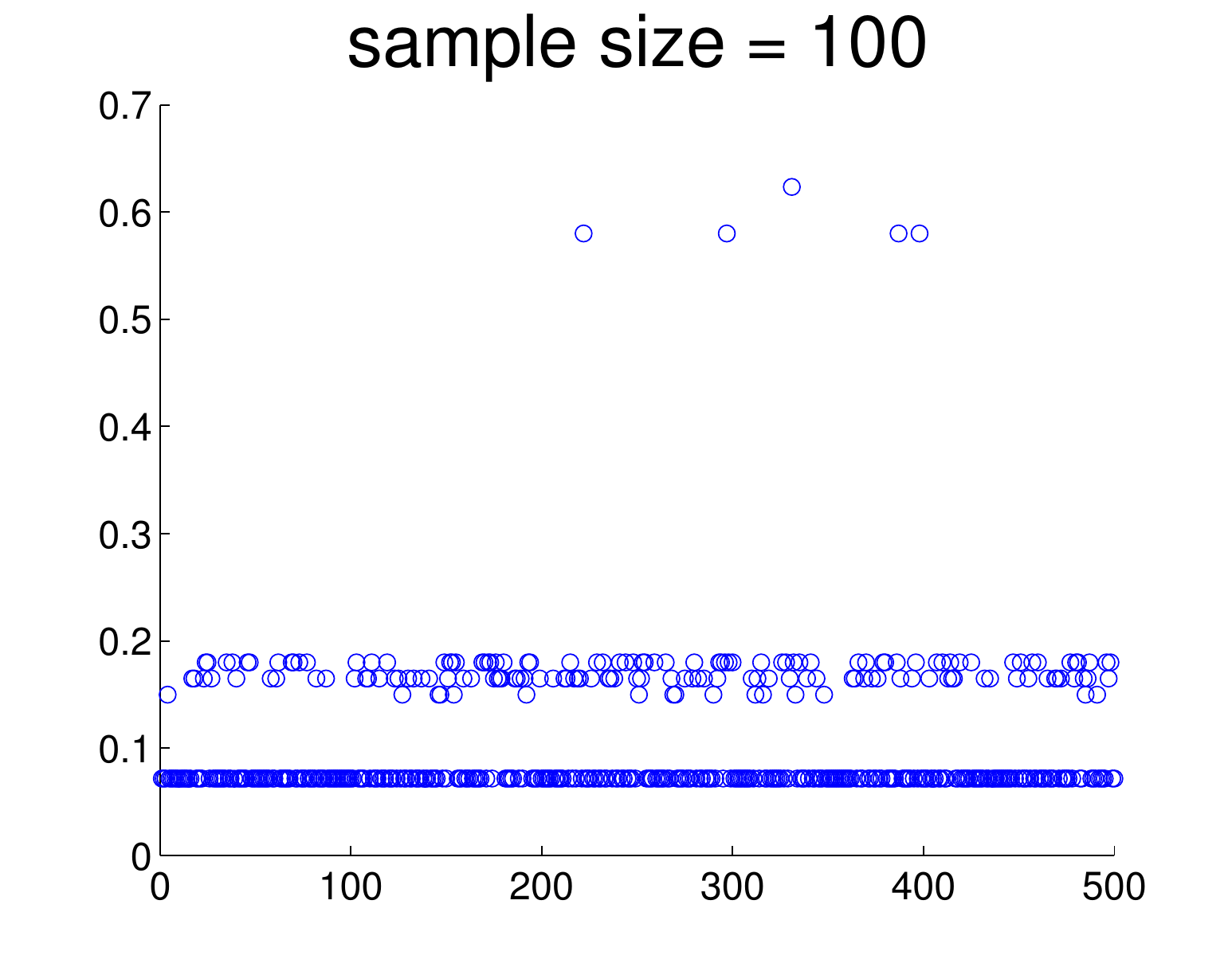}
\end{minipage}
\begin{minipage}{.25\textwidth}
\centering
\includegraphics[width=\textwidth]{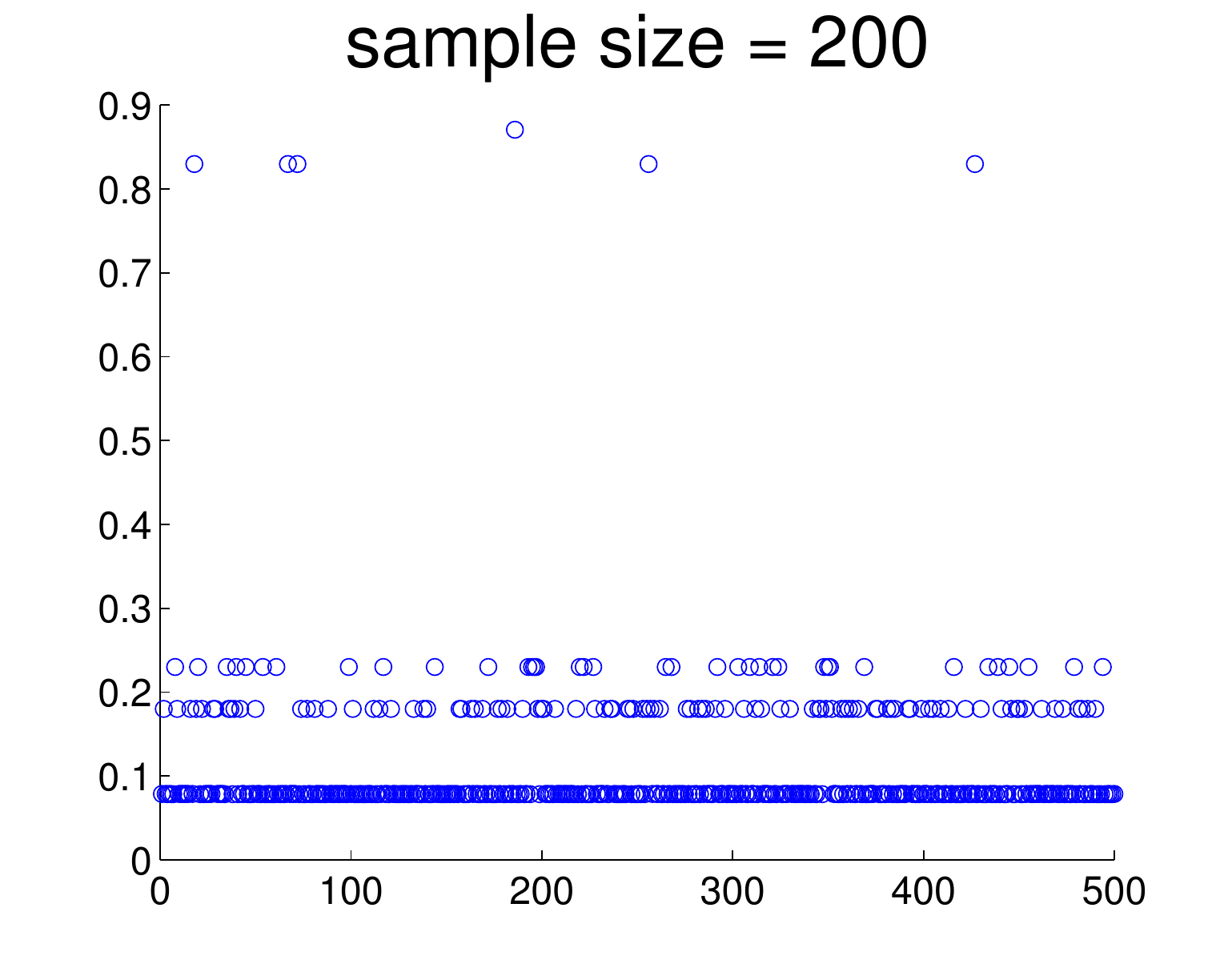}
\end{minipage}
\begin{minipage}{.25\textwidth}
\centering
\includegraphics[width=\textwidth]{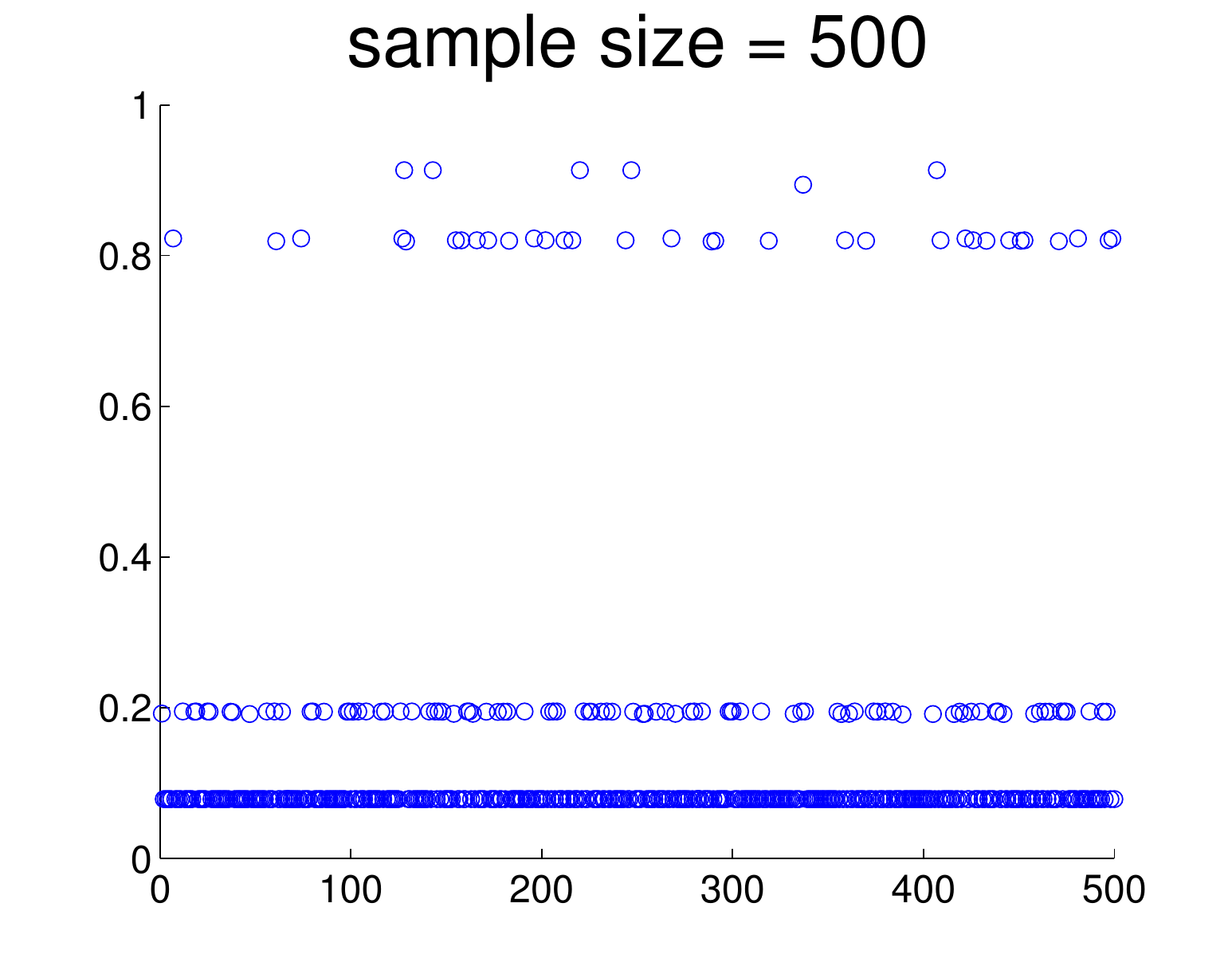}
\end{minipage}
\caption{\small{The objective values computed by the MM+SN algorithm of Example 2.}}
\label{figure:noncvx obj}
\end{figure}

\begin{figure}[H]
\centering
\begin{minipage}{.24\textwidth}
\centering
\includegraphics[width=1\textwidth]{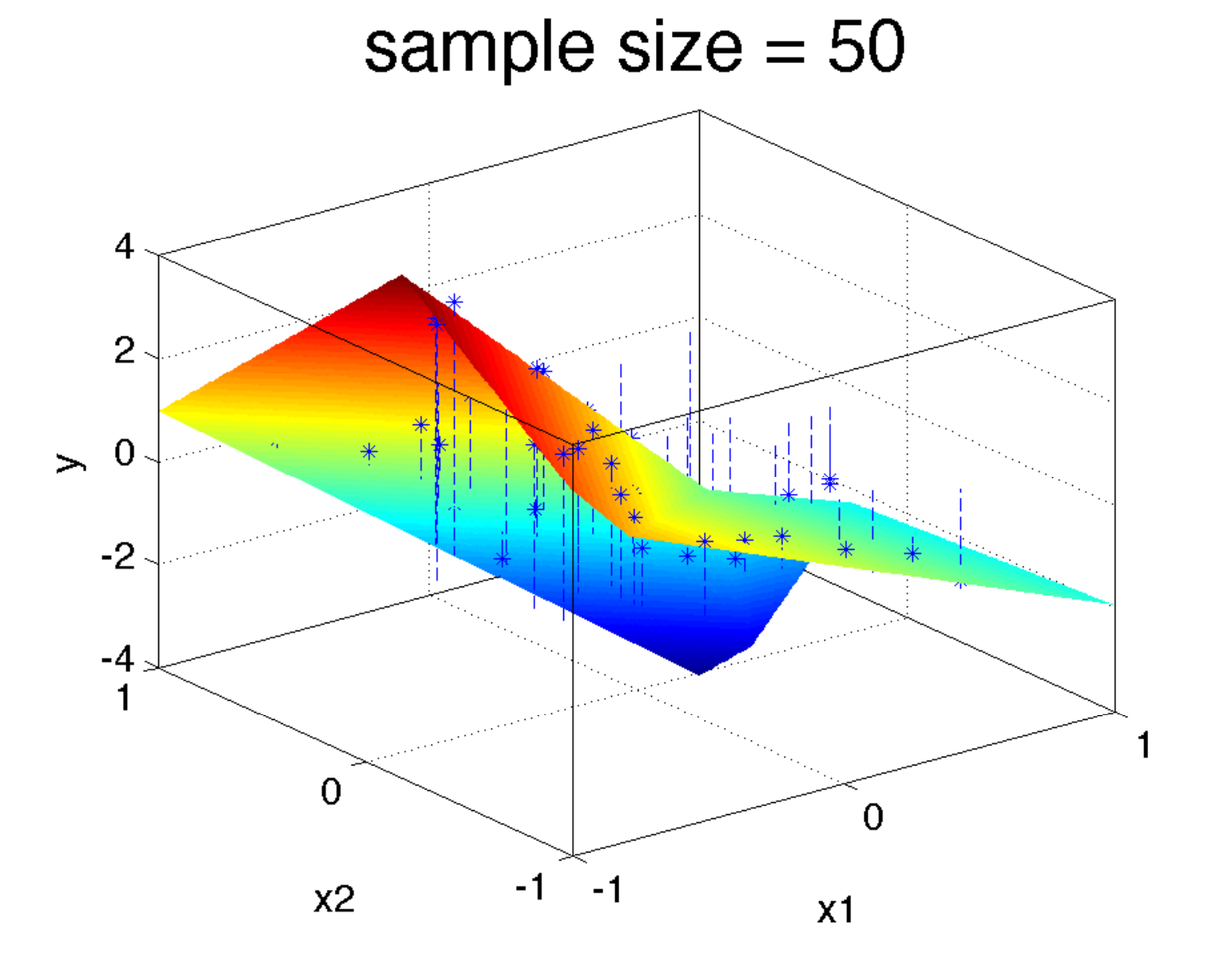}
\end{minipage}%
\begin{minipage}{.24\textwidth}
\centering
\includegraphics[width=1\textwidth]{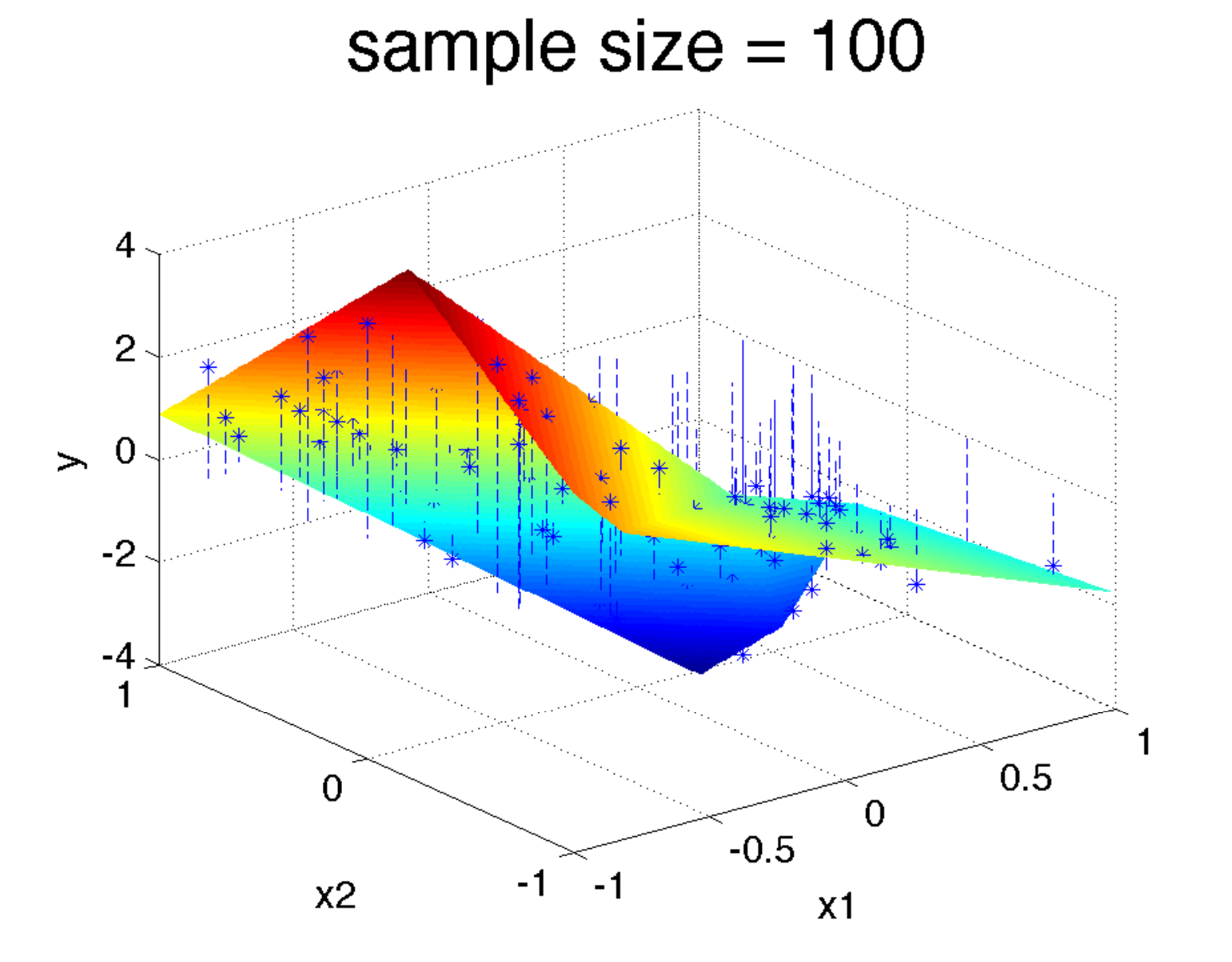}
\end{minipage}
\begin{minipage}{.24\textwidth}
\centering
\includegraphics[width=1\textwidth]{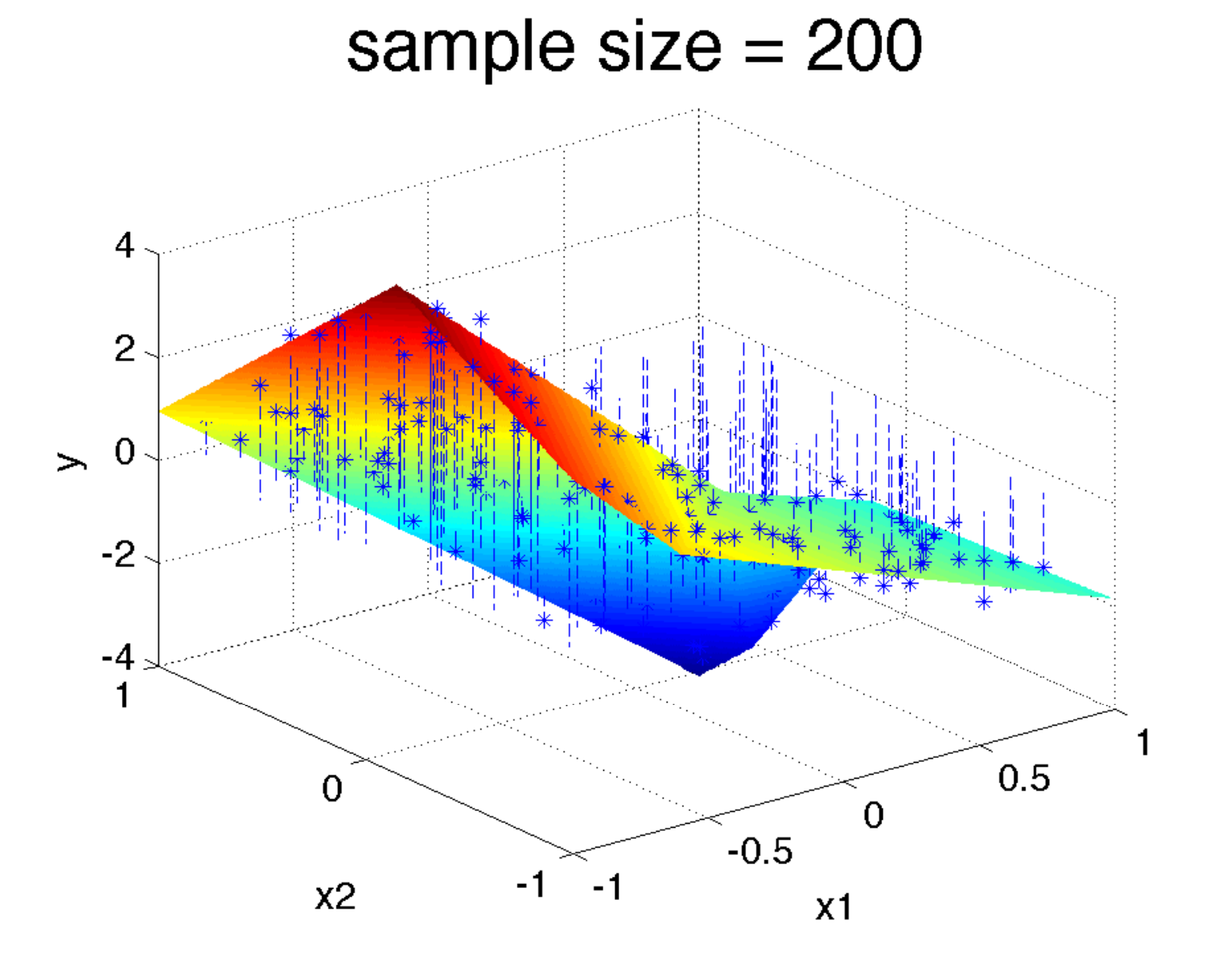}
\end{minipage}
\begin{minipage}{.24\textwidth}
\centering
\includegraphics[width=1\textwidth]{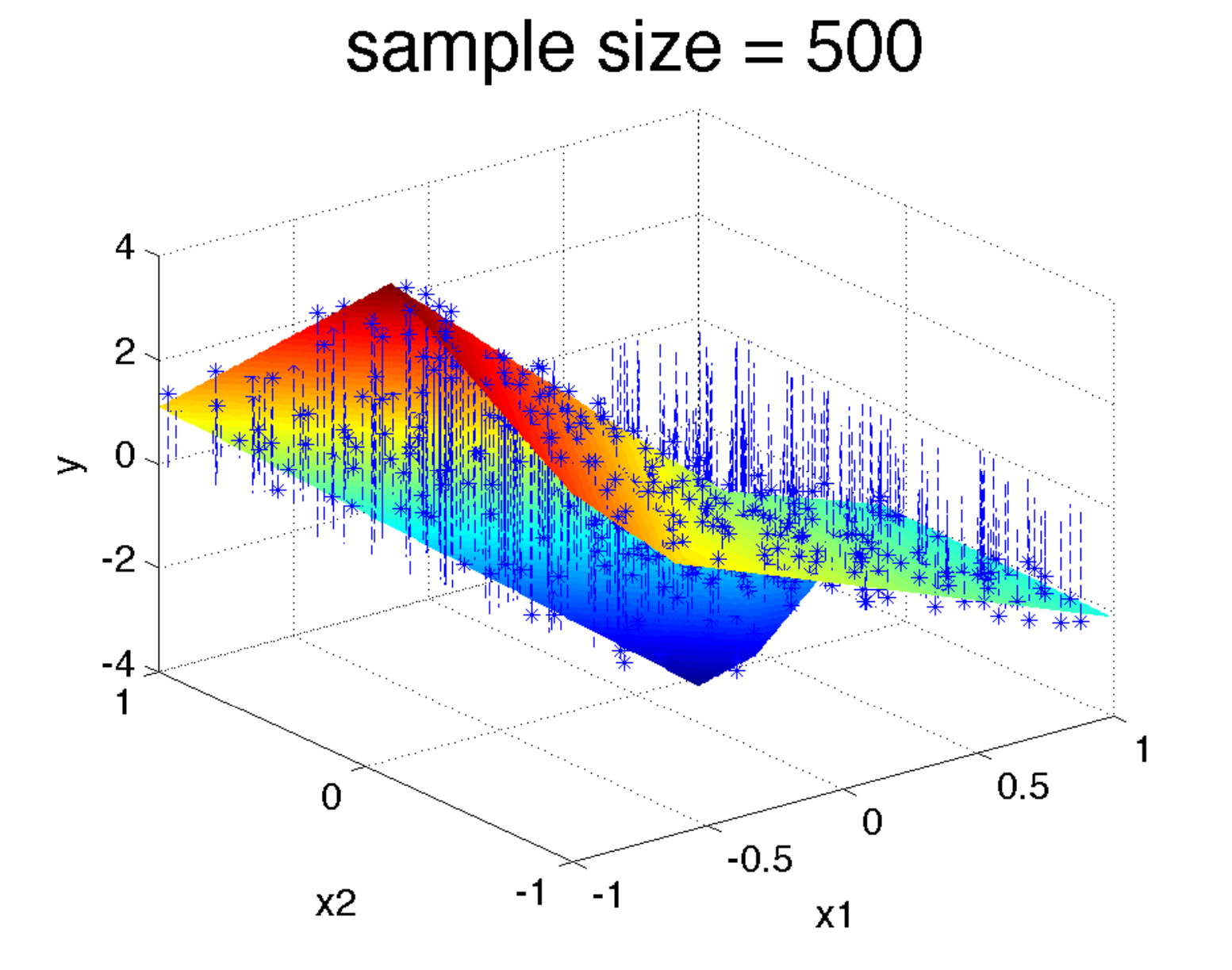}
\end{minipage}
\caption{\small{The solutions of Example 2 with the smallest objective values.}}
\label{figure:noncvx:result}
\end{figure}

\subsection{Real data}\label{section: pwa regression: real data}

We also conduct the experimental evaluation of the proposed  algorithm on the real datasets obtained from the UCI Repository of machine learning
databases \url{https://archive.ics.uci.edu/ml/index.php}.
We first list the problem characteristics and report the performance of the MM+SN method averaged over the pairs $(k_1,k_2)$
employed in the experiments, i.e.,\ satisfying $1\leq k_1, k_2\leq 4$ in Table \ref{table1}, where the two columns under ``iterations'' present, respectively, the number of iterations
of the MM algorithm and the total numbers of iterations of the SN algorithm.
For the first four problems where the sample sizes $N$ are much larger than the dimensions $d$,  no sparse regularization term is added to the model.
For the other two problems where $d$ is relatively large, we take the SCAD function as the sparsity inducing penalty function.
The sparse penalty parameter is estimated by  5-fold cross-validation.
One can see the efficiency of the proposed MM+SN method, where it usually takes no more than $10$
SN iterations for each MM subproblem.  As an illustration,
Figures~\ref{figure: MM speed} and \ref{figure: Newton speed} exhibit the performance of the MM algorithm (objective values vs iteration number) and the
SN algorithm ($\log_{10}\|\nabla \xi(\lambda^k, \mu^k;\wh{z}^{\,\nu})\|$ vs iteration number) for solving the ``auto MPG'' problem with $(k_1,k_2) = (2,2)$.

\begin{table}[H]
\centering
\begin{footnotesize}
\begin{tabular}{|c | c c| c  c| c |}
\hline
	
\multirow{2}{*}{ problem name } & \multirow{2}{*}{$N$} & \multirow{2}{*}{$d$} & \multicolumn{2}{|c|}{iterations} & \multirow{2}{*}{time} \\
\cline{4-5}
 & & &  MM & SN & \\\hline

banknote authentication
     &1372  &  4     & 9     &   72   & 3.5\,s\\[2pt]
 \hline
concrete compressive  strength
     &1030  &  8    & 23    &   197  & 10.8\,s\\[2pt]
 \hline
auto MPG
     &392  &  7     & 25    &    185  & 3.9\,s \\[2pt]
 \hline
 airfoil self-noise & 1503  &  5     & 18    &  156    & 12.3\,s \\[2pt]
 \hline
  Libras Movement  & 360 &  91     & 22    &    187  & 14.7\,s \\[2pt]
 \hline
 Communities and Crime & 2215 &  147    & 31    &    251  & 43.7\,s \\[2pt]
 \hline
\end{tabular}
\end{footnotesize}
 \caption{\small The performance of the MM+SN algorithm} \label{table1}
 \end{table}

\begin{figure}[H]
\begin{center}
\begin{minipage}{.47\textwidth}
\centering
\includegraphics[width=0.6\textwidth]{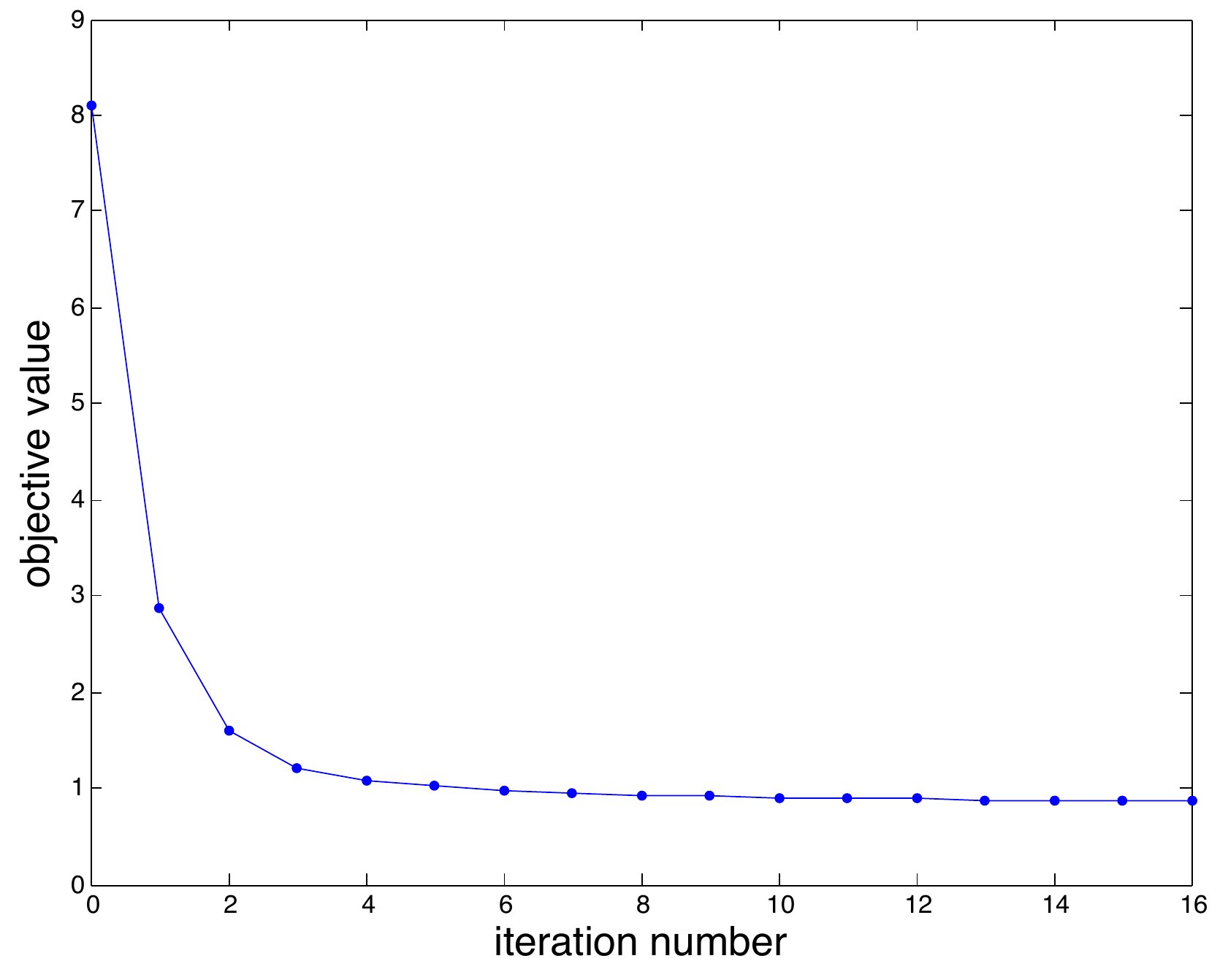}
\caption{\scriptsize{The performance of the MM algorithm.}}
\label{figure: MM speed}
\end{minipage}
\epc
\begin{minipage}{.45\textwidth}
\centering
\includegraphics[width=0.6\textwidth]{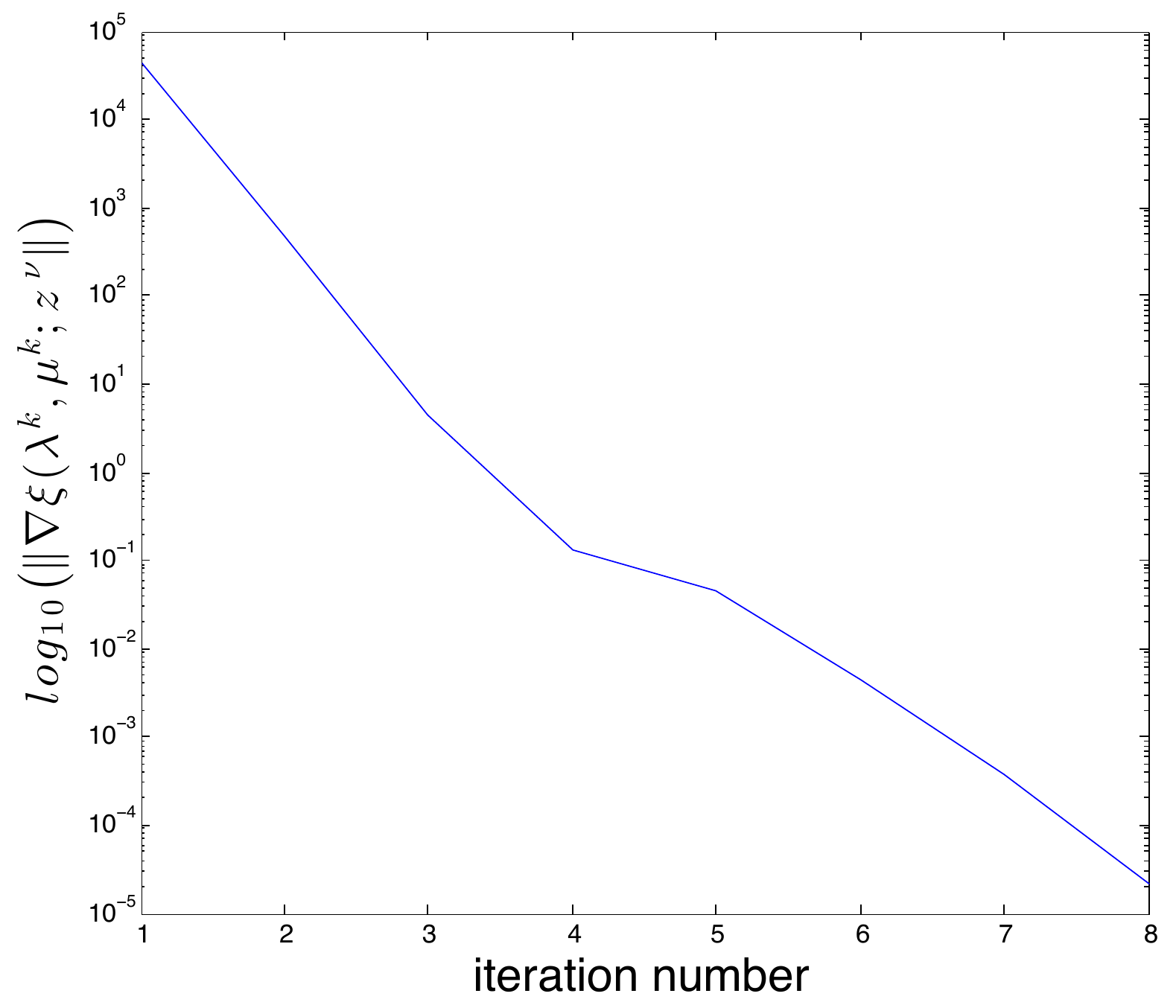}
\caption{\scriptsize{The performance of the SN algorithm.}}
\label{figure: Newton speed}
\end{minipage}
\end{center}
\end{figure}

To demonstrate the advantage of using the piecewise affine model over the classical linear regression model,
we compare the solution of the least-squares piecewise affine regression $\theta^{\,\rm PA}$ with the ordinary least-squares solution
$\theta^{\,\rm LS} \in  \displaystyle\operatornamewithlimits{\mbox{argmin}}_{\theta} \displaystyle\frac{1}{2N}\sum_{s=1}^N\|\,y^{s} - (x^{s},1)^T\theta\,\|^2$
in the following way. Let the full dataset of each instance obtained  from the UCI Repository be $T$, which is further divided  into a training
set $T^{\, \rm f}_{\rm training}$ and a test set $T^{\, \rm f}_{\rm test}$ for $f = 1, \ldots, 5$ by 5-fold cross-validation.
For each fold, we find the estimator $\theta^{\,{\rm PA},\rm f}$  based on the training set $T^{\, \rm f}_{\rm training}$  by choosing the solution
with the smallest objective value over $20$ runs of the MM+SN  algorithm (with each run corresponding to one initial point).
Such a $\theta^{\,\rm PA,f}$
is likely to be an (un-proven) global minimizer of the nonconvex least-squares piecewise affine regression program.
We compute the prediction errors of the two estimators respectively by
{\small \[{\rm E}_{{\rm PA}}\triangleq \displaystyle\sum_{\rm f = 1}^5\frac{1}{|T^{\rm f}_{\rm test}|}\sum_{s\in T_{\rm test}^{\rm f}} \left(y_{s} -
\psi(x^s;\theta^{\rm PA,f})\right)^2,
{\rm E}_{{\rm LS}}\triangleq\displaystyle\sum_{\rm f = 1}^5 \frac{1}{|T^{ \rm f}_{\rm test}|}\sum_{s\in T^{\rm f}_{\rm test}}
\left(y_{s} - (x^{s},1)^T\theta^{\,{\rm LS,f}}\right)^2.
\]}
We take $100$ simulations 
and report the average ratio
${\rm E}_{\rm PA}/{\rm E}_{\rm LS}$ for different choices of $(k_1, k_2)$ in Table \ref{table:ratio PA}.
It can be seen that all piecewise affine results are better than the linear regression results ($k_1 = k_2 = 1$);
moreover, with proper choices of $(k_1, k_2)$, the prediction error given by the piecewise affine regression
model is significantly smaller than that given by the linear regression model as highlighted by the best entry in each table.

As a final remark, we have observed in our simulation studies that even when the true model is linear, the continuous piecewise affine regression fits (for different choices of
$k_1$ and $k_2$) are very close to that of linear regression, i.e., the ratios of the prediction errors (in 5-fold cross validation as given in Table~\ref{table:ratio PA}) are all very close to 1. This provides further evidence that the continuous piecewise affine regression model could be a worthwhile and useful extension of linear regression. 

\begin{table}[h]
\centering
\begin{footnotesize}
\begin{tabular}{| c | c | c | c | c | c | c | c | c | c | c |}
\hline
\multicolumn{5}{|c|}{} && \multicolumn{5}{c|}{}\\[-0.8em]
\multicolumn{5}{|c|}{ banknote authentication} & &\multicolumn{5}{c|}{ concrete compressive strength}\\[0.2em]
 \cline{1-5}\cline{7-11}
 \backslashbox{$\quad k_2$}{$k_1$}  &$1$ & $2$ & $3$ & $4$  & &
 \backslashbox{$\quad k_2$}{$k_1$}  &$1$ & $2$ & $3$ & $4$
  \\
 \cline{1-5}\cline{7-11}
 $1$ & $1.00$ &$0.74$  & $0.68$ & $0.67$ & & $1$ & $1.00$ & $0.74$ & $0.47$ & {\color{blue} $0.38$} \\
 \cline{1-5}\cline{7-11}
 $2$ & $0.84$ &  $0.73$ & $0.65$ & {\color{blue}$0.63$}  & & $2$ & $0.85$ &  $0.46$ & $0.39$ & $0.39$  \\
 \cline{1-5}\cline{7-11}
 $3$ & $0.81$ & $0.72$ & $0.74$ & $0.74$ & &$3$ & $0.82$ & $0.65$ & $0.62$ & $0.61$ \\
 \cline{1-5}\cline{7-11}
 $4$ & $0.85$ & $0.71$ & $0.69$ & $0.74$ & &$4$ & $0.77$  & $0.54$ & $0.57$ & $0.60$  \\
 \cline{1-5}\cline{7-11}
 \multicolumn{5}{|c|}{} && \multicolumn{5}{c|}{}\\[-0.8em]
 \multicolumn{5}{|c|}{auto MPG} & &\multicolumn{5}{c|}{airfoil self-noise}\\[0.2em]
 \cline{1-5}\cline{7-11}
\backslashbox{$\quad k_2$}{$k_1$} & $1$ & $2$ & $3$ & $4$  && \backslashbox{$\quad k_2$}{$k_1$} & $1$ & $2$ & $3$ & $4$  \\
 \cline{1-5}\cline{7-11}
 $1$ & $1.00$ &  $0.77$   & {\color{blue} $0.72$} & $0.77$  && $1$ & $1.00$  & $0.83$ &$0.77$  & $0.74$\\
 \cline{1-5}\cline{7-11}
 $2$ & $0.97$  & $0.74$ & $0.79$ & $0.73$  &&  $2$ & $0.76$ & $0.72$ & $0.72$ & $0.59$  \\
 \cline{1-5}\cline{7-11}
 $3$ & $0.83$ &  $0.78$ & $0.76$ & $0.76$  && $3$ & $0.68$ & $0.68$ & $0.67$ & $0.57$  \\
 \cline{1-5}\cline{7-11}
 $4$ &  $0.85$ & $0.84$ & $0.73$ & $0.73$  &&  $4$ &$0.62$ & $0.62$ & $0.62$ & {\color{blue} $0.425$} \\
\cline{1-5}\cline{7-11}
 \multicolumn{5}{|c|}{} && \multicolumn{5}{c|}{}\\[-0.8em]
  \multicolumn{5}{|c|}{Libras Movement} & &\multicolumn{5}{c|}{Communities and Crime}\\[0.2em]
  \cline{1-5}\cline{7-11}
\backslashbox{$\quad k_2$}{$k_1$} & $1$ & $2$ & $3$ & $4$ && \backslashbox{$\quad k_2$}{$k_1$} & $1$ & $2$ & $3$ & $4$ \\
\cline{1-5}\cline{7-11}
 $1$ & $1.00$ & $0.93$ & $0.82$ & $0.80$ &&  $1$ & $1.00$ & $0.88$ & $0.85$ & $0.83$ \\
\cline{1-5}\cline{7-11}
 $2$ &  $0.83$ & $0.86$ & $0.77$ & $0.77$  &&  $2$ &  $0.862$ & $0.82$ & $0.84$ & $0.84$ \\
\cline{1-5}\cline{7-11}
 $3$ & $0.78$ & $0.72$ & {\color{blue} $0.72$} & $0.76$ &&  $3$ & $0.82$ & {\color{blue} $0.83$} & $0.83$ & $0.83$\\
\cline{1-5}\cline{7-11}
  $4$ & $0.90$ & $0.73$ & $0.78$ & $0.80$ && $4$ & $0.87$ & $0.84$ & $0.86$ & $0.87$\\
 \hline
 \end{tabular}
  \end{footnotesize}
\caption{\small Ratio of the prediction error for different choices of $(k_1,k_2)$.}
\label{table:ratio PA}
 \end{table}

In conclusion, we have demonstrated the practical worthiness of the LS piecewise affine regression model
formulated as a nonconvex nondifferentiable program and solved by the nonmonotone MM+SN method.  Further
applications of this method to other statistical estimation problems will be reported separately.


\end{document}